\documentclass[11pt]{article}

\usepackage{todonotes}
\usepackage{amsmath, amssymb, amsthm, bm}
\usepackage{authblk}
\usepackage{mathrsfs}
\usepackage{geometry}
\usepackage{xcolor}
\usepackage{graphicx}
\usepackage{cite}
\usepackage{enumitem}
\usepackage{titlesec}
\usepackage[colorlinks,
            linkcolor=red,
            anchorcolor=blue,
            citecolor=blue
            ]{hyperref}
\usepackage{array}
\usepackage{subfigure}
\usepackage{booktabs}
\usepackage{multirow}
\usepackage[title]{appendix}
\usepackage[linesnumbered,ruled,vlined]{algorithm2e}
\SetKwInput{KwInput}{Input}                
\SetKwInput{KwReturn}{Return}

\SetCommentSty{mycommfont}
\usepackage{float}
\usepackage{fullpage}

\usepackage{comment}
\usepackage[T1]{fontenc}

\newcolumntype{A}{>{\centering\arraybackslash}m{0.12\columnwidth}}
\newcolumntype{B}{>{\centering\arraybackslash}m{0.4\columnwidth}}

\newcommand{\djs}{D_{\rm JS}}
\newcommand{\tp}{\tilde{p}}
\def\pto{\overset{p}{\to}}
\def\dto{\overset{d}{\to}}

\def\id{\mathbf{I}}
\def\bbE{\mathbb{E}}
\def\bbP{\mathbb{P}}
\def\bbR{\mathbb{R}}
\def\cD{\mathcal{D}}

\def\cG{\mathcal{G}}
\def\cI{\mathcal{I}}
\def\cN{\mathcal{N}}
\def\cP{\mathcal{P}}
\def\cS{\mathcal{S}}
\def\cV{\mathcal{V}}
\def\cX{\mathcal{X}}

\def\cZ{\mathcal{Z}}
\def\tp{\tilde{p}}

\def\gage{\hat\theta_{\rm AGE}}
\def\gfgan{\hat\theta_{f\text{-}\rm GAN}}
\def\dage{\hat\psi_{\rm AGE}}
\def\dfgan{\hat\psi_{f\text{-}\rm GAN}}

\DeclareMathOperator*{\argmin}{argmin}
\DeclareMathOperator*{\argmax}{argmax}

\newcommand{\tz}[1]{\todo[inline,color=yellow]{#1}}
\definecolor{columbiablue}{rgb}{0.61, 0.87, 1.0}
\newcommand{\xw}[1]{\todo[inline,color=columbiablue]{#1}}

\title{Asymptotic Statistical Analysis of $f$-divergence GAN}
\author[$\star$]{Xinwei Shen}
\author[$\ddag$]{Kani Chen}
\author[$\dag$]{Tong Zhang}
\affil[$\star$]{Seminar for Statistics, ETH Z\"urich}
\affil[$\ddag$]{Department of Mathematics, HKUST}
\affil[$\dag$]{Department of Computer Science and Mathematics,  HKUST}
\date{}

\usepackage{fullpage}

\theoremstyle{plain}

\newtheorem{theorem}{Theorem}
\newtheorem{lemma}[theorem]{Lemma}
\newtheorem{corollary}[theorem]{Corollary}
\newtheorem{assumption}{Assumption}
\newtheorem{proposition}[theorem]{Proposition}

\theoremstyle{remark}
\newtheorem{definition}[theorem]{Definition}
\newtheorem*{example}{Example}

\newtheorem{remark}{Remark}

\newcommand{\revise}[1]{{\textcolor{black}{#1}}}

\begin{document}
\maketitle

\begin{abstract}
\noindent
Generative Adversarial Networks (GANs) have achieved great success in data generation. However, its statistical properties are not fully understood.  
In this paper, we consider the statistical behavior of the general $f$-divergence formulation of GAN, which includes the Kullback--Leibler divergence that is closely related to the maximum likelihood principle. 
We show that for parametric generative models that are correctly specified, all $f$-divergence GANs with the same discriminator classes are asymptotically equivalent under suitable regularity conditions. Moreover, with an appropriately chosen local discriminator, they become equivalent to the maximum likelihood estimate asymptotically.   
For generative models that are misspecified, GANs with different $f$-divergences {converge to different estimators}, and thus cannot be directly compared. However, it is shown that for some commonly used $f$-divergences, the original $f$-GAN is not optimal in that one can achieve a smaller asymptotic variance when the discriminator training in the original $f$-GAN formulation is replaced by logistic regression.
The resulting estimation method is referred to as Adversarial Gradient Estimation (AGE). 
Empirical studies are provided to support the theory and to demonstrate the advantage of AGE over the original $f$-GANs under model misspecification. 
\end{abstract}

\section{Introduction}\label{sec:intro}

Generative Adversarial Networks (GANs)~\cite{goodfellow2014gan} have received considerable interest in machine learning. 
It has many practical applications, such as generating photorealistic images~\cite{biggan,stylegan}, videos~\cite{tulyakov2018mocogan}, text~\cite{yu2017seqgan}, and music~\cite{dong2018musegan}. 
Statistically, GAN can be used to sample from an unknown distribution $p_*$. It can be applied to complex densities even when the classical parametric distribution families or nonparametric density estimation approaches, such as the kernel density estimation fail.

We assume that $x_1,\dots,x_n$ are independent and identically distributed (i.i.d.) copies of a random variable $X\sim p_*$ on $\cX\subseteq\bbR^d$.
In the GAN framework, there is a random variable $Z$ with a known distribution $p_z$ (e.g., a Gaussian) on $\mathcal{Z}\subseteq\mathbb{R}^k$. We aim at learning a transformation $G$ of the variable $Z$, known as a {\em generator}, so that the distribution of the generated variable $G(Z)$ becomes close to $p_*$. The generator is parametrized using parameter $\theta\in\Theta$, usually represented by a deep neural network. The distribution of the generated data $G_\theta(Z)$ is denoted by $p_\theta$. 
Let $z_1,\dots,z_m$ be an i.i.d. sample from $p_z$ whose sample size $m$ is usually much larger than that of the real data $n$. The transformations $G_\theta(z_i)$ ($i=1,\ldots,m$) are samples from $p_\theta$. 

To learn the generator $G_\theta$, the original formulation of GAN in~\cite{goodfellow2014gan} introduced a discriminator $D\in\cD$ to solve the following minimax optimization problem
\begin{equation}\label{eq:gan}
	\inf_{\theta\in\Theta}\sup_{D\in\cD}\left[\frac{1}{n}\sum_{i=1}^{n}\ln(1+e^{-D(x_i)})+\frac{1}{m}\sum_{i=1}^{m}\ln(1+e^{D(G_\theta(z_i))})\right],
\end{equation}
where $\ln$ is the natural logarithm. Under appropriate conditions, the above problem is shown to be asymptotically equivalent to minimizing the Jensen--Shannon (JS) divergence between the true and generated distributions~\cite{goodfellow2014gan,biau2020},
\begin{equation}\label{eq:obj_js}
	\inf_{\theta\in\Theta}\djs(p_*,p_\theta),
\end{equation}
where $\djs(p,q)$ denotes the JS divergence of two distributions which will be formally defined later.

In this paper, we consider the more general $f$-divergences as the objective function, which are a broader class of divergences and include the Kullback--Leibler (KL) divergence that is closely related to the maximum likelihood principle in statistics. 
Given two probability measures $\mu$ and $\nu$ with absolutely continuous density functions $p$ and $q$ with respect to the Lebesgue measure on $\cX$, the $f$-divergence is defined by
\begin{equation}\label{eq:f_div}
	D_f(p,q)=\int_\cX p(x) f(q(x)/p(x))dx,
\end{equation}
where $f:\mathbb{R}_+\to\mathbb{R}$ is a convex, lower-semicontinuous function satisfying $f(1) = 0$. Throughout this paper, we focus on the case where $f$ is twice  continuously differentiable and strongly convex so that the second order derivative of $f$, denoted by $f''$, is always positive, which includes the commonly used divergences as listed in Table~\ref{tab:f-div}. 
\begin{table}[b]
\centering
\caption{List of $f$-divergences: KL divergence, reverse KL divergence, JS divergence$\times2$, and squared Hellinger distance.}\label{tab:f-div}
\begin{tabular}{ccc}
\toprule
Name & $f(r)$ & $rf''(r)$  \\\midrule
KL & $-\ln r$ & $1/r$ \\
RevKL & $r\ln r$ & $1$ \\
2JS & $-(r+1)\ln\frac{1+r}{2}+r\ln r$ & $\frac{1}{1+r}$ \\
$H^2$ & $(\sqrt{r}-1)^2$ & $\frac{1}{2\sqrt{r}}$ \\
\bottomrule
\end{tabular}
\end{table}

Analogous to (\ref{eq:obj_js}), the general $f$-divergence formulation of generative modeling is 
\begin{equation}\label{eq:obj_f}
	\inf_{\theta\in\Theta}D_f(p_*,p_\theta).
\end{equation}
Similar to the original GAN (\ref{eq:gan}), a direct minimax formulation of $f$-divergence GAN can be obtained, leading to $f$-GAN~\cite{fgan}.
In the $f$-GAN approach, different $f$-divergences lead to different minimax objective functions analogous to (\ref{eq:gan}). 
We show in Section~\ref{sec:f_equiv} that if the model is correctly specified, various $f$-GANs are asymptotically equivalent.
Nevertheless, in most applications of GANs, the true distribution is so complex that there hardly exists $\theta\in\Theta$ such that $p_*=p_\theta$, leading to model misspecification. In such cases, different divergences lead to different generators, i.e., the solutions of (\ref{eq:obj_f}). 
Therefore it is worthwhile to consider general $f$-divergences and study the statistical properties of generative modeling methods under various divergences. 

However, limited by the inherent minimax formulation, $f$-GAN adopts different discriminator losses for different divergences, as reviewed in Section~\ref{sec:fgan}. 
We show in Section~\ref{sec:d_eff} that this approach does not result in a statistically efficient discriminator estimation. 
Motivated by this finding, we propose to replace the training of the discriminator in $f$-GAN by logistic regression, which can leverage the statistical efficiency of the maximum likelihood estimation. 
This leads to a new method AGE (Adversarial Gradient Estimation) for the general $f$-divergence formulation (\ref{eq:obj_f}). The method can be regarded as an approximate gradient descent algorithm with the gradient being estimated using a discriminator learned by logistic regression. We show in Section~\ref{sec:misspecify} that our AGE method obtains asymptotically more efficient estimators for both the discriminator and the generator than the original $f$-GAN under general cases with model misspecification.
Therefore for general $f$-divergences, AGE can be regarded as a class of principled statistical methods with good asymptotic properties. 


Despite the empirical success of GANs, there were limited theoretical studies on their statistical properties. Most existing works~\cite{arora2017generalization,liang2018well,bai2018approximability,zhang2018on,chen2020statistical} focused on the Integral Probability Metric (IPM) framework which includes WGAN~\cite{wgan}, a celebrated extension of GAN to the Wasserstein distance. However, they do not apply to $f$-GANs or the original GAN. Moreover, they studied the generalization properties of GAN or the weak convergence of the learned distribution under certain metrics, which is complementary to our focus on the asymptotic behavior of the parameter estimation. 
\revise{Another line of work \cite{liushuang2017,dualing2017} studied GANs from the optimization perspective, which is related to one part of our work involving a linear discriminator class, as discussed in Section~\ref{sec:localgan}.}

The current work is motivated by \cite{biau2020}, which studied the asymptotic properties of the original GAN formulation (\ref{eq:gan}) with $m=n$, while did not discuss the statistical consequences of the analysis. 
Our paper investigates additional implications of the asymptotic statistical analysis which were not considered by \cite{biau2020}, thus contributing to a better statistical understanding of GANs. 
First, we extend the JS divergence to general $f$-divergences which includes the commonly used KL divergence, and study how various divergences behave under the GAN framework. It is shown that with correctly specified generative models, all methods are asymptotically equivalent. Second, for misspecified models, our statistical analysis leads to an improved GAN method which has a smaller asymptotic variance than that of the commonly used $f$-GAN method. Third, we allow a much larger sample size $m$ for the generated data than \cite{biau2020}. It is shown that when the ratio $m/n \to \infty$, and with an appropriate local discriminator class, $f$-divergence GANs are asymptotically equivalent to the maximum likelihood estimate under regularity conditions. This demonstrates the statistical efficiency of GANs.  

\revise{
Furthermore, we would like to point out that the mathematical structures of GAN and noise contrastive estimation (NCE) share some similarities. As a method for unnormalized density estimation, the basic idea of NCE is to perform nonlinear logistic regression to discriminate between the observed data and some artificially generated noise. \cite{Gutmann2012NoiseContrastiveEO} derived the asymptotic distribution of NCE and also pointed out that NCE asymptotically attains the Cram\'er-Rao lower bound as the number of noise samples goes to infinity. Although there exists some interesting connections between GAN and NCE to be noted, our work is significantly distinguished from literature on NCE in terms of the problem setup, formulation, algorithm, and analysis. Notably, the theoretical results of NCE cannot be used to infer any results in the current work, since the later involves either more general (regarding discriminators) or irrelevant problems (regarding generators). 
In Appendix~\ref{app:nce}, we provide a detailed discussion on the connection and differences between our work and previous studies on NCE \cite{Gutmann2012NoiseContrastiveEO, Pihlaja2010AFO}. 
}

The remainder of the paper is organized as follows. In Section~\ref{sec:method}, we briefly introduce the original $f$-GAN and present our proposed AGE method. In Section~\ref{sec:theory}, we study the asymptotic behavior of AGE and $f$-GAN, based on which we develop more detailed analysis and consequences in the following two parts. In Section~\ref{sec:misspecify}, we provide a comprehensive discussion on the asymptotic relative efficiency of various approaches under model misspecification. Section~\ref{sec:well_specify} is devoted to an insightful analysis on the relationship between GAN and classical MLE, as well as various $f$-divergence GANs under correct model specification. 
Section~\ref{sec:experiment} presents the simulation results that support our theory. Section~\ref{sec:conclu} concludes. 

\medskip
\noindent
\textbf{Notation} \ \ 
Throughout the paper, all distributions are assumed to be absolutely continuous with respect to the Lebesgue measure unless stated otherwise. For random vectors $X,Y$, let $\mathrm{Cov}(X,Y)=\bbE[(X-\bbE X)(Y-\bbE Y)^\top]$ be their covariance matrix and $\mathrm{Var}(X)$ be the variance matrix of $X$. 
For a scalar function $h(x,y)$, let $\nabla_x h(x,y)$ denote its gradient with respect to $x$, which is a column vector; let $\nabla^2_x h(x,y)$ denote its Hessian matrix with respect to $x$. For a vector function $g(x,y)$, let $\nabla_x g(x,y)$ denote its Jacobian matrix with respect to $x$. 
Without ambiguity, $\nabla_x$ is denoted by $\nabla$ for simplicity.
Notation $\|x\|$ denotes the Euclidean norm for a vector $x$ and the Frobenius norm for a matrix $x$. For a vector $x$, $x^{\otimes2}$ stands for $xx^\top$.
For two deterministic sequences $a_n,b_n>0$, we say $a_n=o(b_n)$ if $\forall M>0,\exists n_0,\forall n>n_0:|a_n|<Mb_n$; $a_n=O(b_n)$ if $\exists M>0,\exists n_0,\forall n>n_0:|a_n|\leq M b_n$; $a_n=\mathbf\Theta(b_n)$ if $\exists M_1>0,M_2>0,\exists n_0,\forall n>n_0: M_1 b_n\leq a_n\leq M_2 b_n$.
For stochastic sequences, we use $o_p$ and $O_p$ to denote the counterparts of $o$ and $O$ in probability. 

We use the following notion of smoothness. 
\begin{definition}\label{def:smooth}
Consider a function $h(x):\bbR^{d}\to\bbR$. $h(x)$ is $\ell_0$-smooth with respect to $x$ if $h(x)$ is differentiable and its gradient is $\ell_0$-Lipschitz continuous, i.e., we have 
\begin{equation*}
	\|\nabla h(x)-\nabla h(x')\| \leq \ell_0\|x-x'\|,\quad \forall x,x'\in\bbR^{d}.
\end{equation*}
\end{definition}

\section{$f$-divergence GAN}\label{sec:method}
In this section, we start with a brief introduction of $f$-GAN~\cite{fgan} and then propose our new approach, both of which are GAN methods to solve the general $f$-divergence formulation (\ref{eq:obj_f}) of generative modeling. We also discuss the comparison between the two methods.

\subsection{$f$-GAN}\label{sec:fgan}
$f$-GAN was proposed in \cite{fgan} based on the variational characterization of $f$-divergences. For a convex function $f$, its conjugate dual function $f^*$ is defined as $f^*(v)=\sup_{u\in\mathrm{dom}_f}[uv-f(u)]$.
One can also represent $f$ as $f(u)=\sup_{v\in\mathrm{dom}_{f^*}}[vu-f^*(v)]$, which is then leveraged to obtain a lower bound on the $f$-divergence,
\begin{equation*}
	D_f(p_*,p_\theta)\geq \sup_{V\in\cV}\left(\bbE_{p_\theta}[V(X)]-\bbE_{p_*}[f^*(V(X))]\right),
\end{equation*}
where $\cV$ is an arbitrary class of variational functions $V:\cX\to\bbR$. Note that the above bound is tight for $V^*(x)=f'(p_\theta(x)/p_*(x))$, under mild conditions for $f$ \cite{nguyen2010estimating}. 
$f$-GAN then formulates the following minimax problem for learning $\theta$,
\begin{equation}\label{eq:obj_fgan}
	\inf_{\theta\in\Theta}\sup_{V\in\cV}\left(\bbE_{p_\theta}[V(X)]-\bbE_{p_*}[f^*(V(X))]\right). 
\end{equation}

To apply the objective (\ref{eq:obj_fgan}) for different $f$-divergences, $f$-GAN respects the domain $\mathrm{dom}_{f^*}$ of the conjugate $f^*$ by representing the variational function in the form $V(x)=a(D(x))$, which is composite of a function $D:\cX\to\bbR$ without any range constraints and an output activation function $a:\bbR\to\mathrm{dom}_{f^*}$ specific to the $f$-divergence used. 
To unify the notation, up to some shift and scaling in $a$, we rewrite the $f$-GAN formulation as follows
\begin{equation*}
	\inf_{\theta\in\Theta}\sup_{D\in\cD}\left(-\bbE_{p_*}[l_1^f(X;D)]-\bbE_{p_\theta}[l_2^f(X;D)]\right),
\end{equation*}
where $l_1^f$ and $l_2^f$ for various $f$-divergences are listed in Table \ref{tab:fgan_loss}, and $\cD$ is a family of discriminators that contains the log-density ratio $D^*(x):=\ln(p_*(x)/p_\theta(x))$.

\begin{table}[b]
\centering
\caption{$f$-GAN loss functions.}\label{tab:fgan_loss}
\begin{tabular}{ccc}
\toprule
Divergence & $l_1^f(x;D)$ & $l_2^f(x;D)$\\\midrule
KL & $-D(x)$ & $e^{D(x)}$ \\
RevKL & $e^{-D(x)}$ & $D(x)$ \\
2JS & $\ln(1+e^{-D(x)})$ & $\ln(1+e^{D(x)})$ \\
$H^2$ & $e^{-D(x)/2}$ & $e^{D(x)/2}$ \\
\bottomrule
\end{tabular}
\end{table}

Given the sample $\cS_n=\{x_i,z_j, i=1,\dots,n,j=1,\dots,m\}$, where $x_i$'s are i.i.d. samples from $p_*$, $z_i$'s are i.i.d. samples from $p_z$, and $m=\lambda n$, the empirical formulation of $f$-GAN is given by
\begin{equation}\label{eq:fgan_emp}
	\inf_{\theta\in\Theta}\sup_{D\in\cD}\left[-\frac{1}{n}\sum_{i=1}^nl_1^f(x_i;D)-\frac{1}{m}\sum_{i=1}^ml_2^f(G_\theta(z_i);D)\right].
\end{equation}
Note that when the JS divergence is used as the objective, $f$-GAN recovers the original GAN in (\ref{eq:gan}). The $f$-GAN algorithm to solve \eqref{eq:fgan_emp} is summarized in Algorithm~\ref{alg:fgan}.

To take a closer look at the discriminator estimation, we write the discriminator loss in $f$-GAN as
\begin{equation}\label{eq:obj_fgan_d}
	L_f(D)=\bbE_{X\sim p_*}[l_1^f(X;D)]+\bbE_{X\sim p_\theta}[l_2^f(X;D)].
\end{equation}
We note that $D^*=\argmin_{D} L_f(D)$ for all $f$-divergences listed in Table \ref{tab:f-div}. This suggests that the discriminator in $f$-GAN for various divergences is intended to estimate the same target, $D^*$.
However, limited by the inherent minimax formulation \eqref{eq:obj_fgan}, $f$-GAN adopts different discriminator losses $L_f(D)$ for different $f$-divergences, all of which in general differ from the logistic regression that is the maximum likelihood estimate of $D^*$. Therefore we expect that $f$-GAN suffers from inferior statistical efficiency, which is formally shown in Section~\ref{sec:misspecify}. Motivated by this statistical finding, in the next section, we propose to replace the estimation method of the discriminator in $f$-GAN by logistic regression to leverage the statistical efficiency of the maximum likelihood estimation. 

{\centering
\begin{minipage}{.76\linewidth}
\vskip 0.1in
\begin{algorithm}[H]
\DontPrintSemicolon
\KwInput{Sample $\cS_n$, initial parameter $\theta_0$, meta-parameter $T$}
\For{$t=0,1,2,\dots,T$}{
Generate $\hat{x}_i=G_{\theta_{t}}(z_i)$ for $i=1,\dots,m$\\
$\hat{D}_t=\argmin_{D\in\cD}\big[\frac{1}{n}\sum_{i=1}^nl_1^f(x_i;D)+\frac{1}{m}\sum_{i=1}^ml_2^f(\hat{x}_i;D)\big]$\\
$\theta_{t+1}=\theta_{t}-\eta \nabla_\theta[-\frac{1}{m}\sum_{i=1}^ml_2^f(G_{\theta_{t}}(z_i);\hat{D}_t)]$ for some $\eta>0$
}
\KwReturn{$\theta_T$}
\caption{$f$-GAN}
\label{alg:fgan}
\end{algorithm}
\end{minipage}
\vskip 0.1in
\par
}

\subsection{Adversarial Gradient Estimation} \label{sec:age}

Now we formally present our new method. We denote the objective function by 
\begin{equation}\label{eq:obj}
	L(\theta)=D_f(p_*,p_\theta).
\end{equation}
The following theorem enables us to evaluate the gradients of the $f$-divergence with respect to the generator parameter $\theta$ without resorting to the explicit form of $p_\theta(x)$. It presents a general formula to evaluate gradients that applies to various $f$-divergences with the only difference being the scaling. The proof is given in Appendix \ref{app:pf_grad}.
\begin{theorem}
\label{thm:grad}
Let $r(x):=p_*(x)/p_\theta(x)=\exp(D^*(x))$. Then we have 
\begin{equation}\label{eq:grad}
\nabla_\theta L(\theta)=-\mathbb{E}_{Z\sim p_z}\big[s^*(G_\theta(Z))\nabla_\theta G_\theta(Z)^\top \nabla_x D^*(G_\theta(Z))\big],
\end{equation}
where $s^*(x)=f''\left(1/r(x)\right)/r(x)$ is the scaling factor depending on the $f$-divergence used.
\end{theorem}

Notice that the gradient in (\ref{eq:grad}) depends on the unknown or implicit densities $p_*$ and $p_\theta$. Thus, the gradient cannot be computed directly. To this end, similar to the adversarial scheme in GANs (but without using the standard minimax formulation of GANs), we train a discriminator that directly estimates the log-density ratio $D^*(x)=\ln (p_*(x)/p_\theta(x))$ from the data using logistic regression.

Formally, for random variable $X$, let label $Y=1$ if $X\sim p_*$ and $Y=0$ if $X\sim p_\theta$. That is, the conditional densities are $p(x|Y=1)=p_*(x)$ and $p(x|Y=0)=p_\theta(x)$. Let $\lambda \geq 1$ be the ratio of the number of generated data from $p_\theta(x)$ to the number of real data from $p_*(x)$. 
It implies that $\bbP(Y=1)=1/(1+\lambda)$ and $\bbP(Y=0)=\lambda/(1+\lambda)$.
While only the situation of $\lambda=1$ was considered in \cite{biau2020}, we study the more general situation of $\lambda \geq 1$ because in practice, we generate many more data from $p_\theta(x)$ in GAN, and thus $\lambda$ is often much larger than $1$. In addition, as we will see later, a large $\lambda$ reduces variance, and the variance can approach that of the maximum likelihood estimate as $\lambda \to \infty$ for well-specified models. 

Given $\lambda$, the marginal distribution of $X$ is given by 
\begin{equation}
p_0(x)=\frac{p_*(x)}{1+\lambda}+\frac{\lambda p_\theta(x)}{1+\lambda} . \label{eq:marginal}
\end{equation}
For the family $\cD$ of discriminators, by using the Bayes rule, we can derive the corresponding family of conditional distributions as 
\[
p_D(y|x)= \frac{e^{y(D(x)-\ln\lambda)}}{1+e^{D(x)}/\lambda} , \qquad D\in\cD ,
\]
where we assume that $\|D\|_1=\int|D(x)|p_0(x)dx$. The population version of 
logistic regression corresponds to the loss function
\begin{equation}\label{eq:obj_d}
	L_d(D)=\bbE_{p_*}[\ln (1+e^{-D(X)}\lambda)]+\lambda\bbE_{p_\theta}[\ln (1+e^{D(X)}/\lambda)].
\end{equation}
The population minimizer is $D^*=\argmin_{D\in\cD} L_d(D).$

Given the sample $\cS_n=\{x_i,z_j, i=1,\dots,n,j=1,\dots,m\}$ as in the previous section where $m=\lambda n$, the empirical logistic regression minimizes the loss function
\begin{equation}\label{eq:obj_d_e}
	\hat{L}_d(D)=\frac{1}{n}\sum_{i=1}^{n}\ln(1+e^{-D(x_i)}\lambda)+\frac{\lambda}{m}\sum_{i=1}^{m}\ln(1+e^{D(G_\theta(z_i))}/\lambda).
\end{equation}
Let $\hat D(x)=\argmin_{D\in\cD} \hat{L}_d(D)$ be the solution to the empirical logistic regression problem (\ref{eq:obj_d_e}). As we will show in Section \ref{sec:theory}, under appropriate conditions, when the sample size $n$ is sufficiently large, we have $\hat D(x)\approx D^*(x)$. 

Next, we introduce the method to learn the generator based on the estimated discriminator. 
Given a discriminator $D$, denote the plug-in estimator for the gradient in (\ref{eq:grad}) by 
\begin{equation}\label{eq:est_grad}
	h_D(\theta) = -\frac{1}{m}\sum_{i=1}^m \left[s(G_\theta(z_i);D)\nabla_\theta G_\theta(z_i)^\top \nabla_x  D(G_\theta(z_i))\right],
\end{equation}
where $s(x;D)=f''(1/e^{D(x)})/e^{D(x)}$. 
We can now solve the optimization problem (\ref{eq:obj_f}) using approximate gradient descent, with gradient estimation based on Theorem~\ref{thm:grad} and approximation of $D^*$ given by $\hat{D}$. This leads to Algorithm \ref{alg:age}. Since the proposed approach involves an adversarially learned discriminator, we call it \textit{Adversarial Gradient Estimation (AGE)}. Note that for simplicity, we define the $\theta_t$ with the minimal estimated gradient $\|h_{\hat{D}_{t}}(\theta_{t})\|$ as the algorithm output. In practice, one can use the last iterator $\theta_T$ as the output estimator with similar theoretical guarantee. Let $\theta^*=\argmin_{\theta\in\Theta}L(\theta)$ be the target parameter and $\gage$ be the output of the algorithm with a sufficiently large $T$. We establish the asymptotic convergence of $\gage$ to $\theta^*$ in Section \ref{sec:theory}, which justifies the method statistically.

{\centering
\begin{minipage}{.91\linewidth}
\vskip 0.1in
\begin{algorithm}[H]
\DontPrintSemicolon
\KwInput{Sample $\cS_n$, initial parameter $\theta_0$, meta-parameter $T$}
\For{$t=0,1,2,\dots,T$}{
Generate $\hat{x}_i=G_{\theta_{t}}(z_i)$ for $i=1,\dots,m$\\
$\hat{D}_t=\argmin_{D\in\cD}\big[\frac{1}{n}\sum_{i=1}^{n}\ln(1+e^{-D(x_i)}\lambda)+\frac{\lambda}{m}\sum_{i=1}^{m}\ln(1+e^{D(\hat{x}_i)}/\lambda)\big]$\\
$\theta_{t+1}=\theta_{t}-\eta {h_{\hat{D}_t}(\theta_{t})}$ for some $\eta>0$
}
\KwReturn{$\argmin_{\theta_t:t=1,\dots,T}\|h_{\hat{D}_{t}}(\theta_{t})\|$}
\caption{Adversarial Gradient Estimation (AGE)}
\label{alg:age}
\end{algorithm}
\end{minipage}
\vskip 0.1in
\par
}


\subsection{Comparison between AGE and $f$-GAN}

To compare the two algorithms, we notice that lines 4 of Algorithms \ref{alg:fgan} and \ref{alg:age} are identical for the same $f$-divergence in that
\begin{equation*}
	\nabla_\theta\bigg[-\frac{1}{m}\sum_{i=1}^ml_2^f(G_{\theta}(z_i);D)\bigg]=h_D(\theta).
\end{equation*}
Hence, for the same $f$-divergence, AGE and $f$-GAN algorithms differ only in line 3, which corresponds to the estimation method for the discriminator. We have seen in Section~\ref{sec:fgan} that the discriminator losses in $f$-GAN and AGE share the same population optimal solution $D^*$. In contrast to $f$-GAN that uses different loss functions for different divergences, AGE always adopts logistic regression (\ref{eq:obj_d}) for estimating the discriminator and hence provides a more unified treatment for different divergences. More importantly, in Section \ref{sec:misspecify}, we will analyze the improved statistical efficiency in estimation of both the discriminator and generator benefited from our modification.

In addition, for JS divergence, $f$-GAN differs from the AGE discriminator loss (\ref{eq:obj_d_e}) only in the factor of $\lambda$. When $\lambda=1$, or equivalently $n=m$, both methods share the same discriminator loss. In this case, as explained below, both methods lead to the identical algorithm for JS divergence. However when $\lambda>1$, AGE and $f$-GAN are different algorithms which result in different statistical property, as discussed formally in Section \ref{sec:misspecify}.
In this paper, we do not introduce $\lambda$ correction into the $f$-GAN formulation as in \eqref{eq:obj_d}. In fact, for some $f$-divergences,  such as the KL divergence, the $f$-GAN objective function with $\lambda$ correction is given by 
\begin{equation*}
	\bbE_{p_*}[\ln\lambda-D(X)]+\lambda\bbE_{p_z}[e^{D(G_\theta(Z))}/\lambda],
\end{equation*}
which is identical to the one without correction.

\section{General asymptotic analysis}\label{sec:theory}
In this section, we study the asymptotic properties of the proposed AGE algorithm as well as $f$-GAN. In Section~\ref{sec:consis}, we prove a consistency result of Algorithm~\ref{alg:age} under mild assumptions with a nonparametric discriminator family. In Sections \ref{sec:disc} and \ref{sec:generator}, we consider parametric models, and obtain under appropriate regularity conditions the typical $\sqrt{n}$-rates of convergence and asymptotic normality guarantees for both the discriminator and the generator of AGE and $f$-GAN. 

Note that our estimator $\gage$ is defined as the output of a procedure rather than the solution of an optimization problem, which makes the analysis more involved than the standard asymptotic analysis of an empirical estimator and requires new techniques. Existing works including \cite{biau2020} that study the minimax problem like (\ref{eq:gan}) cannot handle our case. 
%

\subsection{Consistency}\label{sec:consis}

We begin with some additional notations. 
First we explicitly add subscript $\theta$ to the optimal discriminator of Theorem~\ref{thm:grad} as
 $D^*_\theta(x)=\ln(p_*(x)/p_\theta(x))$, the estimated discriminator $\hat{D}_\theta=\argmin \hat{L}_d(D)$ in \eqref{eq:obj_d_e}, and to the marginal distribution of \eqref{eq:marginal} as $p_{0,\theta}$. 
 We then define $h'_D(z;\theta)$ as
\begin{equation}\label{eq:h'_D}
	h'_D(z;\theta)=-s(G_\theta(z);D)\nabla_\theta G_\theta(z)^\top \nabla_x  D(G_\theta(z)) ,
\end{equation}
so that $h_D(\theta)$ of (\ref{eq:est_grad}) can be written as the average of $h'_D(z_i;\theta)$ for $i=1,\ldots,m$.
Moreover, we let $h'(z;\theta)=h'_{D_\theta^*}(z;\theta)$.

To characterize the consistency of the estimated discriminator and generator, we assume the following regularity conditions.
\begin{enumerate}[label=\textit{A\arabic*}]
\setlength{\itemsep}{2pt}
\setlength{\parskip}{2pt}
\item For all $x\in\cX$, $D^*_{\theta}(x)$ and $\nabla_x D^*_{\theta}(x)$ are Lipschitz continuous with respect to $\theta$.\label{ass:D_smooth}
\item The parameter space $\Theta$ is compact and contains $\theta^*$ as an interior point.\label{ass:g_compact}
\item The modeled discriminator class $\cD$ is compact, and contains the true class $\{D^*_\theta:\theta\in\Theta\}$.\label{ass:d_compact}
\item {On any compact subset $K$ of $\cX$, functions in $\cD$ have uniformly bounded function values, gradients and Hessians, i.e., there exists $B>0$ such that $\forall D\in\cD$, $\forall x\in K$, we have $|D(x)|\leq B$, $\|\nabla D(x)\|\leq B$ and $|tr(\nabla^2 D(x))|\leq B$.}\label{ass:D_bound}
\item On any compact subset $K$ of $\cX$, function classes $\cD$ and $\{\nabla D(x):D\in\cD\}$ are uniformly Lipschitz continuous with respect to $x$, i.e., there exists $\ell>0$ such that every $D\in\cD$ and $\nabla D$ are $\ell$-Lipschitz continuous over $x\in K$.\label{ass:D_cont}
\item For all $\theta\in\Theta$, $\bbE_{p_{0,\theta}(x)}[\sup_{D\in\cD}|D(X)|^2]<\infty$ and $\bbE_{p_{0,\theta}(x)}[\sup_{D\in\cD}\|\nabla D(X)\|^2]<\infty$.\label{ass:envelope}  
\item $\bbE_{p_z}[\sup_{\theta\in\Theta,D\in\cD}\|h'_D(Z;\theta)\|]<\infty$ and $\bbE_{p_z}[\sup_{\theta\in\Theta,D\in\cD}|D(G_\theta(Z))|]<\infty$. \label{ass:emp_envelop}
\item $\bbE_{p_z}\|\nabla_\theta G_\theta(Z)\|^2$ is uniformly bounded.\label{ass:G_smooth} 
\item The objective function $L(\theta)$ is $\ell^*$-smooth with respect to $\theta$ and satisfies the Polyak-\L{ojasiewicz} (PL) condition \cite{polyak1963gradient}, i.e., there exists $c>0$ such that for all $\theta\in\Theta$\label{ass:pl}
\[
L(\theta) - L(\theta^*) \leq c \|\nabla L(\theta)\|_2^2.
\]
\end{enumerate}\vspace{-0.1cm}

\begin{remark}\label{rem:ass_eg}
Here, condition \ref{ass:D_smooth} is about the distributions $p_*$ and $p_\theta$. Condition \ref{ass:g_compact} is a common requirement on the parameter space of $\theta$. Conditions \ref{ass:d_compact}-\ref{ass:D_cont} impose requirements on the modeled discriminator class $\cD$, where \ref{ass:d_compact} is a common regularity condition for statistical estimation, \ref{ass:D_bound} assumes the uniform boundedness and \ref{ass:D_cont} assumes the uniform Lipschitz continuity of functions and derivatives in $\cD$. Note that both properties are assumed on a compact subset, which is much easier to be satisfied than on the whole sample space. Conditions \ref{ass:envelope} and \ref{ass:emp_envelop} are the set of envelope conditions \cite{geer2000empirical} to guarantee some uniform convergence statements. Condition \ref{ass:G_smooth} assumes a certain kind of smoothness of the generator. 
Condition \ref{ass:pl} is about the objective function, where the PL condition asserts that the suboptimality of a model is upper bounded by the norm of its gradient, which is a weaker condition than assumptions commonly made to ensure convergence, such as (strong) convexity. Recent literature showed that the PL condition holds for many machine learning scenarios including some deep neural networks \cite{charles2018stability,liu2020loss}.
To better understand these conditions,
we provide a simple concrete example in Appendix~\ref{app:ass_eg},
where all these assumptions hold.

\end{remark}

We now show in the following theorem that under appropriate conditions, the gradient estimator $h_{\hat{D}_\theta}(\theta)$ is a uniformly consistent estimate of the true gradient $\nabla L(\theta)$. The proof is given in Appendix~\ref{app:pf_grad_conv}.

\begin{theorem}\label{thm:grad_conv}
Under conditions A1-A8, we have as $n\to\infty$
\begin{equation}\label{eq:unif_grad}
	\sup_{\theta\in\Theta}\|h_{\hat{D}_\theta}(\theta)-\nabla L(\theta)\|\pto0,
\end{equation}
 where $\pto$ means converging in probability.
\end{theorem}

Based on the consistency of the gradient estimator, we obtain the consistency of Algorithm~\ref{alg:age} in the following theorem whose proof is given in Appendix~\ref{app:pf_cons}.

\begin{theorem}\label{thm:cons}
Under conditions A1-A9, we have $L(\gage)\pto L(\theta^*)$, as $n\to\infty$.
\end{theorem}
\begin{remark}
	Throughout the paper, $\gage$ is the output of Algorithm~\ref{alg:age} with a sufficiently large $T=\mathbf\Theta(n)$ and a sufficiently small learning rate $\eta<1/(12\ell^*)$. In addition, whenever we study the deviation $\gage-\theta^*$, they are associated with the same $f$-divergence. Note that in cases with model misspecification, $\theta^*$ may differ for different $f$-divergences.	
\end{remark}

 

We then add the identifiability assumption and achieve the consistency in terms of the parameter in Corollary~\ref{cor:cons_param}, whose proof is given in Appendix~\ref{app:pf_cons_param}. Such a parameter identifiability assumption is standard in asymptotic statistical analysis. 
\begin{enumerate}
\item[\textit{A10}] For all $\theta$ such that $L(\theta)\to L(\theta^*)$, we have $\theta\to\theta^*$. 
\end{enumerate}

\begin{corollary}\label{cor:cons_param}
	Under conditions A1-A10, we have $\gage\pto \theta^*$, as $n\to\infty$.
\end{corollary}

In the next few sections, motivated by \cite{biau2020}, we will consider parametric models where both the generator $G$ and the discriminator $D$ 
belong to finite dimensional parametric function classes. Under the consistent condition of Theorem~\ref{thm:cons}, we will consider the asymptotic properties of the learned generator and discriminator, as well as their asymptotic efficiency.

\subsection{Asymptotic normality of discriminator}\label{sec:disc}

Now we consider a parametric discriminator family with parameter $\psi\in\Psi$. 
For unifying analysis, we define $l_1(x;\psi)=\ln(1+e^{-D_\psi(x)}\lambda)\text{ and } l_2(x;\psi)=\lambda\ln(1+e^{D_\psi(x)}/\lambda).$
Then the discriminator loss functions in (\ref{eq:obj_d}) and (\ref{eq:obj_d_e}) can be respectively written as
\begin{equation*}
	L_d(\psi,\theta)=\bbE_{p_*(x)}[l_1(X;\psi)]+\bbE_{p_\theta(x)}[l_2(X;\psi)], 
\end{equation*}
and
\begin{equation*}
	\hat{L}_d(\psi,\theta)=\frac{1}{n}\sum_{i=1}^{n}l_1(x_i;\psi)+\frac{1}{m}\sum_{i=1}^{m}l_2(G_\theta(z_i);\psi),
\end{equation*}
where we explicitly express the dependency on $\theta$ in the loss functions. 
Given $\theta$, we define the target parameter $\psi^*$ and the empirical estimator $\dage$  respectively by 
\[\psi^*(\theta)=\argmin_{\psi\in\Psi} L_d(\psi,\theta), \quad \dage(\theta)=\argmin_{\psi\in\Psi} \hat{L}_d(\psi,\theta).\]
Analogously, for $f$-GAN discriminator loss function (\ref{eq:obj_fgan_d}), we denote $l_i^f(x;\psi)=l_i^f(x;D_\psi),i=1,2$, the population loss by
\begin{equation*}
	L_f(\psi,\theta)=\bbE_{X\sim p_*}[l_1^f(X;\psi)]+\bbE_{X\sim p_\theta}[l_2^f(X;\psi)],
\end{equation*}
and the empirical loss by
\begin{equation*}
	\hat{L}_f(\psi,\theta)=\frac{1}{n}\sum_{i=1}^nl_1^f(x_i;\psi)+\frac{1}{m}\sum_{i=1}^ml_2^f(G_\theta(z_i);\psi).
\end{equation*}
{To clarify the notations of various loss functions, throughout the paper, $L$ and $\hat{L}$ without subscripts stand for the population and empirical loss for the generator, respectively; $L_d$ and $\hat{L}_d$ with subscript $d$ denote the population and empirical loss for the discriminator in the AGE method; $L_f$ and $\hat{L}_f$ with subscript $f$ denote the population and empirical loss for the discriminator in the $f$-GAN method.}

Since we have established the consistency of generator estimator $\gage$, we now restrict our discussion in the following sections in a bounded neighborhood $N(\theta^*)$ of $\theta^*$.  
We assume the following regularity conditions hold for all $\theta\in N(\theta^*)$. 
\begin{enumerate}[label=\textit{B\arabic*}]
\setlength{\itemsep}{2pt}
\setlength{\parskip}{2pt}
\item The parameter space $\Psi$ is compact and contains $\psi^*(\theta)$ as an interior point, satisfying $D_{\psi^*(\theta)}=D^*_\theta$. $L_d(\psi,\theta)$ and $L_f(\psi,\theta)$ achieve the unique minimum at $\psi^*(\theta)$. \label{ass:d_realizable}
\item For all $x\in\cX$, $l_i(x;\psi)$ and $l_i^f(x;\psi)$ are three times continuously differentiable with respect to $\psi$, for $i=1,2$. \label{ass:d_cont}
\item The Hessians $\nabla^2_\psi L_d(\psi^*(\theta),\theta)\succ0$ and $\nabla^2_\psi L_f(\psi^*(\theta),\theta)\succ0$. \label{ass:d_pos_hess}
\item $\bbE_{p_*}[\sup_\psi|l_1(X;\psi)|]<\infty$, $\bbE_{p_\theta}[\sup_\psi|l_2(X;\psi)|]<\infty$, $\bbE_{p_*}[\sup_\psi\|\nabla^2_\psi l_1(X;\psi)\|]<\infty$, and $\bbE_{p_\theta}[\sup_\psi\|\nabla^2_\psi l_2(X;\psi)\|]<\infty$, which also hold analogously for $l_1^f,l_2^f$. \label{ass:d_envelop}
\end{enumerate}

In the following two theorems, we obtain the consistency and more importantly the asymptotic normality of the estimated discriminator parameter. See Appendix \ref{app:pf_d_par_cons} and \ref{app:pf_d_eff} for the proofs, respectively.

\begin{theorem}\label{thm:d_par_cons}
	Under conditions B1-B4, for all $\theta\in N(\theta^*)$, we have $\dage(\theta)\pto\psi^*(\theta)$ as $n\to\infty$.
\end{theorem}

\begin{theorem}\label{thm:d_asy_norm}
	Under conditions B1-B4, for all $\theta\in N(\theta^*)$, as $n\to\infty$, we have $$\sqrt{n}\big(\dage(\theta)-\psi^*(\theta)\big)\dto\cN(0,\Sigma_d(\theta)),$$ where $\dto$ means converging in distribution, $\Sigma_d(\theta):=H_d^{-1}V_dH_d^{-1}$, $H_d:=\nabla^2_\psi L_d(\psi^*(\theta),\theta)$, and 
	$V_d:=\mathrm{Var}_{p_*}(\nabla_\psi l_1(X;\psi^*(\theta)))+\mathrm{Var}_{p_\theta}(\nabla_\psi l_2(X;\psi^*(\theta)))/\lambda.$ 
\end{theorem}

We then give out the analogous results on $f$-GAN discriminator estimation. Let 
\[
\dfgan(\theta)=\argmin_{\psi\in\Psi}  \hat{L}_f(\psi,\theta) ,
\]
and $\gfgan$ be the output of Algorithm \ref{alg:fgan}. 
For simplicity, we assume the consistency of $\dfgan$ and $\gfgan$ as follows, which can be derived similarly as in Theorems \ref{thm:grad_conv}, \ref{thm:cons} and \ref{thm:d_par_cons} under some suitable regularity conditions.
\begin{assumption}[$f$-GAN consistency]\label{ass:fgan_cons}
	As $n\to\infty$, we have $\dfgan(\theta)\pto\psi^*(\theta)$ for all $\theta\in N(\theta^*)$, and $\gfgan\pto\theta^*$.
\end{assumption}

The following theorem presents the asymptotic normality of $f$-GAN discriminator estimation. See Appendix \ref{app:pf_f_d_eff} for the proof. In Section \ref{sec:d_eff}, we take a closer look at the asymptotic variances of discriminator estimation of AGE and $f$-GAN, and compare their asymptotic efficiency.

\begin{theorem}\label{thm:f_d_asy_norm}
	Under Assumption \ref{ass:fgan_cons} and conditions B1-B4, for all $\theta\in N(\theta^*)$, as $n\to\infty$, we have $\sqrt{n}(\dfgan(\theta)-\psi^*(\theta))\dto\cN(0,\Sigma_f(\theta))$, where $\Sigma_f(\theta):=H_f^{-1}V_fH_f^{-1}$ with $H_f:=\nabla^2_\psi L_f(\psi^*(\theta),\theta)$ and 
	$V_f:=\mathrm{Var}_{p_*}(\nabla_\psi l_1^f(X;\psi^*(\theta)))+\mathrm{Var}_{p_\theta}(\nabla_\psi l_2^f(X;\psi^*(\theta)))/\lambda.$
\end{theorem}

\subsection{Asymptotic normality of generator}\label{sec:generator}

We proceed to study the asymptotic normality guarantees of the learned generator. Again some additional notations are needed. 
We rewrite the objective function of the generator in (\ref{eq:obj}) as
\[
L(\theta)=\bbE_{Z \sim p_z} \; [l_g(Z;\theta)]
,
\quad
l_g(z;\theta)=e^{D^*_\theta(G_\theta(z))}f\Big(e^{-D^*_\theta(G_\theta(z))}\Big) .
\]
The empirical loss given sample $\cS_n$ can be written as
\begin{equation}\label{eq:obj_g_emp}
	\hat{L}(\theta)=\frac{1}{m}\sum_{i=1}^m l_g(z_i;\theta).
\end{equation} 


%

To obtain the asymptotic distribution of $\gage$, we assume the following regularity conditions.
\begin{enumerate}[label=\textit{C\arabic*}]
\setlength{\itemsep}{2pt}
\setlength{\parskip}{2pt}
\item The parameter space $\Theta$ is compact and contains $\theta^*$ as an interior point. $L_g(\theta)$ achieves the unique minimum at $\theta^*$.
\item For all $z\in\cZ$, $l_g(z;\theta)$ is three times continuously differentiable and smooth with respect to $\theta$. 
\item The Hessian $\bbE_{p_z}[\nabla^2_{\theta}l_g(Z;\theta^*)]\succ0$. \label{ass:g_pos_hess}
\item $\bbE_{p_z}[\sup_{\theta}|l_g(Z;\theta)|]<\infty$. 
\end{enumerate}

We now present the asymptotic normality of the AGE estimator $\gage$, followed by an analogous result for $f$-GAN estimator $\gfgan$. 
See Appendix \ref{app:pf_asy_normal} and \ref{app:pf_f_asy_normal} for the proofs and more insights on the asymptotic behavior of GAN algorithms.

\begin{theorem}\label{thm:asy_normal}
	Under sets A-C of conditions, we have as $n\to\infty$,
\begin{equation*}
	\sqrt{n}(\gage-\theta^*)\dto\cN(0,\mathbf\Sigma),
\end{equation*}
where the asymptotic variance is given by $\mathbf\Sigma=\mathrm{Var}(\zeta+\xi)$ with 
\begin{align*}
	\zeta&:=-H_g^{-1} h'(Z;\theta^*)/\sqrt\lambda,\\
	\xi&:=H_g^{-1}CH_d^{-1}\left[\nabla_{\psi}l_1(X;\psi^*)+\nabla_{\psi}l_2(G_{\theta^*}(Z);\psi^*)/\sqrt\lambda\right],\label{eq:xi}
\end{align*}
where $H_g:=\nabla^2_\theta L(\theta)$, $C:=\nabla_\psi\bbE[h'_{D_{\psi}}(Z;\theta^*)]|_{\psi^*(\theta^*)}$ and $H_d$ is defined in Theorem~\ref{thm:d_asy_norm}.
\end{theorem}

\begin{theorem}\label{thm:f_asy_normal}
	Under conditions B1-B4 and Assumption \ref{ass:fgan_cons}, we have as $n\to\infty$,
\begin{equation*}
	\sqrt{n}(\gfgan-\theta^*)\dto\cN(0,\mathbf\Sigma_f),
\end{equation*}
where the asymptotic variance is given by $\mathbf\Sigma_f=\mathrm{Var}(\zeta+\xi_f)$ with 
\begin{align*}
	\zeta&:=-H_g^{-1} h'(Z;\theta^*)/\sqrt\lambda\\
	\xi_f&:=H_g^{-1}CH_f^{-1}\left[\nabla_{\psi}l^f_1(X;\psi^*)+\nabla_{\psi}l^f_2(G_{\theta^*}(Z);\psi^*)/\sqrt\lambda\right].
\end{align*}
\end{theorem}

When we consider the JS divergence and $\lambda=1$ (i.e., the real data sample and generated sample share the same sample size $n=m$), the above result recovers \cite[Theorem 4.3]{biau2020}. It is also important to point out that the asymptotic variance with $\lambda=1$ as derived in \cite{biau2020} is not the best possible that can be achieved by GAN. Note that by Theorem \ref{thm:f_asy_normal}, as $\lambda$ grows, the variance of $\gfgan$ actually decreases, indicating a more efficient estimator than $\lambda=1$. It will be shown later that in the case of the correctly specified model, with an appropriately chosen local discriminator family, one can achieve the optimal variance with $\lambda \to \infty$, matching that of the maximum likelihood estimate. 
Furthermore, for misspecified generator models (which is also considered in \cite{biau2020}), we will show in Section \ref{sec:g_eff} that the AGE estimator $\gage$ is asymptotically more efficient than the $f$-GAN estimator $\gfgan$ for fixed $\lambda >1$.

In the next two sections, we will study the consequences of the asymptotic theory developed in this section,  and analyze generative algorithms both with and without model misspecification. \revise{Note that in this section, we do not assume on the specification of the generative models.}

\section{Model misspecification}\label{sec:misspecify}

As mentioned in Section \ref{sec:intro}, in most GAN applications, the true distribution $p_*$ is so complex that the generative model will be misspecified, which is formally stated in Assumption \ref{ass:model_mis}. In this section, we compare AGE with $f$-GAN under this common case and show the superiority of AGE in terms of asymptotic efficiency of estimating both the discriminator and the generator. 
\begin{assumption}[Generative model misspecification]\label{ass:model_mis}
	There does not exist $\theta\in\Theta$ such that $p_\theta(x)=p_*(x)$ almost everywhere.
\end{assumption}

However, we assume that the discriminator is still well-specified in that for any $\theta \in \Theta$, $D^*_\theta(x) = \ln (p_*(x)/p_\theta(x)) \in \cD$. Similar to \cite{biau2020}, we can also tolerate a small approximation error in the discriminator, which makes no essential difference. 
\revise{Such as assumption ties in with the fact that generative models usually require stronger assumptions on model specification than discriminative models~\cite{hastie2009elements}. For example, in linear discriminant analysis which is a generative model, different classes are assumed to be Gaussian distributed with a common covariance matrix so that the log-density ratio of two classes has a linear form. In contrast, in discriminative classification, the assumption of a linear discriminator class does not require Gaussians.}

\subsection{Discriminator efficiency}\label{sec:d_eff}
In Theorems \ref{thm:d_asy_norm} and \ref{thm:f_d_asy_norm}, we have the asymptotic normality of AGE estimator $\dage$ and $f$-GAN estimator $\dfgan$ for the discriminator. To compare their asymptotic efficiency, we explicitly compute the asymptotic variances of the discriminators of AGE and $f$-GAN for all $f$-divergences listed in Table \ref{tab:f-div}. The results are summarized in Theorem \ref{thm:d_asy_var_compare}, which indicates that $\dage$ is asymptotically more efficient than $\dfgan$. See Appendix \ref{app:pf_d_asy_var_compare} for the calculations and proof. 

\begin{theorem}\label{thm:d_asy_var_compare}
	Suppose Assumptions \ref{ass:fgan_cons}-\ref{ass:model_mis} and conditions B1-B4 hold. Without loss of generality, suppose the first dimension of the parameter $\psi$ corresponds to the intercept, i.e., the first entry of $\nabla_\psi D_\psi(x)$ equals 1. Then Table \ref{tab:var_d} explicitly lists the asymptotic variances of interest, all of which we assume to be finite. Furthermore, for $1<\lambda<\infty$, for all $\theta\in N(\theta^*)$, the asymptotic variances of AGE estimator $\dage(\theta)$ and $f$-GAN estimator $\dfgan(\theta)$ satisfy $\Sigma_d(\theta)\prec\Sigma_f(\theta)$. 
\end{theorem}

\begin{table}
\centering
\caption{Discriminator asymptotic variances. The first line is AGE and the remaining four lines are $f$-GAN for various divergences.  We use subscripts $k,r,j$ and $h$ to stand for KL, reverse KL, JS divergence and squared Hellinger distance, respectively. $\Sigma_0$ is a matrix whose first diagonal entry is 1 while all other entries are 0.}\label{tab:var_d}
\def\arraystretch{1.5}
\begin{tabular}{@{}cc@{}}
\toprule\\[-5.6ex]
Method & Variance \\[-.5ex]\midrule
AGE & $\Sigma_d=\bbE_{p_*}^{-1}\big[\frac{p_\theta}{p_\theta+p_*/\lambda}{\nabla_\psi D_{\psi^*}}^{\otimes2}\big]-\big(1+\frac{1}{\lambda}\big)\Sigma_0$ \\
$f$-KL & $\Sigma_k=\bbE_{p_*}^{-1}\big[{\nabla_\psi D_{\psi^*}}^{\otimes2}\big] \bbE_{p_*}\big[\frac{p_\theta+p_*/\lambda}{p_\theta}{\nabla_\psi D_{\psi^*}}^{\otimes2}\big] \bbE_{p_*}^{-1}\big[{\nabla_\psi D_{\psi^*}}^{\otimes2}\big]-\big(1+\frac{1}{\lambda}\big)\Sigma_0$ \\
$f$-RevKL & $\Sigma_r=\bbE_{p_\theta}^{-1}\big[{\nabla_\psi D_{\psi^*}}^{\otimes2}\big] \bbE_{p_\theta}\big[\frac{p_\theta+p_*/\lambda}{p_*}{\nabla_\psi D_{\psi^*}}^{\otimes2}\big] \bbE_{p_\theta}^{-1}\big[{\nabla_\psi D_{\psi^*}}^{\otimes2}\big]-\big(1+\frac{1}{\lambda}\big)\Sigma_0$ \\
$f$-2JS & $\Sigma_j=\bbE_{p_*}^{-1}\big[\frac{p_\theta}{p_\theta+p_*}{\nabla_\psi D_{\psi^*}}^{\otimes2}\big]\bbE_{p_*}\big[\frac{p_\theta(p_\theta+p_*/\lambda)}{(p_\theta+p_*)^2}\nabla D^{\otimes2}\big] \bbE_{p_*}^{-1}\big[\frac{p_\theta}{p_\theta+p_*}\nabla D^{\otimes2}\big]-\big(1+\frac{1}{\lambda}\big)\Sigma_0$ \\
$f$-$H^2$ & $\Sigma_h=\bbE_{p_*}^{-1}\big[\sqrt{\frac{p_\theta}{p_*}}{\nabla_\psi D_{\psi^*}}^{\otimes2}\big] \bbE_{p_*}\big[\frac{p_\theta+p_*/\lambda}{p_*}{\nabla_\psi D_{\psi^*}}^{\otimes2}\big] \bbE_{p_*}^{-1}\big[\sqrt{\frac{p_\theta}{p_*}}{\nabla_\psi D_{\psi^*}}^{\otimes2}\big]-\big(1+\frac{1}{\lambda}\big)\Sigma_0$ \\
\bottomrule
\end{tabular}
\end{table}

As suggested in Theorem \ref{thm:d_asy_var_compare}, the two crucial assumptions for AGE to enjoy more efficient discriminator estimation than $f$-GANs are model misspecification and a finite $\lambda$. 
Regarding model specification, in the rare case where the model is correctly specified, when $p_*=p_\theta$, all variances are identical. However, this almost never happens in applications of GANs. 
Moreover, empirically one observes that at the early stage of the algorithms, the two distributions $p_*$ and $p_\theta$ often differ significantly. In such case, the asymptotic variance of discriminator estimation in AGE is strictly smaller than those in $f$-GANs. Therefore, one can expect that the AGE algorithm is more robust empirically. This is confirmed by the simulation results illustrating that in some cases AGE has much smaller variances than $f$-GAN.

Regarding the ratio $\lambda$ of the real and generated sample sizes, we notice that all the variances (ingoring the intercept) decreases as $\lambda$ grows. More specifically, we have the following proposition, which suggests as $\lambda\to\infty$, $f$-GAN-KL becomes as efficient as AGE, while $f$-GANs for the other three divergences still remains inferior to AGE. See Appendix \ref{app:pf_var_diff} for the proof. In practice, due to the computational complexity, only a finite $\lambda$ is applicable, in which case AGE is favored.
\begin{proposition}\label{prop:var_diff}
	Under Assumption \ref{ass:model_mis}, for all $\theta\in\Theta$, as $\lambda\to\infty$, we have $\|\Sigma_k(\theta)-\Sigma_d(\theta)\|=\mathbf\Theta(1/\lambda)$, and $\|\Sigma_{f}(\theta)-\Sigma_d(\theta)\|=\mathbf\Theta(1)$ for $f=r,j,h$.
\end{proposition}

\subsection{Generator efficiency}\label{sec:g_eff}

In Theorems \ref{thm:asy_normal} and \ref{thm:f_asy_normal}, we have the asymptotic normality of AGE estimator $\gage$ and $f$-GAN estimator $\gfgan$ for the generator. 
We first simplify the asymptotic variances of $\gage$ and $\gfgan$ as follows to obtain more informative conclusions:
\begin{equation}\label{eq:var_g}
	\mathbf\Sigma=H_g^{-1}C\Sigma_{d,\theta^*} C^\top H_g^{-1}+\frac{1}{\lambda}H_g^{-1}\left(\bbE[h'(Z;\theta^*)h'(Z;\theta^*)^\top]+\Sigma_c+\Sigma_c^\top\right)H_g^{-1},
\end{equation}
where $\Sigma_{d,\theta^*}:=\Sigma_d(\theta^*)$ is the asymptotic variance of AGE discriminator estimator $\dage(\theta^*)$ at the optimal generator $\theta^*$, and $\Sigma_c:=\mathrm{Cov}\big(-h'(Z;\theta^*),CH_d^{-1}\nabla_{\psi}l_2(G_{\theta^*}(Z);\psi^*(\theta^*))\big)$.
\begin{equation}\label{eq:var_g_fgan}
	\mathbf\Sigma_f=H_g^{-1}C\Sigma_{f,\theta^*} C^\top H_g^{-1}+\frac{1}{\lambda}H_g^{-1}\left(\bbE[h'(Z;\theta^*)h'(Z;\theta^*)^\top]+\Sigma^f_c+{\Sigma_c^f}^\top\right)H_g^{-1},
\end{equation}
where $\Sigma_{f,\theta^*}:=\Sigma_f(\theta^*)$ and $\Sigma_c^f:=\mathrm{Cov}\big(-h'(Z;\theta^*),CH_f^{-1}\nabla_{\psi}l^f_2(G_{\theta^*}(Z);\psi^*(\theta^*))\big)$. 

Now we compare the asymptotic variances $\mathbf\Sigma$ and $\mathbf\Sigma_f$ under model misspecification. We notice that for a particular $f$-divergence, they only differ in two terms: $\Sigma_{d,\theta^*}$ and $\Sigma_c$ in (\ref{eq:var_g}) versus $\Sigma_{f,\theta^*}$ and $\Sigma_c^f$ in (\ref{eq:var_g_fgan}).
By Theorem \ref{thm:d_asy_var_compare}, we know $\Sigma_{d,\theta^*}\prec \Sigma_{f,\theta^*}$ under model misspecification, so the first term in $\mathbf\Sigma$ is small than the first term in $\mathbf\Sigma_f$. However, theoretical comparison of the covariance terms $\Sigma_c$ and $\Sigma_c^f$ in general cases remains open due to the complication in calculating the covariance terms. 

We analyze their difference as follows.
From Proposition \ref{prop:var_diff}, we know that for the reverse KL, JS divergence and squared Hellinger distance, the difference in the first terms of (\ref{eq:var_g}) and (\ref{eq:var_g_fgan}) is of order $O(1)$, while the difference in the second terms is of order $O(1/\lambda)$. Hence, we can set a large enough $\lambda$ such that the overall variances satisfy $\mathbf\Sigma\prec\mathbf\Sigma_f$, for $f=r,j,h$. For the KL divergence, however, 
the differences in both the first and the second terms of $\mathbf\Sigma$ and $\mathbf\Sigma_k$ are of $O(1/\lambda)$, which indicates that when $\lambda$ becomes sufficiently large, AGE-KL and $f$-GAN-KL achieves the same asymptotic variance in generator estimation. This is consistent with the discriminator behavior in the KL case. Nevertheless, for more practical cases with finite $\lambda$, we consider a one-dimensional special case and explicitly compute $\Sigma_c$ and $\Sigma_c^f$ based on numerical integration. As shown in Figure \ref{fig:cov}, AGE-KL always achieves a smaller covariance term than $f$-GAN-KL with varying $\lambda$. Moreover, we conduct simulations in Section \ref{sec:experiment} to empirically demonstrate that AGE generally achieves smaller variances than $f$-GAN for all divergences with a wide range of finite $\lambda$. 

\begin{figure}
\centering
\includegraphics[width=0.6\linewidth]{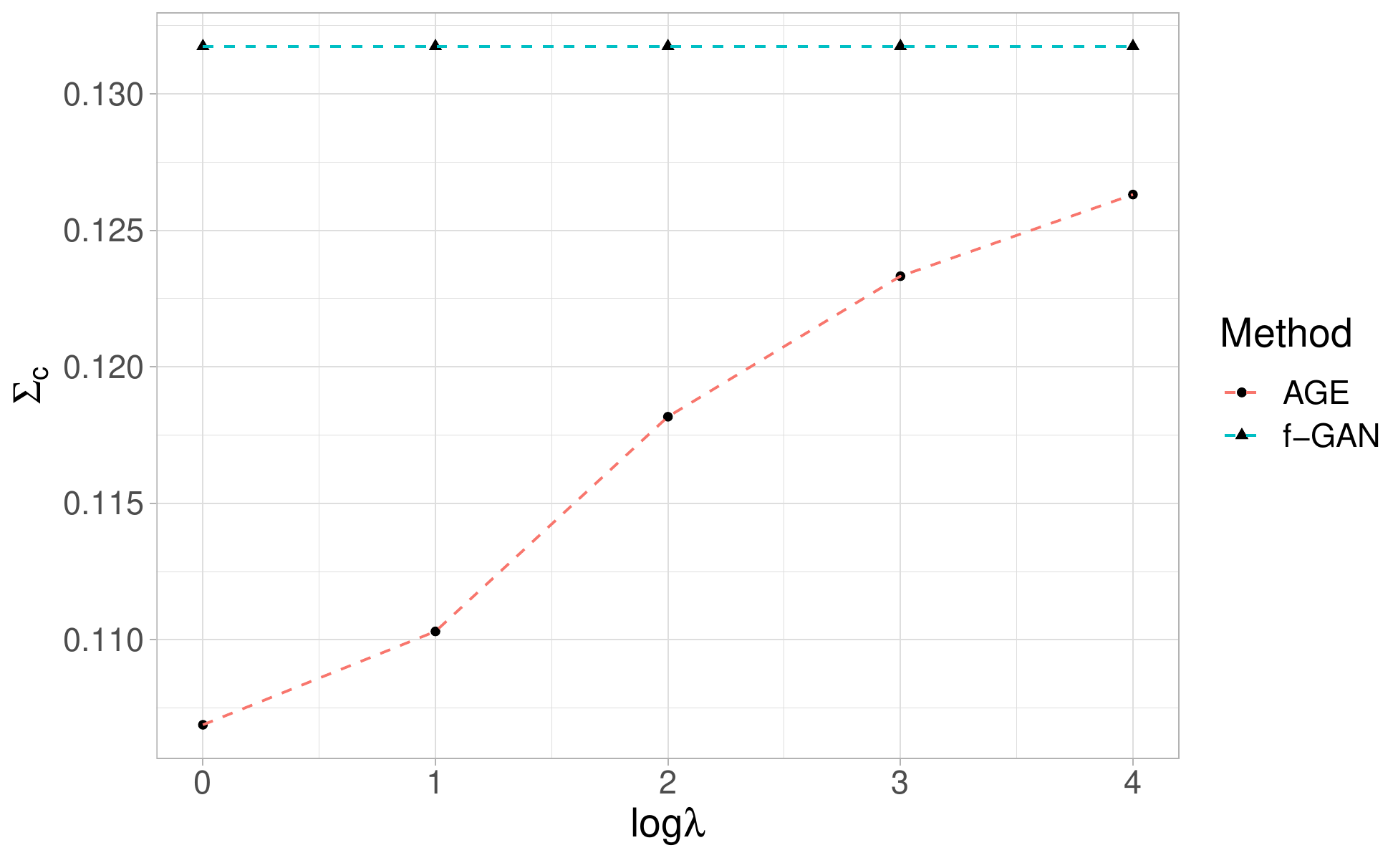}
\caption{Values of $\Sigma_c$ (AGE-KL) and $\Sigma_c^f$ ($f$-GAN-KL) in the Laplace-Gaussian example described in Section \ref{sec:exp_g} with $\lambda$ varying from 1 to $10^4$.}
\label{fig:cov}
\end{figure}

Alternatively, one may consider using two independent samples in updating the discriminator and generator. Formally, suppose we have two independent samples $\cS^1_n=\{x_1,\dots,x_n,z^1_1,\dots,z^1_m\}$ and $\cS^2_m=\{z^2_1,\dots,z^2_m\}$, where $x_i$'s are i.i.d. samples from $p_*$, $z^1_i$'s and $z^2_i$'s are i.i.d. samples from $p_z$. We use $\cS^1_n$ to estimate the discriminator in AGE and $f$-GAN (line 3 of Algorithms~\ref{alg:fgan} and \ref{alg:age}) and $\cS^2_m$ to update the generator (line 4 of Algorithms~\ref{alg:fgan} and \ref{alg:age}). This scheme is feasible in practice because we can generate as many data $G(Z)$ as we desire by first drawing $Z$ from $p_z$ and then transforming it using generator $G$. 

Under the two-sample scheme, the asymptotic results still hold where the covariance term in asymptotic variance \eqref{eq:var_g}  becomes
\begin{equation*}
	\Sigma_c=\mathrm{Cov}\big(-h'(Z_1;\theta^*),CH_d^{-1}\nabla_{\psi}l_2(G_{\theta^*}(Z_2);\psi^*(\theta^*))\big)=\mathbf{0},
\end{equation*}
where $Z_1$ and $Z_2$ independently follows $p_z$. Similarly, the covariance term $\Sigma_c^f$ in \eqref{eq:var_g_fgan} also becomes a zero matrix. Then the asymptotic variances $\mathbf\Sigma$ and $\mathbf\Sigma_f$ differ only in $\Sigma_{d,\theta^*}$ versus $\Sigma_{f,\theta^*}$. Since $\Sigma_{d,\theta^*}\prec \Sigma_{f,\theta^*}$ under model misspecification, we immediately have $\mathbf\Sigma\prec\mathbf\Sigma_f$, that is, the AGE estimator $\gage$ has a smaller asymptotic variance than $f$-GAN estimator $\gfgan$. This directly suggests how the generator estimation benefits from the efficient discriminator estimation adopted by AGE.

\section{Correct model specification}\label{sec:well_specify}

When the generative model is correctly specified, let $\theta^*\in\Theta$ be such that $p_{\theta^*}(x)=p_*(x)$, a.e. Then $\theta^*$ is naturally the target parameter defined earlier by $\argmin_{\theta\in\Theta} L(\theta)$ for various $f$-divergences. Therefore, the first interesting question is how different $f$-divergences behave in estimating the common true value $\theta^*$. 
Moreover, it is well known that when the model is correctly specified, under mild regularity conditions, maximum likelihood estimation (MLE) achieves the optimal parametric efficiency among the class of asymptotically unbiased estimators. Thus, another intriguing question is about the relationship between GAN and MLE, in particular, how GAN behaves compared with MLE and whether GAN can be as optimal as MLE in terms of statistical efficiency. 
Although in most applications we do not enjoy correct model specification, this section is mainly devoted to theoretical understanding, in which we satisfactorily answer the above questions.

\subsection{Various $f$-divergence GANs are asymptotically equivalent}\label{sec:f_equiv}
We have the following corollary from the general asymptotic normality results in Theorems \ref{thm:asy_normal} and \ref{thm:f_asy_normal}. See Appendix \ref{app:pf_f_equiv} for the proof.

\begin{corollary}\label{cor:f_equiv}
	Assume $p_*\in\{p_\theta:\theta\in\Theta\}$. Under sets A-C of conditions, as $n\to\infty$, we have $\sqrt{n}(\gage-\theta^*)\dto\cN(0,\mathbf\Sigma')$ and $\sqrt{n}(\gfgan-\theta^*)\dto\cN(0,\mathbf\Sigma')$, where $\mathbf\Sigma':={H'_g}^{-1}C'\Sigma'_{d,\theta^*} {C'}^\top {H'_g}^{-1}$ with $H'_g:=\bbE\big[\nabla_\theta[\nabla_\theta G_{\theta^*}(Z)^\top\nabla_x D^*_{\theta^*}(G_{\theta^*}(Z))]\big]$, $C':=\bbE[\nabla_\theta G_{\theta^*}(Z)^\top\nabla_\psi\nabla_x D_{\psi}(G_{\theta^*}(Z))|_{\psi^*(\theta^*)}]$, and $\Sigma'_{d,\theta^*}:=(1+1/\lambda)\left(\bbE_{p_*}^{-1}\big[{\nabla_\psi D_{\psi^*}}^{\otimes2}\big]-\Sigma_0\right)$, all of which do not depend on the $f$-divergence used.
\end{corollary}

In Corollary \ref{cor:f_equiv}, we find that the asymptotic distributions of both AGE estimator $\gage$ and $f$-GAN estimator $\gfgan$ for all $f$-divergences listed in Table \ref{tab:f-div} are the same, which indicates that under the circumstances with correct model specification, AGE and $f$-GAN for all $f$-divergences are asymptotically equivalent. 
Their difference only comes out in the case of a misspecified generative model, where AGE is provably favorable, as shown earlier in Section \ref{sec:misspecify}.

\revise{
Note that the asymptotic equivalence of statistical inference based on $f$-divergences can also be interpreted from a different perspective of information geometry~\cite{amari2000methods} by analyzing the property of $f$-divergences. In fact, in the correctly specified case, $f$-divergences are reduced to the $\chi^2$-divergence asymptotically up to higher order terms~\cite{nielsen2013chi}. Specifically, we define
\begin{equation}\label{eq:chi-div}
	\chi^i(p_*,p_\theta)=\int\frac{(p_\theta(x)-p_*(x))^i}{p_*(x)^{i-1}}dx,
\end{equation}
for $i=0,1,2,\dots$, where $i=2$ corresponds to the $\chi^2$-divergence. Then under mild conditions, we have
\begin{equation}\label{eq:f_div_chi}
	D_f(p_*,p_\theta)=\frac{f''(1)}{2}\chi^2(p_*,p_\theta)+\sum_{i=3}^\infty\frac{1}{i!}f^{(i)}(1)\chi^i(p_*,p_\theta).
\end{equation}
This suggests that the statistical inference using $f$-divergences leads to the same asymptotic variance. 
We provide the technical details including the proof of \eqref{eq:f_div_chi} in Appendix \ref{app:pf_f_equiv}. Nevertheless, in the problem of generative models, we in general face the case where standard statistical estimation such as empirical $f$-divergence minimization (e.g., maximum likelihood estimation), does not work. Instead, we resort to procedures like GANs whose statistical properties are not fully reflected from the behavior of the $f$-divergence itself at the population level. To analyze the asymptotic equivalence of $f$-divergence GANs, we provide a detailed argument regarding the variance induced from the estimated discriminator, where the effect of $\lambda$ is taken into account. Therefore, Corollary~\ref{cor:f_equiv} provides a complementary result to the asymptotic equivalence of $f$-divergence \eqref{eq:f_div_chi} in the context of generative models.
}

\subsection{Optimal GAN}\label{sec:opt_gan}
This section discusses the relationship between GAN and MLE. We start with a simple example of estimating the mean of a multivariate Gaussian distribution, where we show that GAN can achieve the optimal (among the class of asymptotically unbiased estimators) efficiency of MLE. Then we extend beyond the Gaussian case and propose a local GAN approach which can provably be as efficient as MLE under general circumstances. Note from the previous section that in the correctly specified case, AGE and $f$-GAN are asymptotically equivalent, so here we only study the AGE algorithm as a representative of GANs.


\subsubsection{Gaussian mean estimation}\label{sec:gaus_mean_est}
Consider the true distribution $p_*=\cN(\mu_0,\id_d)$, where $\mu_0\in\bbR^d$ is the mean vector and $\id_d$ denotes a $d$-dimensional identity matrix. The goal is to learn the mean vector $\mu_0$ using a generator class $\{G_\theta(Z)=\theta+Z:\theta\in\Theta\}$, where $Z\sim\cN(0,\id_d)$ and $\Theta$ is a compact parameter space containing the true value $\theta^*=\mu_0$ as an interior point. The induced generated distribution is given by $p_\theta(x)=\exp(-(d\ln(2\pi)+\theta^\top\theta)/2+\theta^\top x-x^\top x/2)$.

Given an i.i.d. sample $\{x_1,\dots,x_n\}$ from $p_*$, the maximum likelihood estimator is defined by 
\begin{equation}\label{eq:mle}
	\hat\theta_{\mathrm{MLE}}=\argmax_{\theta\in\Theta}\frac{1}{n}\sum_{i=1}^n \ln p_\theta(x_i).
\end{equation}
According to the classical MLE theory \cite[Theorem 3.10]{lehmann2006theory}, we have as $n\to\infty$,  
\begin{equation*}
	\sqrt{n}(\hat\theta_{\mathrm{MLE}}-\theta^*)\dto\cN(0,\cI(\theta^*)^{-1}),
\end{equation*}
where the asymptotic variance is the Fisher information $\cI(\theta^*)=-\bbE_{p_*}[\nabla^2_\theta\ln p_{\theta^*}(x)]=\id_d$ which achieves the Cram{\'e}r-Rao lower bound of $\theta$. 

Note that when solving (\ref{eq:mle}) analytically or using gradient descent, one leverages the explicit form of $p_\theta(x)$. Now we consider utilizing such information in an alternative way under the GAN framework. We obtain the optimal discriminator for each generator $\theta$ in the following form 
\begin{equation*}
	D^*_\theta(x)=\ln(p_*(x)/p_\theta(x))=(\mu_0-\theta)^\top x+(\theta^\top\theta-\mu_0^\top \mu_0)/2,
\end{equation*}
which motivates us to construct a linear discriminator class
\begin{equation*}
	\cD=\{D_\psi(x)=\psi_0+\psi_1^\top x:\psi_0\in\bbR, \psi_1\in\bbR^d,\psi=(\psi_0,\psi_1^\top)^\top\in\Psi\},
\end{equation*}
where $\Psi$ is a compact subset of $\bbR^{d+1}$ containing the optimal discriminator class $\{\psi^*\in\bbR^{d+1}:D_{\psi^*}=D^*_\theta,\theta\in\Theta\}$. Hence this example is a relatively simple case of GAN in the sense that both the discriminator and the generator are linear. 

The following corollary for the AGE estimator $\gage$ can be obtained from the general asymptotic normality result in Theorem \ref{thm:asy_normal}. See Appendix \ref{app:pf_cor_gaus} for the proof.
\begin{corollary}\label{cor:gaus}
	Under the above scenario, as $n\to\infty$, we have 
	\begin{equation*}
		\sqrt{n}(\gage-\theta^*)\dto\cN(0,(1+1/\lambda)\id_d).
	\end{equation*}
\end{corollary}

Here we focus only on the statistical complexity concerning the sample size $n$ of the real data, while do not care about the computational complexity. In practice we can generate as many data $G(Z)$ as we desire by first drawing $Z$ from $p_z$ and then transforming it using $G$. Therefore in practice, we can use a large $\lambda$ if we ignore its computational complexity. Notice in Corollary \ref{cor:gaus} that as $\lambda\to\infty$, the asymptotic variance of $\gage$ approaches $\id_d=\cI(\theta^*)^{-1}$ which coincides with that of MLE. This suggests that if we take $\lambda =n$ for each $n$, then as $n \to \infty$,  GAN achieves the same asymptotic variance as that of MLE. Therefore, in this case, GAN is asymptotically as efficient as MLE.

\subsubsection{Local GAN with score discriminator}\label{sec:localgan}


Next, we consider a general true distribution $p_*$ and a general generator $G_\theta(Z)$, with $Z\sim \cN(0,\id)$, which induces the generated distribution class $\{p_\theta:\theta\in\Theta\}$. Here, unlike the Gaussian case, globally we no longer have a linear discriminator class. However, as we will elaborate next, given a root-$n$ consistent generator estimator, locally there exists a linear discriminator class with the Fisher score as its feature, which can be utilized to develop a local GAN algorithm that can provably achieve the same asymptotic variance as MLE. 

Let $S(\theta;x)=\nabla_\theta\ln p_\theta(x)$ be the Fisher score function. Although in most GAN applications, $S(\theta;x)$ cannot be explicitly computed (e.g., when the generator $G_\theta$ is a multilayer perceptron with the ReLU activation function), here to study the relationship between GAN and MLE, since MLE utilizes the explicit from of the score, we also assume it can be computed, which can be readily generalized to the case with a root-$n$ consistent score estimator. 

Suppose the following regularity conditions hold, where conditions \textit{D1-D3} are commonly required in the asymptotic theory of MLE, and condition \textit{D4} is required for the generator, which can be easily shown to hold in the above Gaussian case.
\begin{enumerate}[label=\textit{D\arabic*}]
\setlength{\itemsep}{2pt}
\setlength{\parskip}{2pt}
\item The support $\{x\in\cX:p_\theta(x)>0\}$ is independent of $\theta$.
\item For all $x\in\cX$, the density $p_\theta(x)$ is three times differentiable with respect to $\theta$.
\item For all $\theta\in\Theta$, $\bbE_{p_\theta}[\nabla_\theta p_\theta(X)]=0$ and $\bbE_{p_\theta}[-\nabla^2_\theta p_\theta(X)]=\bbE_{p_\theta}[\nabla_\theta p_\theta(X)\nabla_\theta p_\theta(X)^\top]=:\cI(\theta)\succ0$.
\item For all $z\in\cZ$, $S(\theta';G_\theta(z))$ is continuous in $(\theta',\theta)$; $\sup_{\theta',\theta\in\Theta}\bbE_{p_z}[\nabla S(\theta';G_\theta(Z))]<\infty$.\label{ass:localgan_ulln}
\end{enumerate}
Under conditions \textit{D1-D3}, the MLE, defined by (\ref{eq:mle}), satisfies $$\sqrt{n}(\hat\theta_{\mathrm{MLE}}-\theta^*)\dto\cN(0,\cI(\theta^*)^{-1})$$ as $n\to\infty$, where now the Fisher information is no longer an identity matrix as in the Gaussian example.

Now we describe the approach to make GAN as efficient as MLE. According to Theorem \ref{thm:asy_normal}, we can obtain a root-$n$ consistent estimator $\hat\theta$ (i.e., $\|\hat\theta-\theta^*\|=O_p(1/\sqrt{n})$) by adopting the AGE algorithm normally, whose asymptotic variance, however, may not be optimal. Then based on $\hat\theta$, we adopt a local GAN within a neighborhood of $\theta^*$ with a radius of order $O_p(1/\sqrt{n})$. When $\|\theta-\theta^*\|=O(1/\sqrt{n})$, we have for all $x\in\cX$,
\begin{equation}\label{eq:linear_d}
	D^*(x)=\ln(p_{\theta^*}(x)/p_\theta(x))=(\theta^*-\theta)^\top S(\theta^*;x) +O(1/n),
\end{equation}
where the leading term is linear with respect to the score $S(\theta^*;x)$. 
Due to the root-$n$ consistency of $\hat\theta$, we have for all $x\in\cX$ that $S(\hat\theta;x)-S(\theta^*;x)=O_p(1/\sqrt{n})$. Then we replace the score at the unknown true value $\theta^*$ in (\ref{eq:linear_d}) by the score at its estimator and obtain 
\begin{equation*}
	D^*(x)=(\theta^*-\theta)^\top S(\hat\theta;x) +O_p(1/n).
\end{equation*}
This motivates us to construct discriminators with the estimated score as the feature, which gives a linear discriminator class in terms of the score
\begin{equation}\label{eq:localgan_dis}
	\cD_l=\{D_\psi(s)=\psi^\top s: \psi\in\Psi\},
\end{equation}
where $\Psi$ is a compact subset of $\bbR^{d_\theta}$ containing the approximate optimal discriminator class $\{\psi^*=\theta^*-\theta: \theta\in\Theta\}$, and $d_\theta$ is the dimension of $\theta$.

Based on the linear score discriminator, the gradient estimator \eqref{eq:est_grad} can be written as
\begin{eqnarray*}
	\tilde{h}_\psi(\theta)=\frac{1}{m}\sum_{i=1}^m \left[-\nabla_\theta S(\hat\theta;G_\theta(z_i))^\top\nabla_s D_\psi(S(\hat\theta;G_\theta(z_i)))\right].
\end{eqnarray*}
Note that since various $f$-divergences lead to asymptotically equivalent algorithms in the correctly specified case by Corollary \ref{cor:f_equiv}, here we use the gradient estimator corresponding to the reverse KL divergence which has the simplest scaling factor, 1. 
We summarize the whole procedure of local GAN in Algorithm \ref{alg:localgan}.

{\centering
\begin{minipage}{.92\linewidth}
\vskip 0.1in
\begin{algorithm}[H]
\DontPrintSemicolon
\KwInput{Sample $\cS_n$, meta-parameter $T$}
Obtain an initial estimator $\hat\theta$, e.g., from Algorithm \ref{alg:age}\\
Initial parameter $\theta_0=\hat\theta$\\
\For{$t=0,1,2,\dots,T$}{
True scores $s_i=S(\hat\theta;x_i)$ for $i=1,\dots,n$\\
Generated scores $\hat{s}_i=S(\hat\theta;G_{\theta_{t}}(z_i))$ for $i=1,\dots,m$\\
$\hat{\psi}_t=\argmin_{\psi\in\Psi}\big[\frac{1}{n}\sum_{i=1}^{n}\ln(1+e^{-D_\psi(s_i)}\lambda)+\frac{\lambda}{m}\sum_{i=1}^{m}\ln(1+e^{D_\psi(\hat{s}_i)}/\lambda)\big]$\\
$\theta_{t+1}=\theta_{t}-\eta \tilde{h}_{\hat{\psi}_t}(\theta_{t})$ for some $\eta>0$
}
\KwReturn{$\argmin_{\theta_t:t=1,\dots,T}\|\tilde{h}_{\hat{\psi}_t}(\theta_{t})\|$}
\caption{Local GAN}
\label{alg:localgan}
\end{algorithm}
\end{minipage}
\vskip 0.1in
\par
}

Let $\hat\theta_{\mathrm{local}}$ be the output of Algorithm \ref{alg:localgan} with a sufficiently large $T=\mathbf\Theta(n)$. Then the following corollary from Theorem \ref{thm:asy_normal} provides the asymptotic normality result of the local GAN estimator, whose proof is given in Appendix \ref{app:pf_cor_localgan}.
\begin{corollary}\label{cor:localgan}
	Under the above scenario and conditions \textit{D1-D4}, as $n\to\infty$, we have 
	\begin{equation*}
		\sqrt{n}(\hat\theta_{\mathrm{local}}-\theta^*)\dto\cN\big(0,(1+1/\lambda)\cI(\theta^*)^{-1}\big).
	\end{equation*}
\end{corollary}

Similar to the discussion in the Gaussian case, without the worry of computational complexity, by letting $\lambda=n$, the asymptotic variance of $\hat\theta_{\mathrm{local}}$ becomes identical to that of MLE. Therefore, in this more general case, as long as we have access to an infinite amount of generated data, GAN can be asymptotically as efficient as MLE through the process of local GAN. In other cases where a larger discriminator class is adopted, the variance of GAN may be consequently enlarged, resulting in a less efficient estimator. 

\revise{We would like to point out that the idea of the local GAN is related to Le Cam's one-step estimator~\cite{le1956asymptotic}, which is a method in statistics to attain statistical efficiency based on a consistent but possibly inefficient estimator. Considering the same setup as MLE, where we have a parametrized model class $\{p_\theta\}$, and given a root-$n$ consistent estimator $\hat\theta$, the one-step estimator is given by a single iterative step of Newton's method:
\begin{equation*}
	\hat\theta^{(1)}=\hat\theta-\left(\sum_{i=1}^n\nabla^2_\theta\ln p_{\hat\theta}(x_i)\right)^{-1}\sum_{i=1}^n\nabla_\theta\ln p_{\hat\theta}(x_i).
\end{equation*}
Then under regularity conditions \textit{D1-D3}, we have as $n\to\infty$ 
\begin{equation*}
	\sqrt{n}(\hat\theta^{(1)}-\theta^*)\dto\cN(0,\cI(\theta^*)^{-1}),
\end{equation*}
which means that $\hat\theta^{(1)}$ is asymptotically as efficient as MLE. 
In the idea of local GAN, we have in mind the same spirit to utilize an initial estimator which is consistent but inefficient. Local GAN can be viewed as a method that implement this idea under the framework of GAN to investigate whether a modified version of GAN can also attain the Cram\'er-Rao lower bound. Due to the complication in GANs, our proposed local GAN is a different and much more involved approach than the one-step estimator.  
}

\revise{In addition, GANs with a linear discriminator class have appeared in previous work from the perspective of optimization. For example, \cite{liushuang2017} and \cite{dualing2017} considered a linear discriminator class with general feature maps. In particular, \cite{liushuang2017} showed that at a population level, the solution set of linear $f$-GAN satisfies the desired moment matching condition in terms of features. \cite{dualing2017} reformulated the saddle point objective into a maximization problem based on conjugate duality when restricted to linear discriminators. Neither work analyzed the statistical property of their proposals of GAN with linear discriminators. In contrast, we investigated the statistical point of view and our proposed local GAN aims to improve the statistical efficiency of the original GAN by constructing a linear discriminator class with the Fisher score as features. Our linear discriminator class is not considered for computational simplicity but is motivated from the expansion of the optimal discriminator \eqref{eq:linear_d} around a local neighborhood of $\theta^*$. We then show the asymptotic efficiency of local GAN through Corollary~\ref{cor:localgan}, which sheds light on the relationship between GAN and the well-established method MLE.}

\section{Empirical analysis}\label{sec:experiment}
In this section, we provide simulation studies on the estimation performance of both the discriminator and the generator, and compare our method with $f$-GANs. The experiments on real data and local GAN are presented in Appendix \ref{app:realdata} and \ref{app:exp_localgan}, respectively, and all implementation details are described in Appendix \ref{app:exp_detail}.

\subsection{Discriminator estimation}\label{sec:exp_disc}

This section illustrates the results in Sections \ref{sec:disc} and \ref{sec:d_eff} regarding the discriminator estimation. 
We consider a classification task of two 2-dimensional Gaussians whose means and variances are different, i.e., $p_1=\cN((0,0),\sigma_1^2\id_2)$ and $p_2=\cN((\mu_2,\mu_2),\sigma_2^2\id_2)$, where $\mu_2\in\bbR$ determines the distance between the two distributions, $\sigma_1^2=0.1$ and $\sigma_2^2=0.05$. We regard $p_1$ as the real data distribution and $p_2$ as the generated distribution. Then the optimal discriminator is given by
\begin{equation*}
	D^*(x)=\ln(p_1(x)/p_2(x))=\ln\frac{\sigma_2^2}{\sigma_1^2}+\frac{\mu_2^2}{\sigma_2^2}-\frac{\mu_2}{\sigma_2^2}(x_1+x_2)+\frac{1}{2}\bigg(\frac{1}{\sigma_2^2}-\frac{1}{\sigma_1^2}\bigg)(x_1^2+x_2^2),
\end{equation*}
where $x=(x_1,x_2)^\top$. 
We assume a quadratic discriminator class $\cD=\{D_\psi(x)=\psi_0+\psi_1x_1+\psi_2x_2+\psi_3x_1^2+\psi_4x_2^2+\psi_5x_1x_2:\psi\in\Psi\}$ where $\Psi$ is a compact subset of $\bbR^6$ containing $\psi^*$ associated with $D^*$. We adopt the discriminator estimation methods in AGE and various $f$-GANs to estimate $\psi^*$. 
Suppose we are given an imbalanced sample $\{X_i\overset{\text{i.i.d.}}{\sim} p_1,i=1,\dots,n,X_i\overset{\text{i.i.d.}}{\sim} p_2,i=n+1,\dots,n+m\}$, where $n=10^4$ and $m=10n$. 
As $\mu_2$ increases, the two distributions become farther away from each other, which corresponds to a higher level of misspecification in a task of generative modeling. 

Table \ref{tab:d_mis} shows the estimation results of different methods as the misspecification level increases, where two metrics are reported: the sum of the empirical variances of each dimension, Var $=\sum_{i=0}^5  \hat{\mathrm{Var}}(\hat\psi_i)$, and the sum of the estimated squared biases of each dimension, Bias$^2=\sum_{i=0}^5 (\hat\bbE\hat\psi_i-\psi^*_i)^2$, where the empirical variances and means are obtained from 500 random repetitions. In Section \ref{sec:disc}, we prove or assume the consistency of $\hat\psi$ to $\psi^*$ using AGE or $f$-GANs, respectively, which is verified in the simulations in that the biases are much smaller than variances. As to statistical efficiency, AGE always has the lowest variance. As $\mu_2$ grows, which corresponds to the case with more severe model misspecification, all methods become less efficient while AGE exhibits even more significant advantages compared with others. 

\begin{table}
\centering
\caption{Results of discriminator estimation for varying distances between the two classes, where the last four columns of each subtable correspond to the $f$-GAN for a particular divergence, e.g., $f$-RKL means $f$-GAN for the reverse KL divergence.}
\label{tab:d_mis}
\subtable[Var]{
\begin{tabular}{cccccc}
\toprule
$\mu_2$ & AGE & $f$-KL & $f$-RKL & $f$-JS & $f$-$H^2$\\\midrule
0		&	0.0400	&	0.1069	&	0.0616	&	0.0482	&	0.0428\\
0.1	&	0.0477	&	0.1785	&	0.0750	&	0.0566	&	0.0509\\
0.2	&	0.0627	&	0.6809	&	0.1248	&	0.0789	&	0.0725\\
0.3	&	0.0958	&	2.5991	&	0.3448	&	0.1378	&	0.1263\\
0.4	&	0.1747	&	9.6381	&	1.4361	&	0.2458	&	0.2908\\
0.5	&	0.3823	&	26.621	&	9.1281	&	0.6790	&	0.7074\\\bottomrule
\end{tabular}}
\subtable[Bias$^2$]{
\begin{tabular}{ccccc}
\toprule
AGE 		&	$f$-KL	&	$f$-RKL	&	$f$-JS	&	$f$-$H^2$ \\\midrule
0.0001	&	0.0028	&	0.0003	&	0.0003	&	0.0002\\
0.0001	&	0.0108	&	0.0004	&	0.0003	&	0.0002\\
0.0002	&	0.1620	&	0.0004	&	0.0003	&	0.0003\\
0.0002	&	1.4570	&	0.0009	&	0.0005	&	0.0003\\
0.0003	&	3.5922	&	0.0196	&	0.0012	&	0.0011\\
0.0004	&	16.366	& 0.6542	&	0.0013	&	0.0023\\\bottomrule
\end{tabular}}
\end{table}

Next, we take the ratio $\lambda=m/n$ into account and study its role in estimation of the discriminator using different approaches. Table \ref{tab:d_ratio} shows the estimation results as $\lambda$ grows with $n=1000$ fixed, where the metrics are computed similarly as in Table \ref{tab:d_mis}. As more and more negative samples from $p_2$ are available, all methods become more statistically efficient. Specifically, compared with $f$-GAN-KL, AGE has much smaller variances when $\lambda$ is small, while as $\lambda$ becomes sufficiently large, their variances become close. However, we notice that even in this simple simulation setting, $\lambda$ needs to be very large for $f$-GAN-KL to perform comparable to AGE, which means high computational complexity. The discriminator losses of $f$-GAN for the other three divergences perform obviously worse than AGE even when $\lambda$ is fairly large. For JS, as mentioned earlier, when $\lambda=1$, $f$-GAN or GAN are identical to AGE, but for general cases with $\lambda>1$, AGE outperforms GAN. All these empirical findings are consistent to the theoretical results in Section \ref{sec:d_eff}. 

\begin{table}
\centering
\caption{Results of discriminator estimation for varying $\lambda$.}
\label{tab:d_ratio}
\subtable[Var]{
\begin{tabular}{@{}cccccc}
\toprule
$\lambda$ & AGE & $f$-KL & $f$-RKL & $f$-JS & $f$-$H^2$\\\midrule
$1$		&	1.8423	&	16.258	&	2.2940	&	1.8423	&	2.3744\\
$10$	&	0.6259	&	3.0741	&	1.3320	&	0.8691	&	0.6984\\
$10^2$&	0.3635	&	0.9131	&	1.3051	&	0.7522	&	0.5391\\
$10^3$&	0.2895	&	0.4636	&	1.2998	&	0.6983	&	0.5227\\
$10^4$&	0.2412	&	0.2857	&	1.2751	&	0.6885	&	0.5127\\
5$\cdot10^4$\hspace{-4pt}&	0.2415	&	0.2517	&	1.2534	&	0.6790	&	0.5132\\\bottomrule
\end{tabular}}
\subtable[Bias$^2$]{
\def\arraystretch{1.07}
\begin{tabular}{ccccc@{}}
\toprule
AGE 		&	$f$-KL	&	$f$-RKL	&	$f$-JS	&	$f$-$H^2$ \\\midrule
0.0052	&	7.1760	&	0.0083	&	0.0052	&	0.0510\\
0.0023	&	0.9532	&	0.0080	&	0.0036	&	0.0053\\
0.0035	&	0.1548	&	0.0030	&	0.0029	&	0.0021\\
0.0010	&	0.0166	&	0.0026	&	0.0017	&	0.0018\\
0.0007	&	0.0031	&	0.0014	&	0.0012	&	0.0015\\
0.0008	&	0.0037	&	0.0017	&	0.0008	&	0.0011\\\bottomrule
\end{tabular}}
\end{table}

\subsection{Generator estimation}\label{sec:exp_g}
This section illustrates the results in Sections \ref{sec:consis}, \ref{sec:generator} and \ref{sec:g_eff} regarding the generator estimation. 
We begin with a one-dimensional Laplace distribution $p_*(x)=e^{|x|/b}/2b$ with $b=1.5$, like in \cite{biau2020}. We learn the scale parameter $b$ using a misspecified Gaussian distribution family through a generator $G_\theta(Z)=\theta Z$ where $\theta\in[0.1,10^3]$ and $Z\sim\cN(0,1)$. Then the generated distribution $p_\theta$ is $\cN(0,\theta^2)$. Hence, for a generator $G_\theta$, the optimal discriminator is given by 
\begin{equation*}
	D^*_\theta(x)=\ln(p_*(x)/p_\theta(x))=\ln\bigg(\frac{\sqrt{2\pi}\theta}{2b}\bigg)-\frac{|x|}{b}+\frac{x^2}{2\theta^2},
\end{equation*}
which motivates the construction of the discriminator class $$\cD=\{D_\psi(x):\psi_0+\psi_1|x|+\psi_2 x^2,\psi=(\psi_0,\psi_1,\psi_2)^\top\in\Psi\},$$ where $\Psi$ is a compact subset of $\bbR^3$ containing the optimal discriminator class $\{\psi^*\in\bbR^3:D_{\psi^*}=D^*_\theta,\theta\in[0.1,10^3]\}$. We call this setting Laplace-Gaussian for short.

In the second setting, we consider a one-dimension Gaussian distribution with non-zero mean $p_*(x)=\cN(\mu_0,\sigma^2)$ with $\mu_0=1$ and $\sigma=1$. Again, we learn the scale parameter $\sigma$ using a Gaussian distribution family with a misspecified mean $p_\theta=\cN(0,\theta^2)$ through a generator $G_\theta(Z)=\theta Z$ where $\theta\in[0.1,10^3]$ and $Z\sim\cN(0,1)$. Then the optimal discriminator is
\begin{equation*}
	D^*_\theta(x)=\ln\bigg(\frac{\theta}{\sigma_0}\bigg)-\frac{\mu_0^2}{2\sigma_0^2}+\frac{\mu_0}{\sigma_0^2}x+\frac{1}{2}\left(\frac{1}{\theta^2}-\frac{1}{\sigma_0^2}\right)x^2,
\end{equation*}
which motivates the construction of the discriminator class $$\cD=\{D_\psi(x):\psi_0+\psi_1x+\psi_2 x^2,\psi=(\psi_0,\psi_1,\psi_2)^\top\in\Psi\}$$ with $\Psi$ being a compact subset of $\bbR^3$ containing the optimal discriminator class $\{\psi^*\in\bbR^3:D_{\psi^*}=D^*_\theta,\theta\in[0.1,10^3]\}$. This setting is called Gaussian2.

\begin{figure}[b]
\centering
\subfigure[Laplace-Gaussian]{
\includegraphics[width=0.48\linewidth]{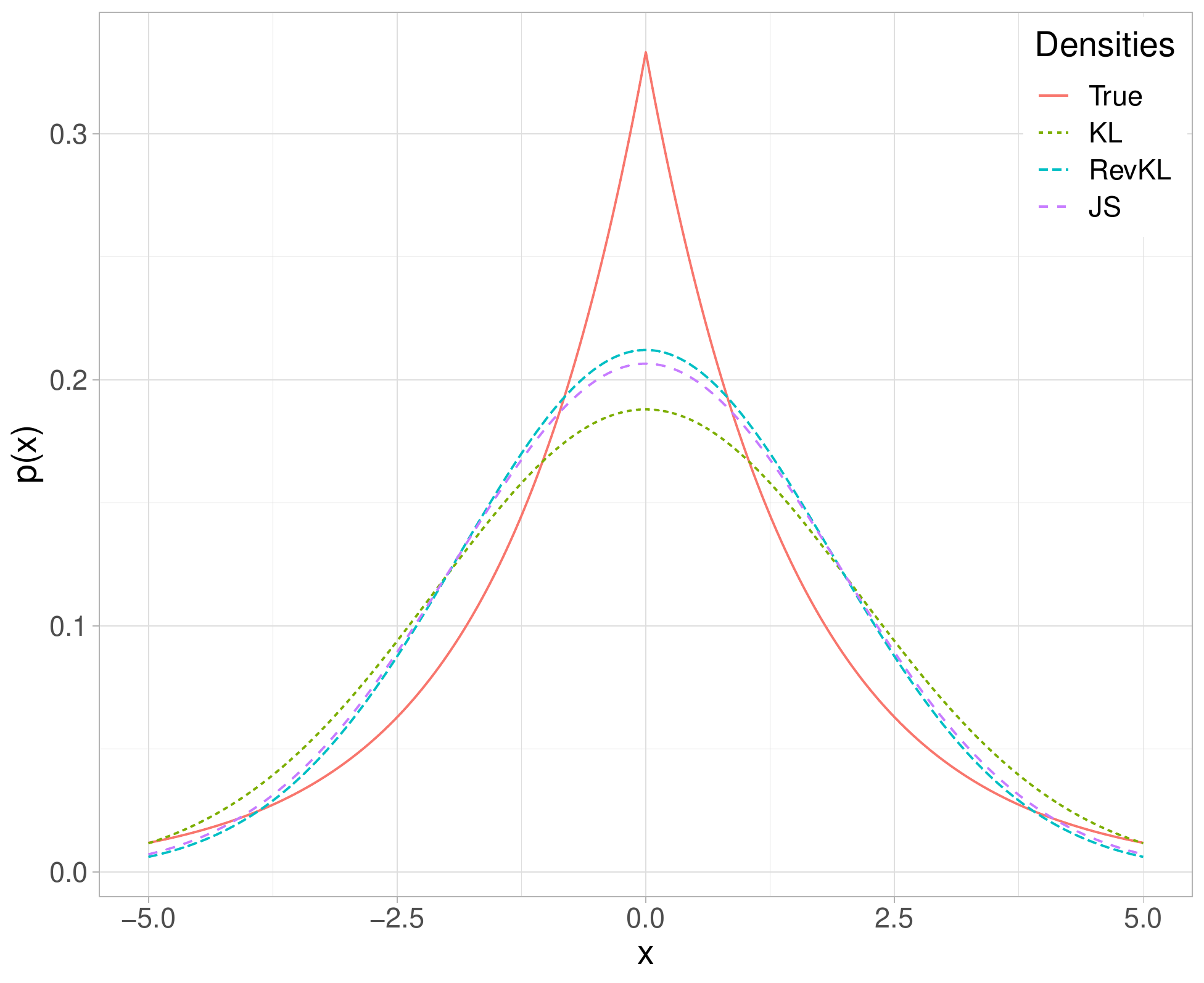}}
\subfigure[Gaussian2]{
\includegraphics[width=0.48\linewidth]{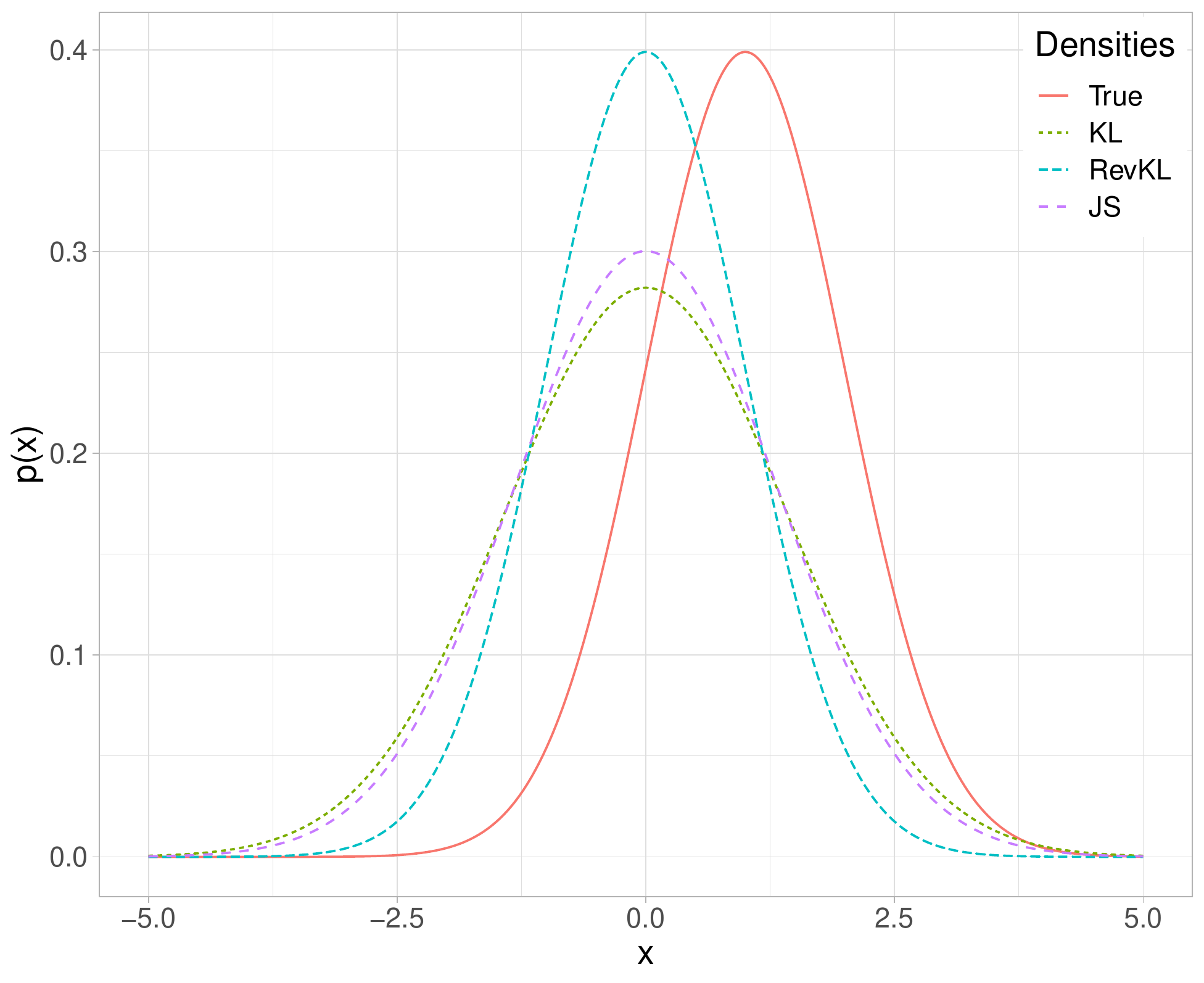}}
\caption{True $p_*$ and optimal generated densities $p_{\theta^*}$ with various divergences as the objective. $\theta^*$ for KL and reverse KL can be obtained analytically, while $\theta^*$ for JS is estimated using the average over 500 random repetitions of the setup $n=1000,\lambda=1000$ with an optimal discriminator $D^*$ in AGE algorithm. }
\label{fig:density}
\end{figure}

We consider KL, reverse KL and JS divergences as the objectives respectively, and adopt AGE and $f$-GAN with $n=100$ or $1000$ and $\lambda$ varying from 1 to $1000$. Figure \ref{fig:density} plots the true densities $p_*$ and the optimal generated densities $p_{\theta^*}$ under the three divergences. 
We see that with misspecified models, the KL objective tends to overestimate the variance while the reverse KL objective tends to underestimate the variance, which is consistent with common statistical knowledge. The JS divergence behaves in between of KL and reverse KL. As mentioned in Section \ref{sec:intro}, different applications may favor different divergences as the objective function, so it is worth discussing the statistical properties of generative modeling methods under various divergences.

Tables \ref{tab:lap_gaus_kl}-\ref{tab:gaus2_js} report the results of generator estimation in the above two settings with the three divergences as the objective. Here we present the results of the original AGE and $f$-GAN algorithms with the one-sample, while the results of the two-sample scheme introduced at the end of Section~\ref{sec:g_eff} are deferred to Appendix~\ref{app:exp_twosample}. As in the previous section, we report two metrics, the empirical variance Var $=\hat{\mathrm{Var}}(\hat\theta)$ and squared bias Bias$^2=(\hat\bbE\hat\theta-\theta^*)^2$, both of which are obtained from 500 random repetitions. 

In general, we observe that AGE achieves lower variances than $f$-GANs, especially when the sample size $n$ or the ratio $\lambda$ is small. As $n$ or $\lambda$ increases, the variances decrease. 
 For all methods, the biases are significantly small compared with the variances, which supports the consistency results in Section \ref{sec:consis}, so we regard the variance as the measure of the estimation error. Specifically, for KL, as shown in Tables \ref{tab:lap_gaus_kl} and \ref{tab:gaus2_kl}, $f$-GAN performs extremely poor with small $n$ or $\lambda$, while becomes more comparable to AGE when $\lambda$ is fairly large, e.g., $\lambda=1000$. 
For reverse KL, as shown in Tables \ref{tab:lap_gaus_revkl} and \ref{tab:gaus2_revkl}, $f$-GAN always exhibits a gap from AGE regardless of $\lambda$. For JS, since $f$-GAN (or GAN) algorithm is identical to AGE when $\lambda=1$ and close to AGE when $\lambda$ is small, the gap between $f$-GAN and AGE is not as large as that for the other two divergences. However, as shown in Tables \ref{tab:lap_gaus_js} and \ref{tab:gaus2_js}, AGE still exhibits an advantage over $f$-GAN. These empirical findings are consistent with the theoretical results in Section \ref{sec:g_eff} and the simulations in the previous section.

\begin{table}[b]
\centering
\caption{Results of generator estimation under KL objective (Laplace-Gaussian).}
\label{tab:lap_gaus_kl}
\subtable[Var]{
\begin{tabular}{cccc}
\toprule
$n$						&	$\lambda$	&	AGE	&	$f$-GAN \\\midrule
\multirow{4}{*}{$100$}	&	1		&	0.0882	&	380.84	\\
										&	10		&	0.0614	&	188.42	\\
										&	100	&	0.0583	&	15.302	\\
										&	1000	&	0.0585	&	0.0693	\\\midrule
\multirow{4}{*}{$1000$}	&	1		&	0.0135	&	23.585	\\
										&	10		&	0.0089	&	0.1132	\\
										&	100	&	0.0067	&	0.0084 \\
										&	1000	&	0.0055	&	0.0056	\\\bottomrule
\end{tabular}}
\subtable[Bias$^2\times10^2$]{
\begin{tabular}{cc}
\toprule
AGE	&	$f$-GAN \\\midrule
1.0341	&	33107	\\
0.3894	&	3637.8	\\
0.0562	&	50.455	\\
0.0105	&	0.0162	\\\midrule
0.4153	&	466.17	\\
0.1159	&	0.1396	\\
0.0606	&	0.0717 \\
0.0010	&	0.0218	\\\bottomrule
\end{tabular}}
\end{table}

\begin{table}
\centering
\caption{Results of generator estimation under reverse KL objective (Laplace-Gaussian).}
\label{tab:lap_gaus_revkl}
\subtable[Var]{
\begin{tabular}{cccc}
\toprule
$n$						&	$\lambda$	&	AGE	&	$f$-GAN \\\midrule
\multirow{4}{*}{$100$}	&	1		&	0.0602	&	3.7240	\\
										&	10		&	0.0442	&	3.7551	\\
										&	100	&	0.0400	&	2.6711	\\
										&	1000	&	0.0379	&	2.4325	\\\midrule
\multirow{4}{*}{$1000$}	&	1		&	0.0066	&	1.4725	\\
										&	10		&	0.0049	&	0.7929	\\
										&	100	&	0.0044	&	0.0059 \\
										&	1000	&	0.0041	&	0.0054	\\\bottomrule
\end{tabular}}
\subtable[Bias$^2\times10^3$]{
\begin{tabular}{cc}
\toprule
AGE	&	$f$-GAN \\\midrule
1.6465	&	2928.3	\\
2.0814	&	2597.8	\\
1.0642	&	885.97	\\
0.8081	&	658.05	\\\midrule
0.0856	&	2471.9	\\
0.0104	&	51.859	\\
0.0113	&	0.0368 \\
0.0108	&	0.0336	\\\bottomrule
\end{tabular}}
\end{table}

\begin{table}
\centering
\caption{Results of generator estimation under JS objective (Laplace-Gaussian).}
\label{tab:lap_gaus_js}
\subtable[Var]{
\begin{tabular}{cccc}
\toprule
$n$						&	$\lambda$	&	AGE	&	$f$-GAN \\\midrule
\multirow{4}{*}{$100$}	&	1		&	0.0763	&	0.0763	\\
										&	10		&	0.0523	&	0.0603	\\
										&	100	&	0.0426	&	0.0513	\\
										&	1000	&	0.0422	&	0.0498	\\\midrule
\multirow{4}{*}{$1000$}	&	1		&	0.0072	&	0.0072	\\
										&	10		&	0.0043	&	0.0049	\\
										&	100	&	0.0041	&	0.0048 \\
										&	1000	&	0.0041	&	0.0050	\\\bottomrule
\end{tabular}}
\subtable[Bias$^2\times10^4$]{
\begin{tabular}{cc}
\toprule
AGE	&	$f$-GAN \\\midrule
1.5445	&	1.5445	\\
0.6202	&	0.4277	\\
0.2038	&	0.8254	\\
0.3865	&	0.6833	\\\midrule
0.0787	&	0.0787	\\
0.0998	&	0.1165	\\
0.0662	&	0.0721 \\
0.0459	&	0.0484	\\\bottomrule
\end{tabular}}
\end{table}

\begin{table}
\centering
\caption{Results of generator estimation under KL objective (Gaussian2).}
\label{tab:gaus2_kl}
\subtable[Var]{
\begin{tabular}{cccc}
\toprule
$n$						&	$\lambda$	&	AGE	&	$f$-GAN \\\midrule
\multirow{4}{*}{$100$}	&	1		&	0.0290	&	100.74	\\
										&	10		&	0.0095	&	54.782	\\
										&	100	&	0.0076	&	14.082	\\
										&	1000	&	0.0068	&	2.2878	\\\midrule
\multirow{4}{*}{$1000$}	&	1		&	0.0025	&	18.187	\\
										&	10		&	0.0009	&	1.6849	\\
										&	100	&	0.0006	&	0.2074 \\
										&	1000	&	0.0006	&	0.0053	\\\bottomrule
\end{tabular}}
\subtable[Bias$^2\times10^3$]{
\begin{tabular}{cc}
\toprule
AGE	&	$f$-GAN \\\midrule
0.3695	&	15330	\\
0.1626	&	9730.5	\\
0.0374	&	5871.7	\\
0.0386	&	401.66	\\\midrule
0.0372	&	4749.9	\\
0.0116	&	485.38	\\
0.0015	&	22.795 \\
0.0029	&	0.0265	\\\bottomrule
\end{tabular}}
\end{table}

\begin{table}
\centering
\caption{Results of generator estimation under reverse KL objective (Gaussian2).}
\label{tab:gaus2_revkl}
\subtable[Var]{
\begin{tabular}{cccc}
\toprule
$n$						&	$\lambda$	&	AGE	&	$f$-GAN \\\midrule
\multirow{4}{*}{$100$}	&	1		&	0.0183	&	0.2547	\\
										&	10		&	0.0060	&	0.1092	\\
										&	100	&	0.0050	&	0.0592	\\
										&	1000	&	0.0047	&	0.0532	\\\midrule
\multirow{4}{*}{$1000$}	&	1		&	0.0020	&	0.0048	\\
										&	10		&	0.0009	&	0.0040	\\
										&	100	&	0.0005	&	0.0038 \\
										&	1000	&	0.0005	&	0.0034	\\\bottomrule
\end{tabular}}
\subtable[Bias$^2\times10^4$]{
\begin{tabular}{cc}
\toprule
AGE	&	$f$-GAN \\\midrule
9.2068	&	139.20	\\
4.9493	&	45.554	\\
1.1988	&	10.720	\\
2.1428	&	7.5374	\\\midrule
0.0329	&	1.0071	\\
0.0173	&	0.4603	\\
0.0046	&	0.4937	\\
0.0018	&	0.2115	\\\bottomrule
\end{tabular}}
\end{table}

\begin{table}
\centering
\caption{Results of generator estimation under JS objective (Gaussian2).}
\label{tab:gaus2_js}
\subtable[Var$\times10^2$]{
\begin{tabular}{cccc}
\toprule
$n$						&	$\lambda$	&	AGE	&	$f$-GAN \\\midrule
\multirow{4}{*}{$100$}	&	1		&	2.6237	&	2.6237	\\
										&	10		&	0.8137	&	0.8387	\\
										&	100	&	0.7033	&	0.7509	\\
										&	1000	&	0.6964	&	0.7401	\\\midrule
\multirow{4}{*}{$1000$}	&	1		&	0.2402	&	0.2402	\\
										&	10		&	0.0905	&	0.0956	\\
										&	100	&	0.0899	&	0.0935 \\
										&	1000	&	0.0711	&	0.0754	\\\bottomrule
\end{tabular}}
\subtable[Bias$^2\times10^5$]{
\begin{tabular}{cc}
\toprule
AGE	&	$f$-GAN \\\midrule
9.0256	&	9.0256	\\
0.6790	&	0.4277	\\
1.0043	&	0.8254	\\
0.1153	&	0.6833	\\\midrule
0.1408	&	0.1408	\\
0.1181	&	0.2195	\\
0.0261	&	0.0221 \\
0.0493	&	0.0564	\\\bottomrule
\end{tabular}}
\end{table}

\section{Conclusion}\label{sec:conclu}

This paper systematically studied the asymptotic properties of $f$-divergence GANs and investigated the statistical consequences of the analysis, thus contributing to a better understanding of GANs. 
We showed that with correctly specified generative models, various $f$-divergence GANs are asymptotically equivalent under suitable regularity conditions. Moreover, with an appropriately chosen local discriminator, they become asymptotically equivalent to the maximum likelihood estimate. 
Under model misspecification, our analysis showed the lack of statistical efficiency of the original $f$-GAN approach and hence led to an improved method AGE that can achieve a lower asymptotic variance. 
We provided empirical studies that support the theory and demonstrate that our AGE outperforms $f$-GAN under various settings.






{
\nocite{*}
\bibliography{main.bib}
\bibliographystyle{ieeetr}}

\clearpage
\begin{appendices}

\section{Preliminaries}

This section presents some preliminary notions and lemmas which will be used in proofs.

\begin{definition}[Bracketing covering number \cite{geer2000empirical}]\label{def:cover_num}
	Consider a function class $\cG=\{g(x)\}$ and a probability measure $\mu$ defined on $\cX$. Given any positive number $\delta>0$. Let $N_{1,B}(\delta,\cG,\mu)$ be the smallest value of $N$ for which there exist pairs of functions $\{[g_j^L,g_j^U]\}_{j=1}^N$ such that $\int |g_j^L(x)-g_j^U(x)|d\mu\leq\delta$ for all $j=1,\dots,N$, and such that for each $g\in\cG$, there is a $j=j(g)\in\{1,\dots,N\}$ such that $g_j^L\leq g\leq g_j^U$. Then $N_{1,B}(\delta,\cG,\mu)$ is called the $\delta$-bracketing covering number of $\cG$.
\end{definition}

\begin{definition}[Stochastic uniform equicontinuity]\label{def:equicont}
	A sequence of random functions $f_n(\theta)$ is stochastic uniform equicontinuous if for all $\epsilon>0$,
	\begin{equation*}
		\lim_{\delta\to0}\limsup_{n\to\infty}\bbP\left(\sup_{\|\theta-\theta'\|<\delta}\|f_n(\theta)-f_n(\theta')\|>\epsilon\right)=0.
	\end{equation*}
\end{definition}

\begin{lemma}\label{lem:conv_density}
	Let $\mu_n$ and $\mu$ be a sequence of measures on probability space $(\cX,\Sigma)$ with densities $p_n(x)$ and $p(x)$. Given any compact subset $K$ of $\cX$. Suppose $p_n$ is uniformly bounded and Lipschitz on $K$ ($*$). If $H^2(\mu_n,\mu)\pto0$, then $\sup_{x\in K}|p_n(x)-p(x)|\pto0$ as $n\to\infty$, where $H(q_1,q_2)=\left(\int\big(q_1^{1/2}-q_2^{1/2}\big)^2dxdz/2\right)^{1/2}$ denotes the Hellinger distance between two distributions with densities $q_1$ and $q_2$.
\end{lemma}
\begin{proof}
Note that assumptions in ($*$) satisfy the requirements in the Arzel\`a-Ascoli theorem. Thus, for each subsequence of $p_n$, there is a further subsequence $p_{n_m}$ which converges uniformly on compact set $K$, i.e., for some $p_0$ as $m\to\infty$ we have
\begin{equation*}
	\sup_{x\in K}|p_{n_m}(x)-p_0(x)|\to0.
\end{equation*}

By Scheff\'e's Theorem we have $H(p_{n_m},p_0)\to0$. On the other hand we have $H(p_{n_m},p)\pto0$. By triangle inequality,
\begin{equation*}
	H(p,p_0)\leq H(p_{n_m},p_0)+H(p_{n_m},p)\pto0.
\end{equation*}
Since the inequality holds for all $m$ and the LHS is deterministic, we have $H(p,p_0)=0$, which implies $p=p_0$, a.e. wrt the Lebesgue measure. Hence we have
\begin{equation*}
	\sup_{x\in K}|p_{n_m}(x)-p(x)|\to0,\ a.e.
\end{equation*}
Then by \cite[Theorem 2.3.2]{durrett2019probability}, we have $\sup_{x\in K}|p_{n}(x)-p(x)|\pto0$ as $n\to\infty$.
\end{proof}

\begin{lemma}\label{lem:point_unif_conv}
Consider a compact set $\Theta$, a sequence of random functions $f_n(\theta),n=1,2,\dots$, and a deterministic function $f_0(\theta)$ with $\theta\in\Theta$. Suppose $f_0(\theta)$ is Lipschitz continuous with respect to $\theta$ and the sequence $f_n(\theta),n=1,2,\dots$ is stochastic uniformly equicontinuous, as defined in Definition~\ref{def:equicont}. If for each $\theta\in\Theta$, we have $f_n(\theta)\pto f_0(\theta)$ as $n\to\infty$, then we have $\sup_{\theta\in\Theta}\|f_n(\theta)-f_0(\theta)\|\pto0$ as $n\to\infty$.	
\end{lemma}
\begin{proof}


By Definition~\ref{def:equicont}, the stochastic uniform equicontinuity of sequence $f_n(\theta)$ indicates that for all $\epsilon>0$, we have
\begin{align*}
	\lim_{\delta\to0}\limsup_{n\to\infty}\bbP\left(\sup_{\|\theta-\theta'\|<\delta}\|f_n(\theta)-f_n(\theta')\|>\epsilon\right) =0.
\end{align*}
Then by the Lipschitz continuity of $f_0(\theta)$, we have
\begin{align*}
	&\lim_{\delta\to0}\limsup_{n\to\infty}\bbP\left(\sup_{\|\theta-\theta'\|<\delta}\|f_n(\theta)-f_0(\theta)-(f_n(\theta')-f_0(\theta'))\|>\epsilon\right)\\
	\leq & \lim_{\delta\to0}\limsup_{n\to\infty}\bbP\left(\sup_{\|\theta-\theta'\|<\delta}\|f_n(\theta)-f_n(\theta')\|>\epsilon\right) + \lim_{\delta\to0}\bbP\left(\sup_{\|\theta-\theta'\|<\delta}\|f_0(\theta)-f_0(\theta')\|>\epsilon\right)\\
	\leq & 0 + \lim_{\delta\to0}\bbP(\ell\delta>\epsilon) =0,
\end{align*}
which indicates the stochastic uniform equicontinuity of $\{f_n(\theta)-f_0(\theta)\}$.

For simplicity, denote $f'_n(\theta)=f_n(\theta)-f_0(\theta)$. 
Given any $\epsilon>0$ and $\delta>0$. Since $\Theta$ is compact, it can be partitioned into a finite number $N_\delta$ of balls with radius smaller than $\delta$, i.e., $O_j=\{\theta:\|\theta-\theta_j\|<\delta\}$ with $\theta_j$ being the center of the $j$-th ball, $j=1,\dots,N_\delta$.  Then we have
\begin{align*}
	\bbP\left(\sup_{\theta\in\Theta}\|f'_n(\theta)\|>\epsilon\right) &\leq \bbP\left(\max_{1\leq j\leq N_\delta}\sup_{\theta\in O_j}\left[\|f'_n(\theta)-f'_n(\theta_j)\|+\|f'_n(\theta_j)\|\right]>\epsilon\right)\\
	&\leq \bbP\left(\max_{1\leq j\leq N_\delta}\sup_{\theta\in O_j}\|f'_n(\theta)-f'_n(\theta_j)\|>\frac{\epsilon}{2}\right) + \bbP\left(\max_{1\leq j\leq N_\delta}\|f'_n(\theta_j)\|>\frac{\epsilon}{2}\right)\\
	&\leq \bbP\left(\sup_{\|\theta-\theta'\|<\delta}\|f'_n(\theta)-f'_n(\theta')\|>\frac{\epsilon}{2}\right) + \sum_{j=1}^{N_\delta}\bbP\left(\|f'_n(\theta_j)\|>\frac{\epsilon}{2}\right),
\end{align*}
where the third inequality follows from the union bound.
We take $n\to\infty$ to get rid of the second term by pointwise convergence, and take $\delta\to0$ to get rid of the first term by the stochastic uniform equicontinuity, which leads to the desired result.
\end{proof}

\begin{lemma}[Uniform continuous mapping theorem]\label{lem:ucmt}
	Let $X_n$, $X$ be random vectors defined on $\cX$. Let $f:\bbR^d\to\bbR^m$ be uniformly continuous and $T_\theta:\cX\to\bbR^d$ for $\theta\in\Theta$. Suppose $T_\theta(X_n)$ converges uniformly in probability to $T_\theta(X)$ over $\Theta$, i.e., as $n\to\infty$, we have  $\sup_{\theta\in\Theta}\|T_\theta(X_n)-T_\theta(X)\|\pto0$. Then $f(T_\theta(X_n))$ converges uniformly in probability to $f(T_\theta(X))$, i.e., as $n\to\infty$, $\sup_{\theta}\|f(T_\theta(X_n))-f(T_\theta(X))\|\pto0$.
\end{lemma}
\begin{proof}
Given any $\epsilon>0$. 
Because $f$ is uniformly continuous, there exists $\delta>0$ such that $\|f(x)-f(y)\|\leq\epsilon$ for all $\|x-y\|\leq\delta$.

We have
\begin{align}
\bbP\Big(\sup_{\theta\in\Theta}\|T_\theta(X_n)-T_\theta(X)\|\leq\delta\Big)&=\bbP\big(\forall\theta\in\Theta:\|T_\theta(X_n)-T_\theta(X)\|\leq\delta\big)\label{eq:ucmt1}\\
&\leq\bbP\big(\forall\theta\in\Theta:\|f(T_\theta(X_n))-f(T_\theta(X))\|\leq\epsilon\big)\nonumber\\
&=\bbP\Big(\sup_{\theta\in\Theta}\|f(T_\theta(X_n))-f(T_\theta(X))\|\leq\epsilon\Big)\label{eq:ucmt2}.
\end{align}
By the uniform convergence of $T_\theta(X_n)$, we know the left-hand side of (\ref{eq:ucmt1}) converges to 1.
Hence (\ref{eq:ucmt2}) goes to 1, which implies the desired result.
\end{proof}

\begin{lemma}\label{lem:unif_order}
	Let $X_n(\theta)$ be a sequence of random vectors depending on parameter $\theta$ in a compact parameter space $\Theta$. Let $\mu(\theta)=\bbE[X_n(\theta)]$ be their common mean vector. Suppose for all $\theta\in\Theta$ we have as $n\to\infty$, $\sqrt{n}(X_n(\theta)-\mu(\theta))\dto\cN(0,\Sigma(\theta))$ with the asymptotic variance matrix $\Sigma(\theta)\succ0$ and continuous with respect to $\theta$. Then we have $X_n(\theta)-\mu(\theta)=O_p(1/\sqrt{n})$ uniformly for all $\theta$, i.e., $\sup_{\theta\in\Theta}\|X_n(\theta)-\mu(\theta)\|=O_p(1/\sqrt{n})$.
\end{lemma}
\begin{proof}
Let $X^*_n(\theta)=[\Sigma(\theta)]^{-1/2}X_n(\theta)$. Then $\sqrt{n}X^*_n(\theta)\dto\cN(0,\id)$ as $n\to\infty$. Since $\Sigma(\theta)$ is continuous in $\theta$ and $\Theta$ is compact, $\Sigma(\theta)$ is uniformly bounded. Thus there exists $M>0$ such that for all $\theta\in\Theta$, $\|[\Sigma(\theta)]^{1/2}\|\leq M$. Then
\begin{equation*}
	\sqrt{n}\|X_n(\theta)-\mu(\theta)\|=\sqrt{n}\big\|[\Sigma(\theta)]^{1/2}X^*_n(\theta)\big\|\leq \sqrt{n} M \|X^*_n(\theta)\|=O_p(1).
\end{equation*}
Hence $\sup_{\theta\in\Theta}\|X_n(\theta)-\mu(\theta)\|=O_p(1/\sqrt{n})$.
\end{proof}

\section{Proofs in Section~\ref{sec:method}}\label{app:pf_grad}

The proof technique is inspired by that of CFG-GAN \cite{Johnson2019AFO}.
Given a differentiable vector function $g(x):\mathbb{R}^k\to\mathbb{R}^k$, we use $\nabla\cdot g(x)$ to denote its divergence, defined as 
$$\nabla\cdot g(x):=\sum_{j=1}^k\frac{\partial[g(x)]_j}{\partial[x]_j},$$ 
where $[x]_j$ denotes the $j$-th component of $x$. We know that $\int\nabla\cdot g(x)dx=0$ for all vector function $g(x)$ such that $g(\infty) = 0$. Given a matrix function $w(x)=(w_1(x),\dots,w_l(x)):\mathbb{R}^k\to\mathbb{R}^{k\times l}$ where each $w_i(x)$, for $i=1\dots,l$, is a $k$-dimensional differentiable vector function, its divergence is defined as $\nabla\cdot w(x)=(\nabla\cdot w_1(x),\dots,\nabla\cdot w_l(x))^\top$.

To prove Theorem~\ref{thm:grad}, we need the following lemma.
\begin{lemma}\label{lem:pf_grad}
Using the definitions in Theorem~\ref{thm:grad}, we have 
\begin{equation}
\nabla_\theta p_\theta(x) = -g_\theta(x)^\top \nabla_x p_\theta(x) - p_\theta(x)\nabla\cdot g_\theta(x),\label{eq:grad_pgen}
\end{equation}
for all $x\in\cX$, where $g_\theta(G_\theta(z))=\nabla_\theta G_\theta(z)$.
\end{lemma}

\begin{proof}[Proof of Lemma~\ref{lem:pf_grad}]
Let $d_\theta$ be the dimension of parameter $\theta$. 
To simplify the notation, let $X=G_\theta(Z)\in\cX$ and $p$ be the probability density of $X$. For each $i=1,\dots,d_\theta$, let $\Delta=\delta e_i$ where $e_i$ is a $d_\theta$-dimensional unit vector whose $i$-th component is one and all the others are zero, and $\delta$ is a small scalar. Let $X'=G_{\theta+\Delta}(Z)$ and $\delta$ be such that $X'$ is a random variable transformed from $X$ by $X'=X+ g(X) \Delta + o(\delta)$ where $g(X)\in\mathbb{R}^{d\times l}$. Let $p'$ be the probability density of $X'$. For an arbitrary $x'\in\cX$, let $x'=x+g(x)\Delta+o(\delta)$. Then we have
\begin{align}
p'(x')&=p(x)|\det(dx'/dx)|^{-1}\nonumber\\
&=p(x)|\det(\id_d+\nabla g(x)\Delta+o(\delta))|^{-1}\nonumber\\
&=p(x)(1+\Delta^\top\nabla\cdot g(x)+o(\delta))^{-1}\label{eq:l0}\\ 
&=p(x)(1-\Delta^\top\nabla\cdot g(x)+o(\delta))\label{eq:l1}\\
&=p(x)-\Delta^\top p(x')\nabla\cdot g(x') +o(\delta)\label{eq:l2}\\
&=p(x')-\Delta^\top g(x')^\top\cdot\nabla p(x') - \Delta^\top p(x')\nabla\cdot g(x') +o(\delta).\label{eq:l3}
\end{align}
The first two equalities use the multivariate change of variables formula for probability densities. (\ref{eq:l0}) uses the definition of determinant with terms explicitly expanded up to $O(\delta)$. (\ref{eq:l1}) uses the Taylor expansion of $(1+\gamma)^{-1}=1-\gamma+o(\gamma)$ with $\gamma=\Delta^\top\nabla\cdot g(x)$. (\ref{eq:l2}) follows from $p(x')=p(x)+o(1)$ and $\nabla\cdot g(x')=\nabla\cdot g(x)+o(1)$. (\ref{eq:l3}) is due to $p(x)=p(x')-(x'-x)^\top\nabla p(x')+o(\delta)$. 
Since $x'\in\cX$ is arbitrary, above implies that
\begin{eqnarray*}
p'(x)=p(x)-\Delta^\top g(x)^\top\nabla p(x)-\Delta^\top p(x)\nabla\cdot g(x) + o(\delta)
\end{eqnarray*}
for all $x\in\mathbb{R}^d$ and $i=1,\dots,d_\theta$, which leads to (\ref{eq:grad_pgen}) by taking $\delta\to0$, setting $g(x)=g_\theta(x)$, and noting that $p=p_\theta$ as both are the density of $G_\theta(Z)$ and $p'=p_{\theta+\Delta}$ as both are the density of $G_{\theta+\Delta}(Z)$. 
\end{proof}

\begin{proof}[Proof of Theorem~\ref{thm:grad}]
Rewrite the objective (\ref{eq:obj}) as $L(\theta)=\int\ell(p_*(x),p_\theta(x))dx$ where $\ell$ denotes the integrands in definition (\ref{eq:f_div}).
Let $\ell'_2(p_*,p_\theta)=\partial\ell(p_*,p_\theta)/\partial p_\theta$.
Using the chain rule and Lemma~\ref{lem:pf_grad}, we have 
\begin{align}
\nabla_\theta\ell(p_*(x),p_\theta(x))&=\ell'_2(p_*(x),p_\theta(x))\nabla_\theta p_\theta(x)\nonumber\\
&= \ell'_2(p_*(x),p_\theta(x))\big[-g_\theta(x)^\top\nabla_x p_\theta(x) - p_\theta(x)\nabla\cdot g_\theta(x)\big]\nonumber\\
&= p_\theta(x)g_\theta(x)^\top \nabla_x \ell'_2(p_*(x),p_\theta(x)) - \nabla_x\cdot\big[\ell'_2(p_*(x),p_\theta(x))p_\theta(x)g_\theta(x)\big],\label{eq:grad_int}
\end{align}
where the third equality is obtained by applying the product rule as follows
\begin{align*}
\nabla_x\cdot\left[\ell'_2(p_*(x),p_\theta(x))p_\theta(x)g_\theta(x)\right]&=\ell'_2(p_*(x),p_\theta(x))p_\theta(x)\nabla\cdot g_\theta(x) \\
&\ + \ell'_2(p_*(x),p_\theta(x))g_\theta(x)^\top \nabla_x p_\theta(x) \\
&\ + p_\theta(x)g_\theta(x)^\top \nabla_x \ell'_2(p_*(x),p_\theta(x)).
\end{align*}
By integrating (\ref{eq:grad_int}) over $x$ and by using the fact that $\int\nabla\cdot f(x)dx=\mathbf{0}$ with $f(x)=\ell'_2(p_*(x),p_\theta(x))p_\theta(x)g_\theta(x)$, we have
\begin{eqnarray*}
\nabla_\theta L(\theta) = \int\nabla_\theta\ell(p_*(x),p_\theta(x))dx=\int p_\theta(x) g_\theta(x)^\top \nabla_x \ell'_2(p_*(x),p_\theta(x))dx.
\end{eqnarray*}
According to the definition (\ref{eq:f_div}) of $f$-divergences and by noting the fact that $r(x)=e^{D^*(x)}$, we have  
\begin{eqnarray}\label{eq:lll}
\nabla_x \ell'_2(p_*(x),p_\theta(x))=f''\bigg(\frac{1}{r(x)}\bigg)\nabla_x \frac{1}{r(x)}=f''\bigg(\frac{1}{r(x)}\bigg)\frac{1}{r(x)}\nabla_x D^*(x).
\end{eqnarray}
Further by reparametrization, we obtain
\begin{align*}
\nabla_\theta L(\theta) &= -\mathbb{E}_{X\sim p_\theta(x)}\big[s^*(X)g_\theta(X)^\top \nabla_x D^*(x)\big]\\
&=-\mathbb{E}_{Z\sim p_z(z)}\big[s^*(G_\theta(Z))\nabla_\theta G_\theta(Z)^\top \nabla_x D^*(G_\theta(Z))\big],
\end{align*}
which completes the proof. 
\end{proof}

\section{Proofs in Section~\ref{sec:theory}}

To lighten the notation, throughout this section, we denote the AGE estimator $\gage$ by $\hat\theta$.

\subsection{Proof of Theorem~\ref{thm:grad_conv}}\label{app:pf_grad_conv}

We first show the stochastic uniform equicontinuity of $\hat{D}_\theta(x)$ in the following lemma.

\begin{lemma}\label{lem:D_cont}
	Let $$\hat{D}_\theta=\argmin_{D\in\cD}\hat{L}_d(D,\theta)=\argmin_{D\in\cD}\left\{\frac{1}{n}\sum_{i=1}^{n}\ln(1+e^{-D(x_i)}\lambda)+\frac{\lambda}{m}\sum_{i=1}^{m}\ln(1+e^{D(G_\theta(z_i))}/\lambda)\right\}.$$  
	Let $K$ be any compact subset of $\cX$. Then under the assumptions in Theorem~\ref{thm:grad_conv}, $\hat{D}_\theta(x)$ is stochastic uniformly equicontinuous (Definition~\ref{def:equicont}) with respect to $(\theta,x)$ over $\Theta\times K$.
\end{lemma}
\begin{proof}[Proof of Lemma~\ref{lem:D_cont}]
Without loss of generality, we assume $\lambda=1$ in this proof and the case with $\lambda>1$ can be similarly derived.
Note that function $f(x)=\ln(1+e^x)$ is convex with $|f''(x)|<1$. Hence $f(x)$ is 1-Lipschitz with respect to $x$, i.e., for all $x,y$, we have $|\ln(1+e^x)-\ln(1+e^y)|\leq |x-y|$. This implies for all $D\in\cD$ and $\theta,\theta'\in\Theta$,
\begin{align*}
	|\hat{L}_d(D,\theta)-\hat{L}_d(D,\theta')|&\leq\frac{1}{m}\sum_{i=1}^{m}\big|\ln(1+e^{D(G_\theta(z_i))})-\ln(1+e^{D(G_{\theta'}(z_i))})\big|\\
	&\leq \frac{1}{m}\sum_{i=1}^{m}|D(G_\theta(z_i))-D(G_{\theta'}(z_i))|\\
	&\overset{(a)}{=} \bbE_{p_z}|D(G_\theta(Z))-D(G_{\theta'}(Z))|+a_n\\
	&\overset{(b)}\leq \ell_1\|\theta-\theta'\|+a_n.
\end{align*}
where $a_n=o_p(1)$ uniformly for all $D\in\cD$ and $\theta,\theta'\in\Theta$ and $\ell_1>0$ is a constant. To obtain $(a)$, we note the compactness of $\Theta$ and $\cD$ and the envelope condition~\ref{ass:emp_envelop}, and then apply the uniform law of large numbers \cite[Theorem 2]{jennrich1969asymptotic}. $(b)$ is due to the Lipschitz continuity of $D(G_\theta(z))$ with respect to $\theta\in\Theta$. 

Given any $\epsilon>0$ and $\theta\in\Theta$. 
Then for all $\theta'$ such that $\|\theta-\theta'\|\leq\delta$, we have
\begin{equation*}
	\sup_{D\in\cD}|\hat{L}_d(D,\theta)-\hat{L}_d(D,\theta')| \leq \ell_1\|\theta-\theta'\|+a_n\leq\epsilon+a_n.
\end{equation*}

%
This implies for all $\epsilon>0$, $\theta\in\Theta$, there exists $\delta>0$ such that for every $\theta'$ with $\|\theta'-\theta\|\leq\delta$, we have $\sup_{D\in\cD}|\hat{L}_d(D,\theta)-\hat{L}_d(D,\theta')|\leq\epsilon$ asymptotically, that is, as $n\to\infty$, 
\begin{equation}\label{eq:L_unif_D}
	\bbP\bigg(\sup_{D\in\cD}|\hat{L}_d(D,\theta)-\hat{L}_d(D,\theta')|\leq\epsilon\bigg)\to1.
\end{equation}

Now we show that when $\theta$ and $\theta'$ are close, $\hat{D}_\theta$ is close to $\hat{D}_{\theta'}$ under the metric of $\|D\|_1=\bbE_{p_{0,\theta}}|D(X)|$. 
Given any $\theta\in\Theta$ and any $\epsilon>0$. Recall that $\hat{D}_\theta=\argmin_{D\in\cD}\hat{L}_d(D,\theta)$. 
Let $B(\hat{D}_\theta,\epsilon)=\{D\in\cD:\|D-\hat{D}_\theta\|_1<\epsilon\}$. Then we have 
\begin{equation*}
	\min_{D\in\cD\setminus B(\hat{D}_\theta,\epsilon)}\hat{L}_d(D,\theta)>\hat{L}_d(\hat{D}_\theta,\theta).
\end{equation*}
Let $0<\epsilon'<\min_{D\in\cD\setminus B(\hat{D}_\theta,\epsilon)}\hat{L}_d(D,\theta)-\hat{L}_d(\hat{D}_\theta,\theta)$. 
According to the continuity of $\hat{L}_d$ in $\theta$, there exists $\delta_1>0$, such that for every $\theta'$ with $\|\theta-\theta'\|\leq\delta_1$, we have 
\begin{equation}\label{eq:D*}
	|\hat{L}_d(\hat{D}_\theta,\theta)-\hat{L}_d(\hat{D}_\theta,\theta')|\leq\epsilon'/2.
\end{equation}
According to \eqref{eq:L_unif_D}, there exists $\delta_2>0$, such that for every $\theta'$ with $\|\theta-\theta'\|\leq\delta_2$, we have that asymptotically
\[\sup_{D\in\cD\setminus B(\hat{D}_\theta,\epsilon)}|\hat{L}_d(D,\theta)-\hat{L}_d(D,\theta')|\leq\epsilon'/2,\]
which implies  that asymptotically
\begin{equation}\label{eq:Dsup}
	\Big|\min_{D\in\cD\setminus B(\hat{D}_\theta,\epsilon)}\hat{L}_d(D,\theta)-\min_{D\in\cD\setminus B(\hat{D}_\theta,\epsilon)}\hat{L}_d(D,\theta')\Big|\leq\epsilon'/2.
\end{equation}

Let $\delta=\min\{\delta_1,\delta_2\}$ so that for every $\theta'$ with $\|\theta-\theta'\|\leq\delta$, both \eqref{eq:D*} and \eqref{eq:Dsup} hold, indicating that asymptotically
\[\min_{D\in\cD\setminus B(\hat{D}_\theta,\epsilon)}\hat{L}_d(D,\theta')>\hat{L}_d(\hat{D}_\theta,\theta').\] 
Therefore, for all $\theta$, for all $\epsilon>0$, there exists $\delta>0$ such that for all $\theta'$ with $\|\theta-\theta'\|\leq\delta$ we have asymptotically 
\begin{equation*}
	\|\hat{D}_{\theta'}-\hat{D}_\theta\|_1=\bbE_{p_{0,\theta}}|\hat{D}_{\theta'}(X)-\hat{D}_\theta(X)|<\epsilon.
\end{equation*}

Next, we show $\hat{D}_\theta(x)$ is continuous in $\theta$ asymptotically. 
Denote $v(x)=|\hat{D}_{\theta'}(x)-\hat{D}_\theta(x)|$. By the converse of the mean value theorem, if $x$ is not an extremum of $v$, then there exists a compact subset $K(x)$ of $\cX$ such that 
\[v(x)=\frac{1}{\nu(K(x))}\int_{K(x)}v(x')dx',\]
where $\nu$ denotes the Lebesgue measure. 
From the boundedness of $D^*(x)$ on $K(x)$ in condition~\ref{ass:D_bound}, we know that $p_{0,\theta}(x)$ is bounded away from 0 on $K(x)$, that is, there exists $M_0>0$ such that for all $x'\in K(x)$, $p_{0,\theta}(x')>M_0$. Then we can bound the above equation as follows.
\begin{align*}
	v(x)&\leq\frac{1}{M_0\nu(K(x))}\int_{K(x)}p_{0,\theta}(x')v(x')dx'\\
	&\leq \frac{1}{M_0\nu(K(x))}\int_{\cX}p_{0,\theta}(x')v(x')dx'\\
	&\leq\frac{\epsilon}{M_0\nu(K(x))}.
\end{align*}
Therefore, we have for all non-extrema $x$ and all $\theta$, $\hat{D}_\theta(x)$ is continuous in $\theta$ asymptotically. By Lipschitz continuity of $v$ in $x$ over any compact subset, we have for all extrema $x$ and all $\theta$, $\hat{D}_\theta(x)$ is continuous in $\theta$ asymptotically.

Therefore, we have as for all $x$ as $n\to\infty$
\begin{equation}\label{eq:Dhat_cont}
	\bbP\big(\hat{D}_\theta(x)\text{ is continuous in }\theta\big)\to1.
\end{equation}
Then due to the compactness of $\Theta$ and $K$, we have
\begin{equation*}
	\bbP\big(\hat{D}_\theta(x)\text{ is Lipschitz in }(\theta,x)\text{ over }\Theta\times K\big)\to1,
\end{equation*}
or equivalently, there exists a constant $\ell_2>0$ such that
\begin{equation}\label{eq:lip_in_p}
	\bbP\Big(\forall\theta,\theta'\in\Theta,x,x'\in K: |\hat{D}_\theta(x)-\hat{D}_{\theta'}(x')|\leq\ell_2\sqrt{\|\theta-\theta'\|^2+\|x-x'\|^2} \Big)\to1.
\end{equation}

Let $B_\delta=\{(\theta,\theta',x,x'):\theta,\theta'\in\Theta, x,x'\in K, \|\theta-\theta'\|^2+\|x-x'\|^2\leq\delta^2\}$ for any $\delta>0$. Given an arbitrary $\epsilon>0$. 
Let $E_n$ denote the event that 
\begin{equation*}
	\sup_{(\theta,\theta',x,x')\in B_\delta}|\hat{D}_\theta(x)-\hat{D}_{\theta'}(x')|\leq\ell_2\delta,
\end{equation*}
and let $F_n$ be the event that 
\begin{equation*}
	\sup_{(\theta,\theta',x,x')\in B_\delta}|\hat{D}_\theta(x)-\hat{D}_{\theta'}(x')|>\epsilon.
\end{equation*}
We know from \eqref{eq:lip_in_p} that for all $\delta$, $\lim_{n\to\infty}\bbP(E_n)=1$. Also note that $E_n\cap F_n\subseteq\{\ell_2\delta>\epsilon\}$ which implies for all $n$ and $\epsilon$,
\begin{equation*}
	\lim_{\delta\to0}\bbP(E_n\cap F_n)\leq \lim_{\delta\to0}\bbP(\ell_2\delta>\epsilon)=0.
\end{equation*}
Thus, for all $\epsilon>0$, there exists $N_0>0$ and $\delta_0>0$ such that for all $N>N_0$, $\delta<\delta_0$, we have $\bbP(E_n^c)<\epsilon/2$ and $\bbP(E_n\cap F_n)<\epsilon/2$, where $E_n^c$ denotes the complement of $E_n$. We then have
\begin{equation*}
	\bbP(F_n)=\bbP(E_n\cap F_n)+\bbP(E_n^c\cap F_n)\leq \bbP(E_n\cap F_n)+\bbP(E_n^c)<\epsilon.
\end{equation*}

Therefore, we have
\begin{equation*}
	\lim_{\delta\to0}\lim_{n\to\infty}\bbP(F_n)=\lim_{\delta\to0}\lim_{n\to\infty}\bbP\left(\sup_{(\theta,\theta',x,x')\in B_\delta}|\hat{D}_\theta(x)-\hat{D}_{\theta'}(x')|>\epsilon\right)=0,
\end{equation*}
which implies the stochastic uniform equicontinuity of $\hat{D}_\theta(x)$ with respect to $(\theta,x)\in\Theta\times K$.
%
\end{proof}

\begin{proof}[Proof of Theorem~\ref{thm:grad_conv}]
The proof proceeds in three steps.

\medskip
\noindent
\textit{Step I} \ \ We first establish the consistency of $\hat{D}_\theta(x)$ to $D_\theta^*(x)$ as defined in (\ref{eq:d_conv}) below based on the generalization analysis of maximum likelihood estimation. 

Following the probabilistic model in Section~\ref{sec:age}, by the Bayes formula we have $\bbP(Y=1|x)= p_*(x)/((1+\lambda)p_{0,\theta}(x))$ and $\bbP(Y=0|x)=\lambda p_\theta(x)/((1+\lambda)p_{0,\theta}(x))$ which define the probability mass function $p_{0,\theta}(y|x)$, $y\in\{0,1\}$. Let the joint probability functions $p_D(x,y)=p_D(y|x)p_{0,\theta}(x)$ and $p_{0,\theta}(x,y)=p_{0,\theta}(y|x)p_{0,\theta}(x)$.
Let the class
\begin{equation*}
	\cG=\left\{g(x,y)=\frac{1}{2}\ln\frac{p_D(x,y)+p_{0,\theta}(x,y)}{2p_{0,\theta}(x,y)}:D\in\cD\right\}.
\end{equation*}
Note that each element of $\cG$ can be written as
\begin{equation*}
	g(x,y)=\frac{1}{2}\ln\frac{p_D(y|x)+p_{0,\theta}(y|x)}{2p_{0,\theta}(y|x)}.
\end{equation*}

Let $g_\infty=\sup_{g\in\cG}|g|$. From condition~\ref{ass:envelope} we know that $\bbE_{p_{0,\theta}(x)}[\sup_{D\in\cD}|D(X)|]<\infty$. Note that when $|D'|\leq1$, $|\ln(1/(1+e^{-D'}))|\leq\ln(1+e)$; when $|D'|>1$, $|\ln(1/(1+e^{-D'}))|\leq |D'|\ln(1+e)$. Both inequalities hold when we replace $D'$ with $-D'$. Thus
\begin{eqnarray*}
	&&\bbE_{p_{0,\theta}(x,y)}\left[\sup_{D\in\cD}|\ln(p_D(Y|X))|\right]\\
	&=&\frac{1}{1+\lambda}\bbE_{p_*(x)}\left[\sup_{D\in\cD}\left|\ln\left(1/(1+e^{-D(X)}\lambda)\right)\right|\right]+\frac{\lambda}{1+\lambda}\bbE_{p_\theta(x)}\left[\sup_{D\in\cD}\left|\ln\left(1/(1+e^{D(X)}/\lambda)\right)\right|\right]\\
		&=&\frac{1}{1+\lambda}\bbE_{p_*(x)}\left[\sup_{D\in\cD}\left|\ln\left(1/(1+e^{-D}\lambda)\mathbf{1}_{(|D-\ln\lambda|\leq1)}\right)+\ln\left(1/(1+e^{-D}\lambda)\mathbf{1}_{(|D-\ln\lambda|>1)}\right)\right|\right]\\
	&&+\frac{\lambda}{1+\lambda}\bbE_{p_\theta(x)}\left[\sup_{D\in\cD}\left|\ln\left(1/(1+e^{D}/\lambda)\mathbf{1}_{(|D-\ln\lambda|\leq1)}\right)+\ln\left(1/(1+e^{D}/\lambda)\mathbf{1}_{(|D-\ln\lambda|>1)}\right)\right|\right]\\
	&\leq&\bbE_{p_{0,\theta}(x)}\left[\ln(1+e)\sup_{D\in\cD}\big|1+|D-\ln\lambda|\big|\right]\\
	&\leq &\ln(1+e)\cdot\left(1+\ln\lambda+\bbE_{p_{0,\theta}(x)}\left[\sup_{D\in\cD}|D(X)|\right]\right)<\infty.
\end{eqnarray*}
Hence we have $\bbE_{p_{0,\theta}(x,y)}[g_\infty(x,y)]<\infty$. 
Moreover for all $\delta>0$, the compactness of $\cD$ assumed in condition~\ref{ass:d_compact} implies a finite bracketing covering number defined in Definition~\ref{def:cover_num}, i.e., $N_{1,B}(\delta,\cD,\mu^*)<\infty$, where $\mu^*$ is the induced probability measure of density $p_{0,\theta}$. Then it follows from \cite[Theorem~4.3]{geer2000empirical} that  
\begin{equation}\label{eq:conv_in_h}
	H(p_{\hat{D}_\theta}(x,y),p_{0,\theta}(x,y))\to0
\end{equation}
almost surely as $n\to\infty$.

Consider any compact subset $K$ of $\cX$. We know from conditions~\ref{ass:D_bound} and~\ref{ass:D_cont} that for all $D\in\cD$, $D(x)$ is uniformly bounded and Lipschitz on $K$. Then $\cP=\{P_D(Y=1|x):D\in\cD\}$ is uniformly bounded and Lipschitz on $K$. Since $p_{0,\theta}(x)$ is continuous and hence bounded and Lipschitz on $K$, we know $\{p_D(x,y)=p_{0,\theta}(x)p_D(y|x)\}$ is uniformly bounded and Lipschitz with respect to $x$.

Also from the boundedness of $D^*(x)$ on $K$, we know that $p_{0,\theta}(x)$ is bounded away from 0 on $K$. Then it follows from (\ref{eq:conv_in_h}) and Lemma~\ref{lem:conv_density} that
\begin{equation*}
	\sup_{x\in K}|\hat{P}(Y=1|x)-\bbP(Y=1|x)|\pto0.
\end{equation*}
Then by continuous mapping theorem (Lemma~\ref{lem:ucmt}) and noting that $l(p)=\ln(\lambda p/(1-p))$ is uniformly continuous on a closed interval within $(0,1)$, we have as $n\to\infty$
\begin{equation}\label{eq:d_conv}
	\sup_{x\in K}|\hat{D}_\theta(x)-D_\theta^*(x)|\pto0.
\end{equation}

This directly implies the pointwise convergence, i.e., for all $\theta\in\Theta$, $x\in\cX$, as $n\to\infty$, $|\hat{D}_\theta(x)-D^*_\theta(x)|\pto0$. 
Further from conditions~\ref{ass:D_smooth} and~\ref{ass:D_cont}, we know that $D^*_\theta(x)$ is Lipschitz continuous with respect to $(\theta,x)$ over the compact set $\Theta\times K$. By Lemma~\ref{lem:D_cont}, we know $\hat{D}_\theta(x)$ is stochastic uniformly equicontinuous with respect to $(\theta,x)$ over $\Theta\times K$. Then by applying Lemma~\ref{lem:point_unif_conv}, we have as $n\to\infty$
\begin{equation}\label{eq:d_conv_unif}
	\sup_{\theta\in\Theta,x\in K}|\hat{D}_\theta(x)-D_\theta^*(x)|\pto0.
\end{equation}


\medskip\noindent
\textit{Step II} \ \  We then prove the consistency of the gradient $\nabla_x \hat{D}_\theta(x)$ to $\nabla_x D_\theta^*(x)$ as defined in (\ref{eq:grad_int_conv}). 


Given any $r>0$, let $B_r=\{x\in\cX:\|x\|\leq r\}\cap\cX$ and $B_r^c=\cX\setminus B_r$ be the complement.\footnote{We assume $\cX$ to be unbounded. In the case where $\cX$ is bounded, one can skip the introduction of $B_r$. } 
Let $w(x)$ be a function with bounded gradient $\nabla w(x)$ and Hessian $\nabla^2 w(x)$ such that 
\begin{equation*}
	w(x)=
	\begin{cases}
		1 & x\in B_r\\
		0 & x\in B^c_{2r}\\
		[0,1] & \text{otherwise}
	\end{cases}.
\end{equation*}
Let $u(x)=\hat{D}_\theta(x)- D_\theta^*(x)$. Consider partition $u(x)=u_1(x)+u_2(x)$ where $u_1(x):=u(x)w(x)$ and $u_2(x):=u(x)(1-w(x))$. Let $\mu_\theta$ be the probability measure induced by $p_\theta$. We then have
\begin{align}
	\int_\cX \|\nabla u(x)\|^2d\mu_\theta&\leq\int_\cX \|\nabla u_1(x)\|^2d\mu_\theta+\int_\cX \|\nabla u_2(x)\|^2d\mu_\theta\nonumber\\
	&=\int_{B_{2r}} \|\nabla u_1(x)\|^2d\mu_\theta+\int_{B_r^c} \|\nabla u_2(x)\|^2d\mu_\theta.\label{eq:decom_ball}
\end{align}


We first deal with the first term in (\ref{eq:decom_ball}). 
Since $u_1(x)$ is smooth and vanishes at the boundary of $B_{2r}$, we have from integration by parts that
\begin{align*}
	\int_{B_{2r}}\nabla u_1\nabla u_1^\top d\mu_\theta&=-\int_{B_{2r}}u_1\nabla(p_\theta\nabla u_1)dx\\
	&=-\int_{B_{2r}} u_1\nabla^2 u_1d\mu_\theta-\int_{B_{2r}}u_1\nabla u_1\nabla_x p_\theta^\top dx\\
	&=-\int_{B_{2r}} u_1\nabla^2 u_1d\mu_\theta-\int_{B_{2r}}u_1\nabla u_1[\nabla_x \ln p_\theta]^\top  d\mu_\theta,
\end{align*}
which implies
\begin{align*}
	\int_{B_{2r}}\|\nabla u_1\|^2d\mu_\theta &= -\int_{B_{2r}}u_1\ tr(\nabla^2u_1)d\mu_\theta-\int_{B_{2r}} tr(u_1\nabla u_1 [\nabla\ln p_\theta]^\top) d\mu_\theta\\
	&\leq \sqrt{\int_{B_{2r}}|u_1|^2d\mu_\theta\int_{B_{2r}} [tr(\nabla^2u_1)]^2d\mu_\theta} + \sqrt{\int_{B_{2r}}|u_1|^2d\mu_\theta \int_{B_{2r}}(\nabla u_1^\top \nabla\ln p_\theta )^2 d\mu_\theta}.
\end{align*}
by the Cauchy-Schwartz inequality. 


By condition~\ref{ass:D_bound} and noting that $|u_1(x)|\leq|u(x)|$ and $w(x)$ has bounded gradient and Hessian, there exists a constant $c_1>0$ (free of $\theta$) such that for all $\theta\in\Theta$ we have
\begin{equation}\label{eq:in1}
	\int_{B_{2r}}\|\nabla u_1\|^2d\mu_\theta \leq c_1\int_{B_{2r}}|u(x)|^2d\mu_\theta = c_1\int_{B_{2r}}|\hat{D}_\theta(x)-D_\theta^*(x)|^2d\mu_\theta.
\end{equation}

By the uniform convergence in (\ref{eq:d_conv}) over compact ball $B_{2r}$, we have for all $x\in B_{2r}$, there exists a sequence $a_{n}=o_p(1)$ which is free of $x$ and $\theta$ such that $|\hat{D}_\theta(x)-D^*_\theta(x)|^2\leq a_{n}$. Also the continuous function $p_\theta(x)$ is uniformly bounded for all $\theta\in\Theta$ and $x\in B_{2r}$. Then there exists another constant $c_2>0$ such that for all $\theta\in\Theta$ we have
\begin{equation}\label{eq:in2}
	\int_{B_{2r}}|\hat{D}_\theta(x)-D_\theta^*(x)|^2d\mu_\theta \leq \int_{B_{2r}}c_2a_{n}dx = o_p(1),
\end{equation}
where the last term is free of $\theta$. 

By combining \eqref{eq:in1} and \eqref{eq:in2}, given any $r>0$, as $n\to\infty$, we have
\begin{equation}\label{eq:in_conv}
	\sup_{\theta\in\Theta}\int_{B_{2r}} \|\nabla u_1(x)\|^2d\mu_\theta\pto0.
\end{equation} 

We then handle the second term in (\ref{eq:decom_ball}). 
Note that $\nabla u_2(x)=\nabla u(x)(1-w(x))-u(x)\nabla w(x)$. Then there exists a constant $c_3>0$ such that $\|\nabla u_2(x)\|^2\leq \|\nabla u(x)\|^2+c_3|u(x)|^2$ and thus
\begin{equation*}
	\int_{B_r^c} \|\nabla u_2(x)\|^2d\mu_\theta \leq \int_{B_r^c} \left(\|\nabla_x \hat{D}_\theta(x)-\nabla_x D_\theta^*(x)\|^2+c_3|\hat{D}_\theta(x)- D_\theta^*(x)|^2\right)d\mu_\theta.
\end{equation*}
Let $\cD^*=\{D^*_\theta:\theta\in\Theta\}$ be the true discriminator class. Note $\cD^*\subseteq\cD$ by condition~\ref{ass:d_compact}. We have
\begin{eqnarray*}
	&&\|\nabla_x\hat{D}_\theta(x)-\nabla_x D_\theta^*(x)\|^2+c_3|\hat{D}_\theta(x)- D_\theta^*(x)|^2\\
	&\leq&\sup_{D\in\cD,D'\in\cD^*}\|\nabla_x D(x)-\nabla_x D'(x)\|^2+c_3\sup_{D\in\cD,D'\in\cD^*}|D(x)-D'(x)|^2\\
	&\leq& 2\sup_{D\in\cD}\|\nabla_x D(x)\|^2+2c_3\sup_{D\in\cD}|D(x)|^2\\
	&=:& \tilde{D}(x),
\end{eqnarray*}
which implies
\begin{equation}\label{eq:out_bound}
	\int_{B_r^c} \|\nabla u_2(x)\|^2d\mu_\theta \leq \int_{B_r^c} \tilde{D}(x)d\mu_\theta.
\end{equation}

Note that for all $x\in\cX$, $\tilde{D}(x)\mathbf{1}_{\{x\in B_r^c\}}\to0$ as $r\to\infty$, which does not depend on $\theta$. Also, $|\tilde{D}(x)\mathbf{1}_{x\in B_r^c}|\leq|\tilde{D}(x)|$ and $\bbE_{p_\theta}[\tilde{D}(X)]<\infty$ from condition~\ref{ass:envelope}. Thus, by the dominated convergence theorem, as $r\to\infty$, we have for all $\theta\in\Theta$ that
\begin{equation}\label{eq:out_bound_conv}
	\int_{B_r^c} \tilde{D}(x)d\mu_\theta=\int_\cX \tilde{D}(x)\mathbf{1}_{\{x\in B_r^c\}}d\mu_\theta\to0.
\end{equation}


By combining \eqref{eq:decom_ball} and \eqref{eq:out_bound}, we have for any $n$ and $r>0$ that
\begin{equation}\label{eq:decom_bound}
	\int_\cX \|\nabla_x\hat{D}_\theta(x)-\nabla_x D_\theta^*(x)\|^2d\mu_\theta \leq \int_{B_{2r}} \|\nabla u_1(x)\|^2d\mu_\theta+\int_{B_r^c} \tilde{D}(x)d\mu_\theta.
\end{equation}

%

Denote the two terms on the right-hand side of \eqref{eq:decom_bound} by $R_1(\theta)$ and $R_2(\theta)$ respectively.
Given any $\epsilon>0$. \eqref{eq:in_conv} indicates that for all $r>0$ and $\delta>0$, there exists $n_0>0$ such that for every $n>n_0$, we have $\bbP(\sup_{\theta\in\Theta}R_1(\theta)>\epsilon/2)<\delta$. \eqref{eq:out_bound_conv} indicates that there exists $r_0>0$ such that for every $r>r_0$, we have $R_2(\theta)<\epsilon/2$ for all $\theta\in\Theta$ and hence $\sup_{\theta\in\Theta}R_2(\theta)\leq\epsilon/2$. Thus for every $r>r_0$, we further have 
\begin{align*}
	\bbP\bigg(\sup_{\theta\in\Theta}\int_{\cX} \|\nabla_x\hat{D}_\theta(x)-\nabla_x D_\theta^*(x)\|^2d\mu_\theta>\epsilon\bigg) &\leq \bbP\bigg(\sup_{\theta\in\Theta}R_1(\theta)+\sup_{\theta\in\Theta}R_2(\theta)>\epsilon\bigg) \\
	&\leq \bbP\bigg(\sup_{\theta\in\Theta}R_1(\theta)>\epsilon/2\bigg)+\bbP\bigg(\sup_{\theta\in\Theta}R_2(\theta)>\epsilon/2\bigg)\\
	&\leq \delta+0 = \delta.
\end{align*}
Therefore, as $n\to\infty$, we have 
\begin{equation}\label{eq:grad_int_conv}
	\sup_{\theta\in\Theta}\bbE_{p_\theta} \|\nabla_x\hat{D}_\theta(X)-\nabla_x D_\theta^*(X)\|^2\pto0.
\end{equation}

\medskip
\noindent
\textit{Step III} \ \  Based on the convergence statements developed above, we proceed to show the consistency of the estimated gradient $h_{\hat{D}}(\theta)$ and complete the proof.
Recall the definitions of $h_D(\theta)$ in \eqref{eq:est_grad} and $h'_D(z;\theta)$ in  \eqref{eq:h'_D}.

On one hand, from the compactness of $\Theta$ and $\cD$, dominated convergence in condition~\ref{ass:emp_envelop}, and the uniform law of large numbers \cite[Theorem 2]{jennrich1969asymptotic}, we have as $n\to\infty$,
\begin{equation*}
	\sup_{\theta\in\Theta,D\in\cD}\|h_{D}(\theta)-\bbE_{p_z}[h_D'(Z;\theta)]\|\pto0,
\end{equation*}
which implies
\begin{equation}\label{eq:triangle1}
	\sup_{\theta\in\Theta}\|h_{\hat D_\theta}(\theta)-\bbE_{p_z}[h_{\hat D_\theta}'(Z;\theta)]\|\pto0.
\end{equation}

On the other hand, note that $\nabla L(\theta)=\bbE_{p_z} [h_{D^*_\theta}'(Z;\theta)]$. We have 
\begin{align*}
	\sup_{\theta\in\Theta}\big\|\bbE_{p_z}[h_{\hat D_\theta}'(Z;\theta)]-\nabla L(\theta)\big\| &\leq \sup_{\theta\in\Theta}\bbE_{p_z}\big\|h_{\hat D_\theta}'(Z;\theta)-h_{D^*_\theta}'(Z;\theta)]\big\|\\
	&\overset{(a)}{=}\sup_{\theta\in\Theta}\bbE_{p_z}\big\|\nabla_\theta G_\theta(Z)^\top\big(\nabla_x\hat{D}_\theta(G_\theta(Z))-\nabla_x D^*_\theta(G_\theta(Z))\big)\big\|\\
	&\overset{(b)}\leq \sqrt{\sup_{\theta\in\Theta}\bbE_{p_z}\|\nabla_\theta G_\theta(Z)\|^2\sup_{\theta\in\Theta}\bbE_{p_z}\big\|\nabla_x\hat{D}_\theta(G_\theta(Z))-\nabla_x D^*_\theta(G_\theta(Z))\big\|^2}\\
	&\overset{(c)}\leq \sqrt{c_4\sup_{\theta\in\Theta}\bbE_{p_\theta}\big\|\nabla_x\hat{D}_\theta(X)-\nabla_x D^*_\theta(X)\big\|^2},
\end{align*}
where in $(a)$ we consider the estimated gradient of reverse KL divergence as the objective while other divergences can be handled similarly, $(b)$ follows by applying the Cauchy-Schwartz inequality, and $(c)$ follows from condition~\ref{ass:G_smooth} and reparametrization, with a constant $c_4>0$.

Then according to \eqref{eq:grad_int_conv}, we have as $n\to\infty$
\begin{equation}\label{eq:triangle2}
	\sup_{\theta\in\Theta}\big\|\bbE_{p_z}[h_{\hat D_\theta}'(Z;\theta)]-\nabla L(\theta)\big\|\pto0.
\end{equation}

By the triangle inequality, we have 
\begin{equation*}
	\sup_{\theta\in\Theta}\|h_{\hat D}(\theta)-\nabla L(\theta)\|\leq\sup_{\theta\in\Theta}\|h_{\hat D}(\theta)-\bbE_{Z}[h_{\hat D}'(Z;\theta)]\| + \sup_{\theta\in\Theta}\|\bbE_{Z}[h_{\hat D}'(Z;\theta)]-\nabla L(\theta)\|.
\end{equation*}

Then by \eqref{eq:triangle1} and \eqref{eq:triangle2}, we have as $n\to\infty$ 
\begin{equation*}
	\sup_{\theta\in\Theta}\|h_{\hat D}(\theta)-\nabla L(\theta)\|\pto0,
\end{equation*}
which completes the proof.	
\end{proof}

\subsection{Proof of Theorem~\ref{thm:cons}}\label{app:pf_cons}

Let us first consider a general approximate gradient descent algorithm to minimize a function $f(\theta)$ with respect to $\theta$. For $\theta\in\Theta$, let $\hat{g}(\theta)$ be an estimate of gradient $\nabla f(\theta)$ based on a sample of size $n$. In Algorithm~\ref{alg:agd}, the time horizon $T$ is chosen to be sufficiently large such that $T=\mathbf\Theta(n)$. 

{\centering
\begin{minipage}{.6\linewidth}
\vskip 0.1in
\begin{algorithm}[H]
\DontPrintSemicolon
\KwInput{Initial parameter $\theta_0$, meta-parameter $T$}
\For{$t=0,1,2,\dots,T$}{
Compute approximate gradient $\hat{g}_t=\hat{g}(\theta_t)$\\
$\theta_{t+1}=\theta_{t}-\eta \hat{g}_t$ for some $\eta>0$
}
\KwReturn{$\argmin_{\theta_t:t=1,\dots,T}\|\hat{g}_t\|$}
\caption{Approximate Gradient Descent}
\label{alg:agd}
\end{algorithm}
\end{minipage}
\vskip 0.1in
\par
}

\begin{lemma}\label{lem:agd}
	Let $\hat\theta$ be the output of Algorithm~\ref{alg:agd}. Let $\hat\delta(\theta)=\hat{g}(\theta)-\nabla f(\theta)$ and $\hat\delta=\sup_{\theta\in\Theta} \|\hat\delta(\theta)\|$. 
	Suppose $f(\theta)$ is lower bounded and $\ell_0$-smooth for some $\ell_0>0$.
	Then we have $\|\nabla f(\hat\theta)\|=O_p(\hat\delta\vee 1/n)$ where $a\vee b=\max\{a,b\}$. 
\end{lemma}

\begin{proof}[Proof of Lemma~\ref{lem:agd}]
We recall the approximate gradient descent step in Algorithm~\ref{alg:agd}
\begin{equation*}
	\theta_{t+1}=\theta_{t}-\eta \hat{g}(\theta_{t}),
\end{equation*}
where $\eta>0$ is the learning rate.
By the $\ell_0$-smoothness of $f(\theta)$, we have
\begin{equation*}
	f(\theta_{t+1}) \leq f(\theta_{t}) - \eta \hat{g}(\theta_{t})^\top\nabla f(\theta_{t}) + \frac{\eta^2\ell_0}{2}\hat{g}(\theta_{t})^\top \hat{g}(\theta_{t}).
\end{equation*}

Under the case where $\|\nabla f(\theta_{t})\|^2\geq 2\hat\delta^2$, we have
\begin{align*}
	- \eta \hat{g}(\theta_{t})^\top\nabla f(\theta_{t}) &= - \eta \left(\hat\delta(\theta_{t})+\nabla f(\theta_{t})\right)^\top \nabla f(\theta_{t})\\
	&\leq \eta \left(-\|\nabla f(\theta_{t})\|^2 + (\|\nabla f(\theta_{t})\|^2+\hat\delta^2)/2\right)\\
	&=-\frac{\eta}{2}\left(\|\nabla f(\theta_{t})\|^2-\hat\delta^2\right)\\
	&\leq -\frac{\eta}{4}\|\nabla f(\theta_{t})\|^2,
\end{align*}
and
\begin{align*}
	\|\hat{g}(\theta_{t})\|^2 = \|\hat\delta(\theta_{t-1})+\nabla f(\theta_{t})\|^2 \leq 2\big(\hat\delta^2+\|\nabla f(\theta_{t})\|^2\big) \leq 3\|\nabla f(\theta_{t})\|^2.
\end{align*}
Then we have
\begin{align*}
	f(\theta_{t+1}) &\leq f(\theta_{t}) - \frac{\eta}{4}\|\nabla f(\theta_{t})\|^2 + \frac{3\eta^2\ell_0}{2}\|\nabla f(\theta_{t})\|^2\\
	&\leq f(\theta_{t}) - \frac{\eta}{8}\|\nabla f(\theta_{t})\|^2,
\end{align*}
when $\eta<1/12\ell_0$, which can be satisfied with a sufficiently small learning rate.

By summing over $t=0,1,\dots,T-1$, we have  
\begin{equation*}
	f(\theta_{T}) \leq f(\theta_0) - 0.125 \eta \sum_{t=0}^T \|\nabla f(\theta_{t})\|^2.
\end{equation*}
Note that $f(\theta)$ is lower bounded. Then we have $\sum_{t=0}^T \|\nabla f(\theta_{t})\|^2 = O(1)$. Thus there exists $t$ in $\{0,\dots,T\}$ such that $\|\nabla f(\theta_{t})\|^2 = O(1/T)=O(1/n)$, since we set $T=\mathbf\Theta(n)$. 

Otherwise there exists $t$ such that $\|\nabla f(\theta_{t})\|<\sqrt{2}\hat\delta$.

Therefore, by combining the two cases, we have 
we have $$\|\nabla f(\hat\theta)\|=O_p(\hat\delta\vee 1/n).$$
\end{proof}

Now we apply Lemma~\ref{lem:agd} to prove Theorem~\ref{thm:cons}.
\begin{proof}[Proof of Theorem~\ref{thm:cons}]
We apply Lemma~\ref{lem:agd} by taking $f(\theta)=L(\theta)$ and $\hat{g}(\theta)=h_{\hat{D}}(\theta)$. We have from Theorem~\ref{thm:grad_conv} that $\hat\delta=\sup_{\theta\in\Theta}\|h_{\hat{D}}(\theta)-\nabla L(\theta)\|=o_p(1)$; we know $L(\theta)$ is smooth and lower bounded by 0; we also take $T=\mathbf\Theta(n)$. Then Lemma~\ref{lem:agd} implies that as $n\to\infty$ 
\begin{equation*}
	\|\nabla L(\hat\theta)\|\pto0.
\end{equation*}

Then by the Polyak-\L{ojasiewicz} condition~\ref{ass:pl}, we have
\begin{equation*}
	L(\hat\theta) - L(\theta^*) \pto 0,
\end{equation*}
which leads to the desired result.
%
%
\end{proof}

\subsection{Proof of Corollary~\ref{cor:cons_param}}\label{app:pf_cons_param}
\begin{proof}
Given any $\epsilon>0$, by the identifiability condition \textit{A10}, there exists $\delta>0$ such that for every $\theta$ with $|L(\theta)-L(\theta^*)|\leq\delta$, we have $\|\theta-\theta^*\|\leq\epsilon$. Then we have
\begin{equation}\label{eq:cons_param_ineq}
	\bbP\big(\|\gage-\theta^*\|\leq\epsilon\big)\geq \bbP\big(|L(\hat\theta)-L(\theta^*)|\leq\delta\big). 
\end{equation}
By Theorem~\ref{thm:cons}, we know as $n\to\infty$, $L(\hat\theta) - L(\theta^*) \pto 0$, which implies the right-hand side of \eqref{eq:cons_param_ineq} converges to 1. Therefore, the left-hand side of \eqref{eq:cons_param_ineq} also converges to 1, which implies $\hat\theta\pto\theta^*$, as $n\to\infty$. 
\end{proof}

\subsection{Proof of Theorem~\ref{thm:d_par_cons}}\label{app:pf_d_par_cons}
\begin{proof}
To lighten the notation, we denote $\hat\psi_\theta=\dage(\theta)$ and $\psi^*_\theta=\psi^*(\theta)$ in the proofs. 
Given arbitrary $\epsilon>0$ and $\theta\in N(\theta^*)$. Let $B(\psi^*_\theta,\epsilon)=\{\psi: \|\psi-\psi^*_\theta\|<\epsilon\}$. We have 
$$
\min_{\psi\in\Psi \setminus B(\psi^*_\theta,\epsilon)} L_d(\psi,\theta)>L_d(\psi^*_\theta,\theta).
$$
By the weak law of large numbers, we have $\hat{L}_d(\psi^*_\theta,\theta)\pto L_d(\psi^*_\theta,\theta)$ as $n\to\infty$. 

Note from condition~\ref{ass:d_realizable} that the parameter space $\Psi$ is compact. 
For $i=1,2$, condition~\ref{ass:d_cont} ensures that $l_i(x;\psi)$ is continuous with respect to $\psi$ for each $x$; condition~\ref{ass:d_envelop} ensures the existence of an integrable function that uniformly dominates $l_i(x;\psi)$ for all $\psi$. Then by the uniform law of large numbers \cite[Theorem 2]{jennrich1969asymptotic}, we have 
$$
\max_{\psi\in\Psi \setminus B(\psi^*_\theta,\epsilon)}|\hat{L}_d(\psi,\theta)- L_d(\psi,\theta)|\pto0,
$$
%
as $n\to\infty$, which implies 
$$
\min_{\psi\in\Psi \setminus B(\psi^*_\theta,\epsilon)}\hat{L}_d(\psi,\theta)\pto\min_{\psi\in\Psi \setminus B(\psi^*_\theta,\epsilon)}{L}_d(\psi,\theta)>L_d(\psi^*_\theta,\theta),
$$
and then $\bbP\big(\hat{\psi}_\theta\in B(\psi^*_\theta,\epsilon)\big)\to1$ as $n\to\infty$. Since $\epsilon$ is arbitrary, we have for all $\theta\in N(\theta^*)$, $\hat{\psi}_\theta\pto \psi^*_\theta$ as $n\to\infty$.
\end{proof}

\subsection{Proof of Theorem~\ref{thm:d_asy_norm}}\label{app:pf_d_eff}

\begin{proof}
Given any $\theta\in N(\theta^*)$. By consistency in Theorem~\ref{thm:d_par_cons} and Taylor expansion with the integral remainder, we have 
\begin{equation*}
	\nabla_\psi \hat{L}_d(\hat\psi_\theta,\theta)-\nabla_\psi \hat{L}_d(\psi^*_\theta,\theta) = \hat{I}\big(\hat\psi_\theta-\psi^*_\theta\big),
\end{equation*}
where 
\[\hat{I}=\int_0^1\nabla^2_\psi \hat{L}_d\big(\psi^*_\theta+t(\hat\psi_\theta-\psi^*_\theta),\theta\big)dt\]
and we note $\nabla_\psi \hat{L}_d(\hat\psi_\theta,\theta)=0$.
Then
\begin{equation}\label{eq:d_expand}
	\sqrt{n}\left(\hat{\psi}_\theta-\psi^*_\theta\right)=-\hat{I}^{-1}\left[\frac{1}{\sqrt{n}}\sum_{i=1}^n\nabla_{\psi}l_1(x_i;\psi^*)+\frac{1}{\sqrt{\lambda m}}\sum_{i=1}^m\nabla_{\psi}l_2(G_\theta(z_i);\psi^*)\right].
\end{equation}


Note from condition~\ref{ass:d_realizable} that the parameter space $\Psi$ is compact. For $i=1,2$, condition~\ref{ass:d_cont} ensures that $\nabla^2_\psi l_i(x;\psi)$ is continuous with respect to $\psi$ for each $x$; condition~\ref{ass:d_envelop} ensures the existence of an integrable function that uniformly dominates $\nabla^2_\psi l_i(x;\psi)$ for all $\psi$. Then by the uniform law of large numbers \cite[Theorem 2]{jennrich1969asymptotic}, we have as $n\to\infty$ 
\begin{equation*}
	\sup_{\psi\in\Psi}\big\|\nabla^2_\psi \hat{L}_d(\psi,\theta)-\nabla^2_\psi L_d(\psi,\theta)\big\|\pto0,
\end{equation*}
which implies as $n\to\infty$, for all $t\in[0,1]$ we have
\begin{equation}\label{eq:d_hessian_ulln}
	\big\|\nabla^2_\psi \hat{L}_d\big(\psi^*_\theta+t(\hat\psi_\theta-\psi^*_\theta),\theta\big)-\nabla^2_\psi L_d\big(\psi^*_\theta+t(\hat\psi_\theta-\psi^*_\theta),\theta\big)\big\|\pto0.
\end{equation}
From the consistency result in Theorem~\ref{thm:d_par_cons} we have $\tilde\psi_\theta\pto\psi^*_\theta$. By the continuous mapping theorem we have for all $t\in[0,1]$, $\nabla^2_\psi L_d\big(\psi^*_\theta+t(\hat\psi_\theta-\psi^*_\theta),\theta\big)\pto\nabla^2_\psi L_d(\psi^*_\theta,\theta)$ as $n\to\infty$. Then \eqref{eq:d_hessian_ulln} implies
\begin{equation*}
	\hat{I}\pto\nabla^2_\psi L_d(\psi^*_\theta,\theta)=H_d\succ0
\end{equation*}
by condition~\ref{ass:d_pos_hess}. 

By the central limit theorem and independence, we have as $n\to\infty$,
\begin{equation*}
	\frac{1}{\sqrt{n}}\sum_{i=1}^n\nabla_{\psi}l_1(x_i;\psi^*_\theta)+\frac{1}{\sqrt{\lambda m}}\sum_{i=1}^m\nabla_{\psi}l_2(G_\theta(z_i);\psi^*_\theta)\dto\cN(0,V_d),
\end{equation*}
where $V_d=V_1+V_2/\lambda$, $V_1=\mathrm{Var}(\nabla_\psi l_1(X;\psi^*_\theta))$, and $V_2=\mathrm{Var}(\nabla_\psi l_2(G_\theta(Z);\psi^*_\theta))$. Then we have
\begin{align*}
	V_d&=\bbE_{p_*}[\nabla_{\psi}l_1(X;\psi^*_\theta)\nabla_{\psi}l_1(X;\psi^*_\theta)^\top]+\bbE_{p_\theta}[\nabla_{\psi}l_2(X;\psi^*_\theta)\nabla_{\psi}l_2(X;\psi^*_\theta)^\top]/\lambda\\
	&\ \ -(1+1/\lambda)\bbE_{p_*}[\nabla_{\psi}l_1(X;\psi^*_\theta)]\bbE_{p_*}[\nabla_{\psi}l_1(X;\psi^*_\theta)]^\top.
\end{align*}

By the Slutsky's theorem, we have as $n\to\infty$,
\begin{equation*}
\sqrt{n}\left(\hat{\psi}_\theta-\psi^*_\theta\right)\dto\cN(0,\Sigma_d(\theta)),
\end{equation*}
where $\Sigma_d(\theta)=H_d^{-1}V_dH_d^{-1}$.
\end{proof}

\subsection{Proof of Theorem~\ref{thm:f_d_asy_norm}}\label{app:pf_f_d_eff}
\begin{proof} 
	We can show the asymptotic normality of $\dfgan$ by the arguments similar to those in the proof of Theorem~\ref{thm:d_asy_norm}, where $l_i$ is replaced by $l^f_i$ for $i=1,2$.
\end{proof}



\subsection{Proof of Theorem~\ref{thm:asy_normal}}\label{app:pf_asy_normal}
For the purpose of analysis and better comprehension of the estimation error, we introduce some intermediate variables. 
Let 
\[
\tilde\theta=\argmin_{\theta\in\Theta}\hat{L}(\theta)
\]
be the empirical estimator which is intractable to compute due to the unknown densities $p_*$ and $p_\theta$ in the objective function. 
We define the bias of gradient estimation by $\hat\epsilon(\theta)=h_{\hat{D}_\theta}(\theta)-\nabla\hat{L}(\theta)$, where $\hat{D}_\theta=D_{\hat\psi(\theta)}$, and define an objective function with bias correction 
\begin{equation}\label{eq:L'}
	\hat L'(\theta)=\hat L(\theta)+\hat\epsilon(\theta^*)^\top\theta.
\end{equation}
Let 
\[
\theta'=\argmin_{\theta\in\Theta} \hat L'(\theta) .
\]

We first show the consistency of $\tilde\theta$ and $\theta'$ in the following lemma.
\begin{lemma}\label{lem:g_cons}
	Under sets A-C of conditions, as $n\to\infty$, we have $\theta'\pto\theta^*$ and $\tilde\theta\pto\theta^*$.
\end{lemma}
\begin{proof}[Proof of Lemma~\ref{lem:g_cons}]
Given arbitrary $\epsilon>0$. Let $B(\theta^*,\epsilon)=\{\theta:\|\theta-\theta^*\|<\epsilon\}$. The identifiability condition \textit{A10} indicates the uniqueness of the global minimum $\theta^*$ of $L(\theta)$. Then we have 
\begin{equation*}
	\min_{\theta\in\Theta\setminus B(\theta^*,\epsilon)}L(\theta)>L(\theta^*).
\end{equation*}
Note that $\hat L'(\theta)=\hat L(\theta)+\hat\epsilon(\theta^*)^\top\theta$ where
\begin{equation*}
	\hat\epsilon(\theta^*)=h_{\hat{D}}(\theta^*)-\nabla{L}(\theta^*)+\nabla{L}(\theta^*)-\nabla\hat{L}(\theta^*).
\end{equation*}
We know from Theorem~\ref{thm:grad_conv} that as $n\to\infty$, $\|h_{\hat{D}}(\theta^*)-\nabla{L}(\theta^*)\|\pto0$, and from the weak law of large numbers that $\|\nabla{L}(\theta^*)-\nabla\hat{L}(\theta^*)\|\pto0$. Hence $\|\hat\epsilon(\theta^*)\|\pto0$ as $n\to\infty$. Also note $\|\theta\|$ is bounded.
Further by the uniform law of large numbers \cite[Theorem 2]{jennrich1969asymptotic}, we have
\begin{equation*}
	\sup_{\theta\in\Theta}|\hat L'(\theta)-L(\theta)|=\sup_{\theta\in\Theta}|\hat L(\theta)-L(\theta)|+o_p(1)\pto0,
\end{equation*}
as $n\to\infty$, which implies $\hat L'(\theta^*)\pto L(\theta^*)$ and 
\begin{equation*}
	\min_{\theta\in\Theta\setminus B(\theta^*,\epsilon)}\hat L'(\theta) \pto \min_{\theta\in\Theta\setminus B(\theta^*,\epsilon)}L(\theta)>L(\theta^*).
\end{equation*}
Then $\bbP(\theta'\in B(\theta^*,\epsilon))\pto1$ as $m,n\to\infty$. Since $\epsilon$ is arbitrary, we have $\theta'\pto\theta^*$ as $n\to\infty$. The consistency of $\tilde\theta$ can be obtained similarly.
\end{proof}


We first decompose the deviation of estimation into the deviations between the intermediate estimators 
\begin{equation*}
	\hat\theta-\theta^*=\big(\tilde\theta-\theta^*\big)+\big(\theta'-\tilde\theta\big)+\big(\hat\theta-\theta'\big),
\end{equation*}
and present the asymptotic normality results for each error term in the following lemma, which is then used to obtain the asymptotic normality of $\hat\theta$. 
\begin{lemma}\label{lem:g_rate}
Let 
\begin{align*}
	\gamma_n&=-[\nabla^2\hat{L}(\theta^*)]^{-1}\nabla\hat L(\theta^*)\\
	\xi_n&=-[\nabla^2\hat{L}(\theta^*)]^{-1}(h_{\hat D}(\theta^*)-h_{D^*}(\theta^*))\\
	\tau_n&=-[\nabla^2\hat{L}(\theta^*)]^{-1}(h_{D^*}(\theta^*)-\nabla\hat L(\theta^*)).
\end{align*}
Under sets A-C of conditions, we have 
\begin{align}
	\tilde\theta-\theta^*&=\gamma_n+O_p(1/n)\label{eq:o1_lem}\\
	\theta'-\tilde\theta&=\xi_n+\tau_n+O_p(1/n)\label{eq:o2_lem}\\
	\hat\theta-\theta'&=O_p(1/n)\label{eq:o3_lem}
\end{align}
and 
$\sqrt{n}\gamma_n\dto\cN(0,H_g^{-1}\Sigma_gH_g^{-1}/\lambda)$, $\sqrt{n}\xi_n\dto\cN(0,H_g^{-1}C\Sigma_{d,\theta^*}C^\top H_g^{-1})$, and $\sqrt{n}\tau_n\dto\cN(0,H_g^{-1}\Sigma_\tau H_g^{-1}/\lambda)$ as $n\to\infty$, 
$H_g:=\nabla^2_\theta L(\theta)$, $\Sigma_g:=\bbE[\nabla l_g(\theta^*)\nabla l_g(\theta^*)^\top]$, $\Sigma_{d,\theta^*}:=\Sigma_d(\theta)$, $C:=\nabla_\psi\bbE[h'_{D_{\psi}}(Z;\theta^*)]|_{\psi^*(\theta)}$, and $\Sigma_\tau:=\bbE[\tau\tau^\top]$ with $\tau:=h'(Z;\theta^*)-\nabla_\theta l_g(Z;\theta^*)$. We also have
\begin{equation}\label{eq:grad_diff_n}
	\sup_{\|\theta-\theta^*\|=O_p(1/\sqrt{n})}\|h_{\hat D}(\theta)-\nabla\hat L'(\theta)\|=O_p(1/n).
\end{equation}
\end{lemma}

Based on Lemma~\ref{lem:g_rate}, we first discuss the three error terms to provide some insights on the asymptotic behavior of GAN algorithms, followed by the proof of it.
\begin{itemize}
\setlength{\itemsep}{5pt}
\setlength{\parskip}{2pt}
	\item $\tilde\theta-\theta^*$: deviation of the empirical estimator $\tilde\theta$ from the target parameter $\theta^*$, whose asymptotic distribution can be obtained by applying a standard central limit theorem. As we have more generated samples, i.e., as $\lambda$ grows, this term becomes more concentrated and eventually vanishes as $\lambda\to\infty$.
	\item $\theta'-\tilde\theta$: the error caused by the bias of gradient estimation $\hat\epsilon(\theta^*)$ which can be decomposed into $\xi_n$ and  $\tau_n$. $\xi_n$ is the difference $h_{\hat{D}}(\theta^*)-h_{D^*}(\theta^*)$ which essentially comes from the discriminator estimation. $\tau_n$ is the difference $h_{D^*}(\theta^*)-\nabla_\theta\hat{L}(\theta^*)$ between two gradient evaluation approaches, where $h_{D^*}(\theta)$ is the empirical version based on our gradient formula in Theorem~\ref{thm:grad} and $\nabla\hat{L}(\theta)$ is the gradient of the empirical loss (\ref{eq:obj_g_emp}) which, however, is intractable. 
	\item $\hat\theta-\theta'$: the error appeared from the gradient descent algorithm. Since $\theta'$ is the minimizer of the objective $\hat{L}'(\theta)$ already with bias correction, the errors caused by discriminator estimation and different gradient evaluations can be absorbed into a higher order term. 
\end{itemize}

\begin{proof}[Proof of Lemma~\ref{lem:g_rate}]

After establishing the consistency of $\tilde\theta$, $\theta'$ and $\hat\theta$, we have the following basic expansions which will be used later.
\begin{equation}\label{eq:ttheta}
\nabla \hat L(\tilde\theta)-\nabla\hat L(\theta^*)=\nabla^2\hat{L}(\theta^*)(\tilde\theta-\theta^*)+O_p(\|\tilde\theta-\theta^*\|^2),
\end{equation}
\begin{equation}\label{eq:theta'}
\nabla \hat L'(\theta')-\nabla\hat L'(\theta^*)=\nabla^2\hat{L}'(\theta^*)(\theta'-\theta^*)+O_p(\|\theta'-\theta^*\|^2),
\end{equation}
\begin{equation}\label{eq:htheta}
\nabla \hat L'(\hat\theta)-\nabla\hat L'(\theta^*)=\nabla^2\hat{L}'(\theta^*)(\hat\theta-\theta^*)+O_p(\|\hat\theta-\theta^*\|^2).
\end{equation}

\medskip\noindent
\textit{Step I \ \ Convergence rate of $\tilde\theta-\theta^*$}

By noting $\nabla \hat L(\tilde\theta)=0$ and multiplying both sides of (\ref{eq:ttheta}) by the inverse of $\nabla^2\hat{L}(\theta^*)$, we have 
\begin{equation}\label{eq:o1}
\tilde\theta-\theta^* = \gamma_n+O_p(1/n),
\end{equation}
where 
\begin{equation*}
	\gamma_n=-\left[\nabla^2\hat{L}(\theta^*)\right]^{-1}\nabla\hat L(\theta^*).
\end{equation*}
By the weak law of large numbers we have $\nabla^2\hat{L}(\theta^*)\pto\bbE[\nabla^2l_g(\theta^*)]=H_g\succ0$ by condition~\ref{ass:g_pos_hess}. Then by the central limit theorem and Slutsky's theorem, we obtain $$\sqrt{n}\gamma_n=\sqrt{m/\lambda}\gamma_n\dto \cN(0,H_g^{-1}\Sigma_g H_g^{-1}/\lambda)$$ as $n\to\infty$ with $\Sigma_g$ and $H_g$ defined in Lemma~\ref{lem:g_rate}.

\bigskip\noindent
\textit{Step II \ \ Convergence rate of $\theta'-\tilde\theta$}

Note from the definition $\hat L'(\theta)=\hat L(\theta)+\hat\epsilon(\theta^*)^\top\theta$ that in (\ref{eq:theta'}) we have $\nabla \hat L'(\theta')=0$, $\nabla\hat L'(\theta^*)=\nabla\hat L(\theta^*)+\hat\epsilon(\theta^*)$, and $\nabla^2\hat{L}'(\theta^*)=\nabla^2\hat{L}(\theta^*)$. 
Then taking (\ref{eq:ttheta}) $-$ (\ref{eq:theta'}) gives 
\begin{equation}\label{eq:bias}
\hat\epsilon(\theta^*)=\nabla^2\hat{L}(\theta^*)(\tilde\theta-\theta')+O_p(\|\tilde\theta-\theta^*\|^2+\|\theta'-\theta^*\|^2).
\end{equation}

For $\theta\in N(\theta^*)$, to lighten the notation, we denote $\psi^*_\theta=\psi^*(\theta)$, $\hat\psi_\theta=\hat\psi(\theta)$, $D^*_\theta=D_{\psi^*_\theta}$, and $\hat{D}_\theta=D_{\hat\psi_\theta}$.
We decompose the bias of gradient estimation defined by $\hat\epsilon(\theta)=h_{\hat{D}_\theta}(\theta)-\nabla\hat{L}(\theta)$ into
 $$\hat\epsilon(\theta)=\hat\epsilon_1(\theta)+\hat\epsilon_2(\theta),$$ where $\hat\epsilon_1(\theta):=h_{\hat{D}_\theta}(\theta)-h_{D^*_\theta}(\theta)$ and $\hat\epsilon_2(\theta):=h_{D^*_\theta}(\theta)-\nabla\hat{L}(\theta)$.
We have
\begin{equation}\label{eq:eps1_def}
	\hat\epsilon_{1}(\theta)=-\frac{1}{m}\sum_{i=1}^m \left[\nabla_\theta G_\theta(z_i)^\top \left(\hat{s}(G_\theta(z_i))\nabla_x \hat{D}_\theta(G_\theta(z_i))-s^*(G_\theta(z_i))\nabla_x D^*_\theta(G_\theta(z_i))\right) \right],
\end{equation}
where $\hat{s}(x)=s(x;\hat{D}_\theta)$ and $s^*(x)=s(x;D^*_\theta)$ are the scaling factors as defined in Theorem~\ref{thm:grad}.

Let $\xi'_n(\theta)=\hat\psi_\theta-\psi^*_\theta$. 
For any $x,x'\in\mathcal{X}$, by Taylor expansion of function $\hat{s}(x')D_{\psi}(x)$ with respect to $\psi$, we have
\begin{equation}\label{eq:expa1}
\hat{s}(x')\hat{D}_\theta(x)-s^*(x')D^*_\theta(x)=\nabla_{\psi} [s^*(x')D^*_\theta(x)]^\top\xi'_n+\tfrac{1}{2}{\xi'_n}^\top\nabla^2_\psi[\tilde{s}(x')\tilde{D}_{\theta}(x)]\xi'_n,
\end{equation}
where $\tilde{D}_\theta=D_{\tilde\psi(\theta)}$ with $\tilde\psi(\theta)=\psi^*(\theta)+t(\hat\psi(\theta)-\psi^*(\theta))$ for some $t\in[0,1]$, and $\tilde{s}(x)=s(x;\tilde{D}_\theta)$.
We know from Theorem~\ref{thm:d_asy_norm} that $\sqrt{n}\xi'_n(\theta)\dto\cN(0,\Sigma_d(\theta))$. Note that $\bbE[\xi'_n(\theta)]=0$ and $\Sigma_d(\theta)$ is continuous with respect to $\theta$ on bounded set $N(\theta^*)$. Then by applying Lemma~\ref{lem:unif_order} to sequence $\xi'_n(\theta)$, we have $\xi'_n(\theta)=O_p(1/\sqrt{n})$ uniformly for all $\theta\in N(\theta^*)$. 

Let $x^1$ be the first component of $x$ and $e_1$ be a $d$-dimensional unit vector whose first component is one and all the others are zero.
Then we have
\begin{eqnarray}
&&\hat{s}(x)\frac{\partial\hat{D}_\theta(x)}{\partial x^1}-s^*(x)\frac{\partial D^*_\theta(x)}{\partial x^1}\nonumber\\
&=&\hat{s}(x)\lim_{\delta\to0}\frac{\hat{D}_\theta(x+\delta e_1)-\hat{D}_\theta(x)}{\delta}-s^*(x)\lim_{\delta\to0}\frac{{D^*_\theta}(x+\delta e_1)-{D^*_\theta}(x)}{\delta}\nonumber\\
&=&\lim_{\delta\to0}\frac{\hat{s}(x)\hat{D}_\theta(x+\delta e_1)-s^*(x){D^*_\theta}(x+\delta e_1)-\left(\hat{s}(x)\hat{D}_\theta(x)-s^*(x){D^*_\theta}(x)\right)}{\delta}\nonumber\\
&=&\lim_{\delta\to0}\frac{\left(\nabla_\psi [s^*(x)D^*_\theta(x+\delta e_1)]-\nabla_\psi [s^*(x)D^*_\theta(x)]\right)^\top\xi'_n(\theta)}{\delta} \nonumber\\
&&+\frac{1}{2}\lim_{\delta\to0}\frac{{\xi'_n(\theta)}^\top\left(\nabla^2_\psi[\tilde{s}(x)\tilde{D}_{\theta}(x+\delta e_1)]-\nabla^2_\psi[\tilde{s}(x)\tilde{D}_{\theta}(x)]\right) \xi'_n(\theta)}{\delta} \nonumber\\
&=&\nabla_\psi\Big[s^*(x)\tfrac{\partial D^*_\theta(x)}{\partial x^1}\Big] \xi'_n(\theta)+\frac{1}{2}{\xi'_n(\theta)}^\top\nabla^2_\psi\Big[\tilde{s}(x)\tfrac{\partial \tilde{D}_{\theta}(x)}{\partial x^1}\Big] \xi'_n(\theta),\label{eq:grad_expand}
\end{eqnarray}
where the third equality comes from \eqref{eq:expa1}. Since $\nabla^2_\psi\Big[s(x;D_\psi)\tfrac{\partial D_{\psi}(x)}{\partial x^1}\Big]$ is continuous with respect to $\psi$ and $\Psi$ is compact, it is uniformly bounded for all $\psi\in\Psi$. Then the second term in \eqref{eq:grad_expand} is $O_p(\|\xi'_n(\theta)\|^2)=O_p(1/n)$ uniformly for all $\theta\in N(\theta^*)$. 
By applying the above calculations similarly to other components of $x$, we have
\begin{equation}\label{eq:d_grad_expand}
	\hat{s}(x)\nabla_x \hat{D}_\theta(x)-s^*(x)\nabla_x D^*_\theta(x)=\tilde\Sigma(x)\xi'_n(\theta)+O_p(1/n),
\end{equation}
where $\tilde\Sigma(x)=\nabla_\psi [s(x;D_\psi)\nabla_x D_\psi(x)]|_{\psi^*}$ and the higher order term is uniform for all $\theta\in N(\theta^*)$.

By (\ref{eq:d_grad_expand}) and recalling \eqref{eq:eps1_def}, we have
\begin{align}
	\hat\epsilon_{1}(\theta) &=-\frac{1}{m}\sum_{i=1}^m\left[\nabla_\theta G_\theta(z_i)^\top \left(\tilde\Sigma(G_\theta(z_i))\xi'_n(\theta)+O_p(1/n)\right)  \right] \nonumber \\
&= -\frac{1}{m}\sum_{i=1}^m\left[\nabla_\theta G_\theta(z_i)^\top\tilde\Sigma(G_\theta(z_i))\right]\xi'_n(\theta)+O_p(1/n) \nonumber\\
&= C_\theta {\xi'_n(\theta)}+O_p(1/n)\label{eq:eps1_in2}\\
&= \Sigma'(\theta){\xi_{n}^*}+O_p(1/n) ,\label{eq:eps1_rate}
\end{align}
where $$C_\theta:=-\bbE[\nabla_\theta G_\theta(Z)^\top\tilde\Sigma(G_\theta(Z))]=\nabla_\psi\bbE[h'_{D_{\psi}}(Z;\theta)]|_{\psi^*(\theta)},$$ $\Sigma'(\theta):=C_\theta[\Sigma_{d}(\theta)]^{1/2}$, $\xi_{n}^*:=[\Sigma_d(\theta)]^{-1/2}\xi'_n(\theta)$, and $\Sigma_d(\theta)$ is defined in Theorem~\ref{thm:d_asy_norm}.  
To obtain \eqref{eq:eps1_in2}, we know by applying the central limit theorem that 
$$\sqrt{m}\left(\frac{1}{m}\sum_{i=1}^m\left[\nabla_\theta G_\theta(z_i)^\top\tilde\Sigma(G_\theta(z_i))\right]-C_\theta\right)\dto\cN\left(0,\mathrm{Var}\big(\nabla_\theta G_\theta(Z)^\top\tilde\Sigma(G_\theta(Z))\big)\right).$$
Then by applying Lemma~\ref{lem:unif_order} to sequence $\frac{1}{m}\sum_{i=1}^m[\nabla_\theta G_\theta(z_i)^\top\tilde\Sigma(G_\theta(z_i))]$, we have 
$$\frac{1}{m}\sum_{i=1}^m\left[\nabla_\theta G_\theta(z_i)^\top\tilde\Sigma(G_\theta(z_i))\right]-C_\theta=O_p(1/\sqrt{m})$$
uniformly for all $\theta\in N(\theta^*)$. 
Since we have noted that $\xi'_n(\theta)=O_p(1/\sqrt{n})$ uniformly for all $\theta$, the higher order term $O_p(1/n)$ in \eqref{eq:eps1_in2} is uniform for all $\theta\in N(\theta^*)$.



Recall that 
\begin{equation*}
	\nabla L(\theta)=\bbE[\nabla l_g(Z;\theta)]=\bbE[h'(Z;\theta)],
\end{equation*}
\begin{equation*}
	\nabla\hat{L}(\theta)=\frac{1}{m}\sum_{i=1}^m\nabla l_g(z_i;\theta),\ h_{D^*}(\theta)=\frac{1}{m}\sum_{i=1}^m h'(z_i;\theta).
\end{equation*}
By the central limit theorem, we have as $n\to\infty$
\begin{equation}\label{eq:eps2_rate}
	\sqrt{n}\hat\epsilon_2(\theta)=\sqrt{\frac{m}{\lambda}}\hat\epsilon_2(\theta)\dto\cN(0,\Sigma_\tau(\theta)/\lambda),
\end{equation}
where $\Sigma_\tau(\theta)=\bbE[\tau\tau^\top]$ with $\tau=h'(Z;\theta)-\nabla l_g(Z;\theta)$. 
By applying Lemma~\ref{lem:unif_order} to sequence $\hat\epsilon_2(\theta)$, we have $\hat\epsilon_2(\theta)=O_p(1/\sqrt{n})$ uniformly for all $\theta\in N(\theta^*)$.


Combining (\ref{eq:eps1_rate}) and (\ref{eq:eps2_rate}), we have
\begin{equation*}
	\hat\epsilon(\theta^*)=\hat\epsilon_1(\theta^*)+\hat\epsilon_2(\theta^*)=\Sigma'(\theta^*)\xi^*_n+\hat\epsilon_2(\theta^*)+O_p(1/n),
\end{equation*}
which then, by (\ref{eq:bias}), leads to 
\begin{equation}\label{eq:o2}
\theta'-\tilde\theta=\xi_n+\tau_n+O_p(1/n),
\end{equation}
where the first leading term 
\begin{equation*}
	\xi_n=-\left[\nabla^2\hat{L}(\theta^*)\right]^{-1}\Sigma'(\theta^*)\xi_n^*
\end{equation*}
satisfies $\sqrt{n}\xi_n\dto\cN(0,H_g^{-1}C\Sigma_{d,\theta^*} C^\top H_g^{-1})$ as $n\to\infty$, by noticing that $\nabla^2\hat{L}(\theta^*)\pto H_g$ and the asymptotic behavior of $\xi^*_n$ from above and applying the Slutsky's theorem, where $C:=C_{\theta^*}$ and $\Sigma_{d,\theta^*}:=\Sigma_d(\theta^*)$, 
and the second leading term
\begin{equation*}
	\tau_n=-\left[\nabla^2\hat{L}(\theta^*)\right]^{-1}\hat\epsilon_2(\theta^*)
\end{equation*}
satisfies $\sqrt{n}\tau_n\dto\cN(0,H_g^{-1}\Sigma_\tau H_g^{-1}/\lambda)$ as $n\to\infty$ with $\Sigma_\tau=\Sigma_\tau(\theta^*)$, by noticing $\nabla^2\hat{L}(\theta^*)\pto H_g$ and the asymptotic behavior of $\hat\epsilon_2(\theta)$ from (\ref{eq:eps2_rate}) and applying the Slutsky's theorem.

\bigskip\noindent
\textit{Step III \ \ $\sqrt{n}$-rate of convergence of $\hat\theta-\theta^*$}

We apply Lemma~\ref{lem:agd} by taking $f(\theta)=\hat L'(\theta)$ defined in \eqref{eq:L'} and $\hat{g}(\theta)=h_{\hat{D}_\theta}(\theta)$. Note $\hat L'(\theta)$ is smooth and lower bounded. 
Let $\hat\delta(\theta)=h_{\hat{D}_\theta}(\theta)-\nabla\hat L'(\theta)$ and $\hat\delta=\sup_{\theta\in N(\theta^*)}\|\hat\delta(\theta)\|$. 
Observe that
\begin{equation}\label{eq:grad_diff_decom}
\begin{split}
	\hat\delta(\theta)&=h_{\hat{D}_\theta}(\theta)-\big[\nabla\hat{L}(\theta)+\hat\epsilon(\theta^*)\big]\\
	&=\big[h_{D^*_\theta}(\theta)+h_{\hat{D}_\theta}(\theta)-h_{D^*_\theta}(\theta)\big]-\big[\nabla \hat L(\theta)+\hat\epsilon_1(\theta^*)+\hat\epsilon_2(\theta^*)\big]\\
	&=\hat\epsilon_1(\theta)-\hat\epsilon_1(\theta^*)+\hat\epsilon_2(\theta)-\hat\epsilon_2(\theta^*)=O_p(1/\sqrt{n}),
\end{split}
\end{equation}
where the last equality hold uniformly for all $\theta\in N(\theta^*)$ because we have shown in \textit{Step II} that $\hat\epsilon_1(\theta)$ and $\hat\epsilon_2(\theta)$ are both $O_p(1/\sqrt{n})$ uniformly for all $\theta\in N(\theta^*)$. 
Thus we have $\hat\delta=O_p(1/\sqrt{n})$.
Then Lemma~\ref{lem:agd} implies 
\begin{equation*}
	\|\nabla \hat L'(\hat\theta)\|=O_p(1/\sqrt{n}).
\end{equation*}

Taking (\ref{eq:htheta}) $-$ (\ref{eq:theta'}) gives
\begin{equation}\label{eq:htheta-theta'}
	\nabla\hat{L}'(\hat\theta)=\nabla^2\hat{L}'(\theta^*)(\hat\theta-\theta')+O_p(\|\theta'-\theta^*\|^2+\|\hat\theta-\theta'\|^2),
\end{equation}
which further implies $\|\hat\theta-\theta'\|=O_p(1/\sqrt{n}).$
This, together with (\ref{eq:o1}) and (\ref{eq:o2}), implies
\begin{equation}\label{eq:rough_rate}
	\|\hat\theta-\theta^*\|=O_p(1/\sqrt{n}).
\end{equation}


\bigskip\noindent
\textit{Step IV \ \ Convergence rate of $\hat\theta-\theta'$}

We define a compact neighborhood $N'(\theta^*)=N(\theta)\cap\{\theta\in\Theta:\|\theta-\theta^*\|=O_p(1/\sqrt{n})\}$. 
By (\ref{eq:rough_rate}), we know $\hat\theta\in N'(\theta^*)$. 
On $N'(\theta^*)$, continuous function $\Sigma'(\theta)$ is Lipschitz continuous with respect to $\theta$ with Lipschitz constant $\ell_2$. Hence we have for all {$\theta\in N'(\theta^*)$}
\begin{equation*}
	\|\Sigma'(\theta)-\Sigma'(\theta^*)\|\leq \ell_2 \|\theta-\theta^*\| = O_p(1/\sqrt{n}),
\end{equation*}
which is a uniform rate free of $\theta$. 
Then, by recalling (\ref{eq:eps1_rate}) and noting that $\xi^*_n=O_p(1/\sqrt{n})$, both of which have been shown to hold uniformly for all $\theta\in N'(\theta^*)$, we further have
\begin{equation*}
	\hat\epsilon_1(\theta)={\Sigma'_{\theta^*}}\xi^*_n+O_p(1/n),
\end{equation*}
uniformly for all $\theta\in N'(\theta^*)$. Notably, the leading term in the above equation is free of $\theta$.

Thus, we have 
\begin{equation}\label{eq:grad_diff1}
	\hat\epsilon_1(\theta)-\hat\epsilon_1(\theta^*)=O_p(1/n)
\end{equation}
uniformly for all $\theta\in N'(\theta^*)$.

Similarly, we have 
\begin{equation*}
	\hat\epsilon_2(\theta)=\Sigma^{1/2}_\tau(\theta)\tau^*_n=\Sigma^{1/2}_\tau(\theta^*)\tau^*_n+O_p(1/n),
\end{equation*}
uniformly for all $\theta\in N'(\theta^*)$, where the leading term is again free of $\theta$.  
Thus, we have the following uniform rate for all $\theta\in N'(\theta^*)$,
\begin{equation}\label{eq:grad_diff2}
	\hat\epsilon_2(\theta)-\hat\epsilon_2(\theta^*)=O_p(1/n).
\end{equation}

We again apply Lemma~\ref{lem:agd} by taking $f(\theta)=\hat L'(\theta)$ and $\hat{g}(\theta)=h_{\hat{D}}(\theta)$. 
By the uniform rates of (\ref{eq:grad_diff1}) and (\ref{eq:grad_diff2}) and the decomposition in (\ref{eq:grad_diff_decom}), we have 
\begin{equation*}
	\hat\delta(\theta)=\hat\epsilon_1(\theta)-\hat\epsilon_1(\theta^*)+\hat\epsilon_2(\theta)-\hat\epsilon_2(\theta^*)=O_p(1/n),
\end{equation*}
uniformly for all $\theta\in N'(\theta^*)$ and thus $\hat\delta=\sup_{\theta\in N'(\theta^*)}\|\hat\delta(\theta)\|=O_p(1/n)$, which leads to the desired result \eqref{eq:grad_diff_n}. 
Then Lemma~\ref{lem:agd} implies $\|\nabla \hat L'(\hat\theta)\|=O_p(1/n)$. 

Then by recalling (\ref{eq:htheta-theta'}), we have 
\begin{equation*}
	\|\hat\theta-\theta'\|=O_p(1/n).
\end{equation*}
Therefore, the proof is completed.
\end{proof}

Now we are ready to prove Theorem~\ref{thm:asy_normal}.
\begin{proof}[Proof of Theorem~\ref{thm:asy_normal}]
By combining (\ref{eq:o1_lem}) and (\ref{eq:o2_lem}) in Lemma~\ref{lem:g_rate}, 
we have
\begin{equation}\label{eq:o12}
	\theta'-\theta^*=\gamma_n+\tau_n+\xi_n+O_p(1/n)=\zeta_n+\xi_n+O_p(1/n),
\end{equation}
where 
\begin{equation*}
	\zeta_n=\gamma_n+\tau_n=-\left[\nabla^2\hat{L}(\theta^*)\right]^{-1}h_{D^*_{\theta^*}}(\theta^*)
\end{equation*}
satisfies $\sqrt{n}\zeta_n\dto\cN(0,H_g^{-1}\Sigma_h H_g^{-1}/\lambda)$ with $\Sigma_h=\bbE[h'(Z;\theta^*)h'(Z;\theta^*)^\top]$.
(\ref{eq:o12}), together with (\ref{eq:o3_lem}) in Lemma~\ref{lem:g_rate}, leads to
\begin{equation}\label{eq:hattheta}
	\hat\theta-\theta^*= \zeta_n+\xi_n+O_p(1/n).
\end{equation}
By the definition of $\zeta_n$ and $\xi_n$ and the central limit theorem, we have the asymptotic normality 
\begin{equation*}
	\sqrt{n}(\hat\theta-\theta^*)\dto\cN(0,\mathbf\Sigma),
\end{equation*}
as $n\to\infty$, where the asymptotic variance is given by $\mathbf\Sigma=\mathrm{Var}(\zeta+\xi)$ with 
\begin{align*}
	\zeta&=-H_g^{-1} h'(Z;\theta^*)/\sqrt\lambda\\
	\xi&=H_g^{-1}CH_d^{-1}\left[\nabla_{\psi}l_1(X;\psi^*)+\nabla_{\psi}l_2(G_{\theta^*}(Z);\psi^*)/\sqrt\lambda\right].
\end{align*}
\end{proof}

\subsection{Proof of Theorem~\ref{thm:f_asy_normal}}\label{app:pf_f_asy_normal}
\begin{proof}
	We can show the asymptotic normality of $\gfgan$ by the arguments similar to those in the proof of Theorem~\ref{thm:asy_normal} where we replace $l_i$ by $l^f_i$ for $i=1,2$.
\end{proof}

\section{Proofs in Section~\ref{sec:misspecify}}

\subsection{Proof of Theorem~\ref{thm:d_asy_var_compare}}\label{app:pf_d_asy_var_compare}

We begin with a simple technical lemma.
\begin{lemma}\label{lem:sigma0}
	Let random vector $\tilde{X}=(1,X^\top)^\top$ and random variable $Y$. Suppose $\bbE(Y\tilde{X}\tilde{X}^\top)$ is positve definite. Then we have $$\bbE^{-1}(Y\tilde{X}\tilde{X}^\top)\bbE(Y\tilde{X})\bbE(Y\tilde{X})^\top\bbE^{-1}(Y\tilde{X}\tilde{X}^\top)=\Sigma_0,$$ where $\Sigma_0$ is a matrix whose first diagonal entry is 1 while all other entries are 0.
\end{lemma}
\begin{proof}[Proof of Lemma~\ref{lem:sigma0}]
Let $\bbE(Y)=\omega$, $\mu=\bbE(YX)$, and $V=\bbE(YXX^\top)$. We have
\begin{equation*}
	\bbE(Y\tilde{X}\tilde{X}^\top)=
	\begin{pmatrix}
		\omega & \omega\mu^\top\\
		\omega\mu & V
	\end{pmatrix},\ 
	\bbE(Y\tilde{X})\bbE(Y\tilde{X})^\top=
	\begin{pmatrix}
		\omega^2 & \omega\mu^\top\\
		\omega\mu & \mu\mu^\top
	\end{pmatrix}.
\end{equation*}
Let $\kappa=\omega-\mu^\top V^{-1}\mu$. Then, after some simplifications, we have
\begin{equation*}
	\bbE^{-1}(Y\tilde{X}\tilde{X}^\top)\bbE(Y\tilde{X})=
	\begin{pmatrix}
		\frac{1}{\kappa}\left(\omega-\mu^\top V^{-1}\mu\right)\\
		-\frac{\omega}{\kappa}V^{-1}\mu+\left(V-\mu\mu^\top/\omega\right)^{-1}\mu
	\end{pmatrix}=
	\begin{pmatrix}
		1\\\mathbf{0}
	\end{pmatrix},
\end{equation*}
where $\mathbf{0}$ denotes a zero vector of the same dimension as $X$. Then the desired result follows.
\end{proof}

\begin{proof}[Proof of Theorem~\ref{thm:d_asy_var_compare}]
We first explicitly compute the asymptotic variance of $\dage(\theta)$ and $\dfgan(\theta)$ for various $f$-divergences to obtain the results in Table~\ref{tab:var_d}. 

For AGE, we first compute the derivatives of the discriminator loss
\begin{equation*}
	\nabla_\psi l_1(x;\psi)=-\frac{1}{1+e^{D_\psi(x)}/\lambda}\nabla_\psi D_\psi(x),\ \nabla_\psi l_2(x;\psi)=\frac{e^{D_\psi(x)}}{1+e^{D_\psi(x)}/\lambda}\nabla_\psi D_\psi(x);
\end{equation*}
\begin{equation*}
	\nabla^2_\psi l_1(x;\psi)=-\frac{1}{1+e^{D_\psi(x)}/\lambda}\nabla_\psi^2D_\psi(x)+\frac{e^{D_\psi(x)}/\lambda}{(1+e^{D_\psi(x)}/\lambda)^2}\nabla_\psi D_\psi(x)\nabla_\psi D_\psi(x)^\top,
\end{equation*}
\begin{equation*}
	\nabla^2_\psi l_2(x;\psi)=\frac{e^{D(x)}}{1+e^{D(x)}/\lambda}\nabla_\psi^2D_\psi(x)+\frac{e^{D_\psi(x)}}{(1+e^{D_\psi(x)}/\lambda)^2}\nabla_\psi D_\psi(x)\nabla_\psi D_\psi(x)^\top.
\end{equation*}
Then 
\begin{equation*}
	V_d=\bbE_{p_*}\left[\frac{p_\theta}{p_\theta+p_*/\lambda}\nabla_\psi D_{\psi^*_\theta} \nabla_\psi D_{\psi^*_\theta}^\top\right]-\left(1+\frac{1}{\lambda}\right)\bbE_{p_*}\left[\frac{p_\theta}{p_\theta+p_*/\lambda}\nabla_\psi D_{\psi^*_\theta}\right]\bbE_{p_*}\left[\frac{p_\theta}{p_\theta+p_*/\lambda}\nabla_\psi D_{\psi^*_\theta}\right]^\top
\end{equation*}
and
\begin{equation*}
H_d=\bbE_{p_*}\left[\frac{p_\theta}{p_\theta+p_*/\lambda}\nabla_\psi D_{\psi^*_\theta} \nabla_\psi D_{\psi^*_\theta}^\top\right].
\end{equation*}

As supposed in the theorem, the first dimension of the parameter $\psi$ corresponds to the intercept, i.e., the first entry of $\nabla_\psi D_\psi(x)$ equals 1. By Lemma~\ref{lem:sigma0}, we have 
\begin{equation*}
	\bbE_{p_*}^{-1}\Big[\frac{p_\theta}{p_\theta+p_*/\lambda}{\nabla_\psi D_{\psi^*_\theta}}^{\otimes2}\Big] \left(\bbE_{p_*}\Big[\frac{p_\theta}{p_\theta+p_*/\lambda}\nabla_\psi D_{\psi^*_\theta}\Big]\right)^{\otimes2}\bbE_{p_*}\Big[\frac{p_\theta}{p_\theta+p_*/\lambda}{\nabla_\psi D_{\psi^*_\theta} }^{\otimes2}\Big]=\Sigma_0.
\end{equation*}

Then according to Theorem~\ref{thm:d_asy_norm}, the asymptotic variance of $\dage$ is given by 
\begin{equation*}
	\Sigma_d(\theta)=\bbE_{p_*}^{-1}\left[\frac{p_\theta}{p_\theta+p_*/\lambda}{\nabla_\psi D_{\psi^*}}^{\otimes2}\right]-\Big(1+\frac{1}{\lambda}\Big)\Sigma_0.
\end{equation*}

Next, we compute the asymptotic variances of the discriminator of $f$-GANs for various $f$-divergences. To lighten the notation, throughout the remainder of this proof, we denote $\nabla D_{\psi^*}=\nabla_\psi D_{\psi^*_\theta}$.

\medskip
\noindent
\textit{$f$-GAN-KL}

We recall from Table~\ref{tab:fgan_loss} that the discriminator loss of $f$-GAN-KL is
\begin{equation*}
	L_k(\psi,\theta)=\bbE_{p_*}[l_1^f(X;\psi)]+\bbE_{p_\theta}[l_2^f(X;\psi)]=\bbE_{p_*}[-D_\psi(X)]+\bbE_{p_\theta}[e^{D_\psi(X)}].
\end{equation*}
Then we have 
\begin{equation*}
	H_k=\bbE_{p_*}\left[\nabla D_{\psi^*}\nabla D_{\psi^*}^\top\right]
\end{equation*}
\begin{equation*}
	\mathrm{Var}_{p_*}(\nabla_\psi l^k_1(X;\psi^*))=\bbE_{p_*}\left[\nabla D_{\psi^*}\nabla D_{\psi^*}^\top\right]-\bbE_{p_*}[\nabla D_{\psi^*}]\bbE_{p_*}[\nabla D_{\psi^*}]^\top
\end{equation*}
\begin{equation*}
	\mathrm{Var}_{p_\theta}(\nabla_\psi l^k_2(X;\psi^*))=\bbE_{p_*}\left[\frac{p_*}{p_\theta}\nabla D_{\psi^*}\nabla D_{\psi^*}^\top\right]-\bbE_{p_*}[\nabla D_{\psi^*}]\bbE_{p_*}[\nabla D_{\psi^*}]^\top.
\end{equation*}
Hence the asymptotic variance of $\dfgan$ for KL divergence is given by
\begin{equation*}
\Sigma_k(\theta)=\bbE_{p_*}^{-1}\left[\nabla D_{\psi^*}\nabla D_{\psi^*}^\top\right] \bbE_{p_*}\left[\frac{p_\theta+p_*/\lambda}{p_\theta}\nabla D_{\psi^*}\nabla D_{\psi^*}^\top\right] \bbE_{p_*}^{-1}\left[\nabla D_{\psi^*}\nabla D_{\psi^*}^\top\right]-\Big(1+\frac{1}{\lambda}\Big)\Sigma_0.
\end{equation*}

Then we compare $\Sigma_k$ with $\Sigma$ of AGE. Let random variable $\tilde{p}=p_\theta/(p_\theta+p_*/\lambda)$. Consider random vectors $\eta_1=\nabla D_{\psi^*}/\sqrt{\tp}$ and $\eta_2=\nabla D_{\psi^*}\cdot\sqrt{\tp}$. Note that $\Sigma_{d1}=(\bbE[\eta_2\eta_2^\top])^{-1}$ and $\Sigma_{k1}=(\bbE[\eta_1\eta_2^\top])^{-1} \bbE[\eta_1\eta_1^\top](\bbE[\eta_1\eta_2^\top])^{-1}$, where $\Sigma_{d1}$ and $\Sigma_{k1}$ denote the first terms in $\Sigma_d$ and $\Sigma_k$.
Then by Cauchy-Schwartz inequality in the matrix form \cite{tripathi1999matrix}, we know $\Sigma_{d1}\preceq\Sigma_{k1}$, where the equality holds if and only if $\tp=c$. By noting the definition of $\tp$, this leads to $(1-c)p_\theta=cp_*/\lambda$. By integrating both sides, we have $c=\lambda/(1+\lambda)$, that is, $p_\theta=p_*$, which never happens due to model mis-specification in Assumption~\ref{ass:model_mis}. Therefore we have $\Sigma_d(\theta)\prec\Sigma_k(\theta)$ for all $\theta$.

\medskip
\noindent
\textit{$f$-GAN-RevKL}

We recall from Table~\ref{tab:fgan_loss} that the discriminator loss of $f$-GAN-RevKL is
\begin{equation*}
	L_r(\psi,\theta)=\bbE_{p_*}[e^{-D_\psi(X)}]+\bbE_{p_\theta}[D_\psi(X)].
\end{equation*}
Similarly, the asymptotic variance of $\dfgan$ for reverse KL divergence is given by
\begin{equation*}
\Sigma_r(\theta)=\bbE_{p_\theta}^{-1}\left[\nabla D_{\psi^*}\nabla D_{\psi^*}^\top\right] \bbE_{p_\theta}\left[\frac{p_\theta+p_*/\lambda}{p_*}\nabla D_{\psi^*}\nabla D_{\psi^*}^\top\right] \bbE_{p_\theta}^{-1}\left[\nabla D_{\psi^*}\nabla D_{\psi^*}^\top\right]-\Big(1+\frac{1}{\lambda}\Big)\Sigma_0.
\end{equation*}

Write $\Sigma_{d1}=\bbE^{-1}_{p_\theta}\Big[\frac{p_*}{p_\theta+p_*/\lambda}\nabla D_{\psi^*} \nabla D_{\psi^*}^\top\Big]$. Let $\Sigma_{r1}$ denote the first term of $\Sigma_r$. Then by applying Cauchy-Schwartz inequality as in the KL case, we have $\Sigma_{d1}\preceq\Sigma_{r1}$ and hence $\Sigma_d(\theta)\prec\Sigma_r(\theta)$ for all $\theta$.

\medskip
\noindent
\textit{$f$-GAN-JS}

We recall from Table~\ref{tab:fgan_loss} that the discriminator loss of $f$-GAN-JS is
\begin{equation*}
	L_j(\psi,\theta)=\bbE_{p_*}[\ln(1+e^{-D(X)})]+\bbE_{p_\theta}[\ln(1+e^{D(x)})].
\end{equation*}
Similarly, the asymptotic variance of $\dfgan$ for JS divergence is given by
\begin{equation*}
\Sigma_j(\theta)=\bbE_{p_\theta}^{-1}\bigg[\frac{p_\theta}{p_\theta+p_*}\nabla D_{\psi^*}\nabla D_{\psi^*}^\top\bigg] \bbE_{p_\theta}\bigg[\frac{p_\theta^2+p_*p_\theta/\lambda}{(p_\theta+p_*)^2}\nabla D_{\psi^*}\nabla D_{\psi^*}^\top\bigg] \bbE_{p_\theta}^{-1}\bigg[\frac{p_\theta}{p_\theta+p_*}\nabla D_{\psi^*}\nabla D_{\psi^*}^\top\bigg]-\Big(1+\frac{1}{\lambda}\Big)\Sigma_0.
\end{equation*}

Let $\eta_1=\frac{p_\theta}{p_\theta+p_*}\nabla D_{\psi^*}/\sqrt{\tp}$ and $\eta_2=\nabla D_{\psi^*}\cdot\sqrt{\tp}$. Then by applying Cauchy-Schwartz inequality similarly as above, we have $\Sigma_d(\theta)\prec\Sigma_{j}(\theta)$ for all $\theta$.

\medskip
\noindent
\textit{$f$-GAN-$H^2$}

We recall from Table~\ref{tab:fgan_loss} that the discriminator loss of $f$-GAN-$H^2$ is
\begin{equation*}
	L_h(\psi,\theta)=\bbE_{p_*}[e^{-D_\psi(X)/2}]+\bbE_{p_\theta}[e^{D_\psi(X)/2}].
\end{equation*}
Similarly, the asymptotic variance of $\dfgan$ for squared Hellinger distance is given by
\begin{equation*}
\Sigma_h(\theta)=\bbE_{p_*}^{-1}\left[\sqrt{\frac{p_\theta}{p_*}}\nabla D_{\psi^*}\nabla D_{\psi^*}^\top\right] \bbE_{p_*}\left[\frac{p_\theta+p_*/\lambda}{p_*}\nabla D_{\psi^*}\nabla D_{\psi^*}^\top\right] \bbE_{p_*}^{-1}\left[\sqrt{\frac{p_\theta}{p_*}}\nabla D_{\psi^*}\nabla D_{\psi^*}^\top\right]-\Big(1+\frac{1}{\lambda}\Big)\Sigma_0.
\end{equation*}

Recall $\tilde{p}=p_\theta/(p_\theta+p_*/\lambda)$. We have $\sqrt{p_\theta/p_*}=\sqrt{\tp/(1-\tp)}$ and $(p_*+p_\theta)/p_*=1/(1-\tp)$.
Now we let $\eta_1=\nabla D_{\psi^*}/\sqrt{1-\tp}$ and $\eta_2=\sqrt{\tp}\nabla D_{\psi^*}$. Then by applying Cauchy-Schwartz inequality similarly as above, we have $\Sigma_d(\theta)\prec\Sigma_{h}(\theta)$ for all $\theta$.
\end{proof}


\subsection{Proof of Proposition~\ref{prop:var_diff}}\label{app:pf_var_diff}
\begin{proof}
To lighten the notation, in this proof, we denote $d=\nabla_\psi D_{\psi^*}$.
For the KL divergence, we note that 
\begin{equation*}
	\Sigma_{k1}-\bbE_{p_*}^{-1}(dd^\top)=\bbE_{p_*}^{-1}(dd^\top)\bbE_{p_*}\bigg[\frac{p_*/\lambda}{p_\theta}dd^\top\bigg]\bbE_{p_*}^{-1}(dd^\top)=\mathbf\Theta(1/\lambda)
\end{equation*}
\begin{equation*}
	\Sigma_{d1}-\bbE_{p_*}^{-1}(dd^\top)=\bbE_{p_*}^{-1}\bigg[\frac{p_\theta}{p_\theta+p_*/\lambda}dd^\top\bigg]\bbE_{p_*}\bigg[\frac{p_*/\lambda}{p_\theta+p_*/\lambda}dd^\top\bigg]\bbE_{p_*}^{-1}(dd^\top)=\mathbf\Theta(1/\lambda),
\end{equation*}
which leads to $\|\Sigma_{k}(\theta)-\Sigma_d(\theta)\|=\mathbf\Theta(1/\lambda)$.

For the reverse KL divergence, we have 
\begin{equation*}
	\Sigma_{r1}-\frac{1}{\lambda}\bbE_{p_\theta}^{-1}(dd^\top)=\bbE_{p_\theta}^{-1}(dd^\top)\bbE_{p_\theta}\bigg[\frac{p_\theta}{p_*}dd^\top\bigg]\bbE_{p_\theta}^{-1}(dd^\top)=\mathbf\Theta(1)
\end{equation*}
\begin{equation*}
	\Sigma_{d1}-\frac{1}{\lambda}\bbE_{p_\theta}^{-1}(dd^\top)=\bbE_{p_\theta}^{-1}\bigg[\frac{p_*}{p_\theta+p_*/\lambda}dd^\top\bigg]\bbE_{p_\theta}\bigg[\frac{p_\theta}{p_\theta+p_*/\lambda}dd^\top\bigg]\bbE_{p_\theta}^{-1}(dd^\top)=\mathbf\Theta(1),
\end{equation*}
which leads to $\|\Sigma_{r}(\theta)-\Sigma_d(\theta)\|=\mathbf\Theta(1)$. Similarly, one can show $\|\Sigma_{j}(\theta)-\Sigma_d(\theta)\|=\mathbf\Theta(1)$ and $\|\Sigma_{h}(\theta)-\Sigma_d(\theta)\|=\mathbf\Theta(1)$.
\end{proof}

\section{Proofs in Section~\ref{sec:well_specify}}

\subsection{Proof of Corollary~\ref{cor:f_equiv}}\label{app:pf_f_equiv}
\begin{proof}[Proof of Corollary~\ref{cor:f_equiv}]
When the generated model is correctly specified, we have $p_*(x)=p_{\theta^*}(x)$, a.e. Then the optimal discriminator at the optimal generator is a constant function $D^*_{\theta^*}(x)=\ln (p_*(x)/p_{\theta^*}(x))=1$, a.e., which implies $\nabla_x D^*_{\theta^*}(x)=0$. Then $h'(z;\theta^*)=0$, a.e., so in the limiting distribution of $\hat\theta$ in Theorem~\ref{thm:asy_normal}, $\zeta$ vanishes with probability 1. Hence the asymptotic variance of $\hat\theta$ becomes $\mathbf\Sigma=\mathrm{Var}(\xi)=H_g^{-1}C\Sigma_{d,\theta^*} C^\top H_g^{-1}$. We have 
\begin{equation*}
	H_g=\bbE[\nabla^2 l_g(\theta^*)]=\bbE[\nabla_\theta h'(Z;\theta^*)]=\bbE\left[s^*(G_{\theta^*}(Z))\nabla_\theta[\nabla_\theta G_{\theta^*}(Z)^\top\nabla_x D^*_{\theta^*}(G_{\theta^*}(Z))]\right]
\end{equation*}
and 
\begin{equation*}
	C=\nabla_\psi\bbE[h'_{D_{\psi^*}}(Z,\theta^*)]=\bbE\left[s^*(G_{\theta^*}(Z))\nabla_\theta G_{\theta^*}(Z)^\top\nabla_\psi\nabla_x D_{\psi}(G_{\theta^*}(Z))|_{\psi^*(\theta^*)}\right].
\end{equation*}
The scaling factors for various $f$-divergences at the optimal generator, with $r=e^{D_{\theta^*}^*(x)}=1$, is listed in the following table. Then $H_g^{-1}C$ does not depend on the $f$-divergence used. 
\begin{table}[h]
\vskip -0.1in
\centering
\caption{Optimal scaling factors for various $f$-divergences.}
\vskip 0.1in
\begin{tabular}{ccccc}
\toprule
$f$-divergence & KL & RevKL & 2JS & $H^2$\\\midrule
$s^*$ & 1 & 1 & 1/2 & 1/2\\
\bottomrule
\end{tabular}
\end{table} 

From Table~\ref{tab:var_d}, considering $p_{\theta^*}=p_*$, all the asymptotic variances of the discriminator are simplified into 
\begin{equation*}
	\Sigma_{d,\theta^*}=\bigg(1+\frac{1}{\lambda}\bigg)\left(\bbE_{p_*}^{-1}\big[{\nabla_\psi D_{\psi^*}}^{\otimes2}\big]-\Sigma_0\right),
\end{equation*}
which again does not depend on the $f$-divergence used. Then by comparing $H_g,C,\Sigma_{d,\theta^*}$ with $H'_g,C',\Sigma'_{d,\theta^*}$ in the corollary statement, we complete the proof.
\end{proof}

\medskip
\revise{
Next, we provide details on the asymptotic equivalence of statistical inference using $f$-divergences from the perspective of information geometry~\cite{amari2000methods} and prove \eqref{eq:f_div_chi}. Consider any $f$-divergences defined in \eqref{eq:f_div} where $f$ is infinitely continuously differentiable and strongly convex. In the correctly specified case, we have $p_*(x)=p_{\theta^*}(x)$ almost everywhere and thus $D_f(p_*,p_{\theta^*})=0$. Then we have
\begin{align*}
	D_f(p_*,p_\theta)&=D_f(p_*,p_\theta)-D_f(p_*,p_{\theta^*})\\
	&\overset{(a)}=\int p_*(x)\left[f(p_\theta(x)/p_*(x))-f(1) \right]dx\\
	&\overset{(b)}=\int p_*(x)\sum_{i=0}^\infty\frac{1}{i!}f^{(i)}(1)\left(\frac{p_\theta(x)}{p_*(x)}-1\right)^i dx\\
	&\overset{(c)}=\sum_{i=0}^\infty \frac{1}{i!}f^{(i)}(1)\chi^i(p_*,p_\theta)\\
	&\overset{(d)}=\frac{f''(1)}{2}\chi^2(p_*,p_\theta)+\sum_{i=3}^\infty\frac{1}{i!}f^{(i)}(1)\chi^i(p_*,p_\theta),
\end{align*}
where (a) comes from the definition \eqref{eq:f_div} of $f$-divergences, (b) follows from the Taylor expansion of $f$, (c) is due to Fubini theorem and the definition of $\chi^i$-divergences in \eqref{eq:chi-div}, and (d) follows from the fact that $f(1)=0$ (after normalization) and $\chi^1(p_*,p_\theta)=0$.
}

\subsection{Proof of Corollary~\ref{cor:gaus}}\label{app:pf_cor_gaus}
\begin{proof}
Similar to the proof of Corollary~\ref{cor:f_equiv}, the asymptotic variance of $\hat\theta$ is $\mathbf\Sigma=H_g^{-1}C\Sigma_{d,\theta^*} C^\top H_g^{-1}$. 
Then it suffices to show $H_g^{-1}C\Sigma_{d,\theta^*} C^\top H_g^{-1}=(1+1/\lambda)\id_d$ in this case.

Under this scenario, we have $\nabla_\theta G_\theta(z)=\id_d$, $\nabla_x D_\psi(x)=\psi_1$, and $\nabla_\psi D_\psi(x)=(1,x^\top)^\top$. Then $h'_{D_\psi}(\theta)=-e^{D_\psi(G_\theta(z))}\psi_1$. Then we have $H_g=-\bbE_{p_*}[\nabla_\theta\ln p_\theta(X)]=\id_d$, $C=(\mathbf{0},\id_d)$, and
\begin{align*}
	\Sigma_{d,\theta^*}&=\bbE_{p_*}^{-1}\left[\frac{p_{\theta^*}}{p_{\theta^*}+p_*/\lambda}{\nabla_\psi D_{\psi^*}}^{\otimes2}\right]-\Big(1+\frac{1}{\lambda}\Big)\Sigma_0=\Big(1+\frac{1}{\lambda}\Big)\left[\bbE_{p_*}^{-1}
	\begin{pmatrix}
		1 & x^\top\\
		x & xx^\top
	\end{pmatrix}
	-\Sigma_0\right]\\
	&=\Big(1+\frac{1}{\lambda}\Big)\left[
	\begin{pmatrix}
		1 & m_0^\top\\
		m_0 & m_0m_0^\top+\id_d
	\end{pmatrix}^{-1}
	-\Sigma_0\right]=\Big(1+\frac{1}{\lambda}\Big)
	\begin{pmatrix}
		m_0^\top m_0 & -m_0^\top\\
		-m_0 & \id_d
	\end{pmatrix}.
\end{align*}
Multiplication of them gives $\mathbf\Sigma=(1+1/\lambda)\id_d$, which completes the proof.
\end{proof}

\subsection{Proof of Corollary~\ref{cor:localgan}}\label{app:pf_cor_localgan}
For the purpose of analysis, we write down an algorithm analogous to Algorithm~\ref{alg:localgan} in Algorithm~\ref{alg:localgan*}, where $S(\hat\theta;x)$ is replaced by $S(\theta^*;x)$ and the gradient estimator can be written as
\begin{equation}\label{eq:h_psi0}
	h^*_\psi(\theta)=\frac{1}{m}\sum_{i=1}^m \left[-\nabla_\theta S(\hat\theta;G_\theta(z_i))^\top\nabla_s D_\psi(S(\theta^*;G_\theta(z_i)))\right].
\end{equation}
For simplicity, let $\hat\theta_1$ be the output of Algorithm~\ref{alg:localgan} (denoted by $\hat\theta_{\text{local}}$ in the main text) and $\hat\theta_0$ be the output of Algorithm~\ref{alg:localgan*}, both with $T=\mathbf\Theta(n)$. 

{\centering
\begin{minipage}{.92\linewidth}
\vskip 0.1in
\begin{algorithm}[H]
\DontPrintSemicolon
\KwInput{Sample $\cS_n$, initial estimator $\hat\theta$, meta-parameter $T$}
Initialize $\theta_0=\hat\theta$\\
\For{$t=0,1,2,\dots,T$}{
True scores $s_i=S(\theta^*;x_i)$ for $i=1,\dots,n$\\
Generated scores $\hat{s}_i=S(\theta^*;G_{\theta_{t}}(z_i))$ for $i=1,\dots,m$\\
$\hat{\psi}_t=\argmin_{\psi\in\Psi}\big[\frac{1}{n}\sum_{i=1}^{n}\ln(1+e^{-D_\psi(s_i)}\lambda)+\frac{\lambda}{m}\sum_{i=1}^{m}\ln(1+e^{D_\psi(\hat{s}_i)}/\lambda)\big]$\\
$\theta_{t+1}=\theta_{t}-\eta {h}^*_{\hat{\psi}_t}(\theta_{t})$ for some $\eta>0$
}
\KwReturn{$\argmin_{\theta_t:t=1,\dots,T}\|h^*_{\hat{\psi}_t}(\theta_{t})\|$}
\caption{Local GAN with optimal score}
\label{alg:localgan*}
\end{algorithm}
\end{minipage}
\vskip 0.1in
\par
}

We first prove two lemmas which show that $\hat\theta_0$ and $\hat\theta_1$ share the same asymptotic distributions.
\begin{lemma}\label{lem:pf_localgan1}
	Under conditions D1-D4, as $n\to\infty$, we have 
\begin{align}
	\frac{1}{n}\sum_{i=1}^n S(\hat\theta;x_i)&=\frac{1}{m}\sum_{i=1}^m S(\hat\theta;G_{\hat\theta_1}(z_i))+O_p(1/n),\label{eq:local_theta1}\\
	\frac{1}{n}\sum_{i=1}^n S(\theta^*;x_i)&=\frac{1}{m}\sum_{i=1}^m S(\theta^*;G_{\hat\theta_0}(z_i))+O_p(1/n).\label{eq:local_theta0}
\end{align}
\end{lemma}

\begin{proof}[Proof of Lemma~\ref{lem:pf_localgan1}]
Below we prove (\ref{eq:local_theta0}), and (\ref{eq:local_theta1}) can be similarly proved.
We first show that $h^*_{\hat\psi_0}(\hat\theta_0)=O_p(1/n)$, where 

\begin{equation}\label{eq:psi_0}
	\hat{\psi}_0:=\argmin_{\psi}\left[\frac{1}{n}\sum_{i=1}^{n}\ln(1+e^{-D_\psi(S(\theta^*;x_i))}\lambda)+\frac{\lambda}{m}\sum_{i=1}^{m}\ln(1+e^{D_\psi(S(\theta^*;G_{\hat\theta_0}(z_i)))}/\lambda)\right].
\end{equation}

We apply Lemma~\ref{lem:agd} by taking $f(\theta)=\hat L'(\theta)$ and $\hat{g}(\theta)=h^*_{\hat{\psi}}(\theta)$. Note $\hat L'(\theta)$ is smooth and lower bounded. 
Let $\hat\delta(\theta)=h^*_{\hat{\psi}}(\theta)-\nabla\hat L'(\theta)$ and $\hat\delta=\sup_{\|\theta-\theta^*\|=O_p(1/\sqrt{n})}\|\hat\delta(\theta)\|$. 
According to \eqref{eq:grad_diff_n} in Lemma~\ref{lem:g_rate}, we have $\hat\delta=O_p(1/n)$. 
Then Lemma~\ref{lem:agd} implies $\|\nabla \hat L'(\hat\theta_0)\|=O_p(1/n)$ and $\|h^*_{\hat{\psi}}(\hat\theta_0)\|=O_p(1/n)$.

Then by noting that $\nabla_s D_\psi(s)=\psi$, the gradient estimator \eqref{eq:h_psi0} becomes
\begin{equation}\label{eq:h_psi0_Op}
	\frac{1}{m}\sum_{i=1}^m \left[-\nabla_\theta S(\hat\theta;G_\theta(z_i))^\top|_{\hat\theta_0}\right]\hat\psi_0=O_p(1/n).
\end{equation}
Condition~\ref{ass:localgan_ulln} ensures that $S(\theta';G_\theta(z))$ is continuous with respect to $(\theta',\theta)$ for each $z$, and that there exists an integrable function that uniformly dominates $S(\theta';G_\theta(z))$ for all $(\theta',\theta)$. Then by the uniform law of large numbers \cite[Theorem 2]{jennrich1969asymptotic}, we have as $n\to\infty$ (or equivalently $m\to\infty$ since $m=\lambda n$)
\begin{equation}\label{eq:score_grad_ulln}
	\sup_{\theta',\theta\in\Theta}\bigg\|\frac{1}{m}\sum_{i=1}^m \nabla_\theta S(\theta';G_\theta(z_i))-\bbE[\nabla_\theta S(\theta';G_\theta(Z))]\bigg\|\pto0,
\end{equation}
which implies
\[
\frac{1}{m}\sum_{i=1}^m \nabla_\theta S(\hat\theta;G_\theta(z_i))|_{\hat\theta_0} \pto \bbE\big[\nabla_\theta S(\hat\theta;G_\theta(Z))|_{\hat\theta_0}\big].
\]
Therefore, \eqref{eq:h_psi0_Op} implies $\hat\psi_0=O_p(1/n)$. 

Also, by setting the gradient of the objective in \eqref{eq:psi_0} with respect to $\psi=\hat\psi_0$ to 0 and noting that $\nabla_\psi D_\psi(s)=s$,
we have
\begin{equation*}
	-\frac{1}{n}\sum_{i=1}^{n}\frac{1}{1+e^{D_{\hat\psi_0}(S(\theta^*;x_i))}/\lambda}S(\theta^*;x_i)+\frac{1}{m}\sum_{i=1}^{m}\frac{e^{D_{\hat\psi_0}(S(\theta^*;G_{\hat\theta_0}(z_i)))}}{1+e^{D_{\hat\psi_0}(S(\theta^*;G_{\hat\theta_0}(z_i)))}/\lambda}S(\theta^*;G_{\hat\theta_0}(z_i)) =0.
\end{equation*}
Further from the Taylor expansions
\begin{align*}
	\frac{1}{1+e^{D_\psi(s)}/\lambda}&=\frac{1}{1+1/\lambda}-\frac{1/\lambda}{(1+1/\lambda)^2}\psi^\top s+o(\psi^\top s)\\
	\frac{1}{e^{-D_\psi(s)}+1/\lambda}&=\frac{1}{1+1/\lambda}+\frac{1}{(1+1/\lambda)^2}\psi^\top s+o(\psi^\top s),
\end{align*}
we obtain
\begin{equation*}
	-\frac{1}{n}\sum_{i=1}^n S(\theta^*;x_i)+\frac{1}{m}\sum_{i=1}^m S(\theta^*;G_{\hat\theta_0}(z_i))+O_p(1/n)=0,
\end{equation*}
which leads to (\ref{eq:local_theta0}).
\end{proof}

\begin{lemma}\label{lem:pf_localgan2}
	Under conditions D1-D4, as $n\to\infty$, we have $\hat\theta_1-\hat\theta_0=O_p(1/n)$. 
\end{lemma}

\begin{proof}[Proof of Lemma~\ref{lem:pf_localgan2}]
According to Lemma~\ref{lem:pf_localgan1}, (\ref{eq:local_theta1})$-$(\ref{eq:local_theta0}) gives
\begin{equation*}
	\frac{1}{n}\sum_{i=1}^n S(\hat\theta;x_i)-\frac{1}{n}\sum_{i=1}^n S(\theta^*;x_i)=\frac{1}{m}\sum_{i=1}^m S(\hat\theta;G_{\hat\theta_1}(z_i))-\frac{1}{m}\sum_{i=1}^m S(\theta^*;G_{\hat\theta_0}(z_i))+O_p(1/n).
\end{equation*}
Since $\|\hat\theta-\theta^*\|=O_p(1/\sqrt{n})$, we have the following expansions
\begin{align*}
	\frac{1}{n}\sum_{i=1}^n S(\hat\theta;x_i)-\frac{1}{n}\sum_{i=1}^n S(\theta^*;x_i)&=\frac{1}{n}\sum_{i=1}^n H'(\theta^*;x_i) (\hat\theta-\theta^*)+O_p(1/n)\\
	\frac{1}{m}\sum_{i=1}^m S(\hat\theta;G_{\hat\theta_1}(z_i))-\frac{1}{m}\sum_{i=1}^m S(\theta^*;G_{\hat\theta_1}(z_i))&=\frac{1}{m}\sum_{i=1}^m H'(\theta^*;G_{\hat\theta_1}(z_i))(\hat\theta-\theta^*)+O_p(1/n)\\
	\frac{1}{m}\sum_{i=1}^m S(\theta^*;G_{\hat\theta_1}(z_i))-\frac{1}{m}\sum_{i=1}^m S(\theta^*;G_{\hat\theta_0}(z_i))&=\frac{1}{m}\sum_{i=1}^m \nabla_\theta S(\theta^*;G_{\theta}(z_i))|_{\hat\theta_0}(\hat\theta_1-\hat\theta_0)+O_p(1/n),
\end{align*}
where $H'(\theta;x):=\nabla_\theta S(\theta;x)$. 
Then we have
\begin{equation*}
	\hat\theta_1-\hat\theta_0=\left[\frac{1}{m}\sum_{i=1}^m \nabla_\theta S(\theta^*;G_{\theta}(z_i))|_{\hat\theta_0}\right]^{-1}\left[\frac{1}{n}\sum_{i=1}^n H'(\theta^*;x_i)-\frac{1}{m}\sum_{i=1}^m H'(\theta^*;G_{\hat\theta_1}(z_i))\right](\hat\theta-\theta^*)+O_p(1/n),
\end{equation*}
where we notice  
\begin{equation*}
	\left[\frac{1}{m}\sum_{i=1}^m \nabla_\theta S(\theta^*;G_{\theta}(z_i))|_{\hat\theta_0}\right]^{-1}=O_p(1)
\end{equation*}
by the uniform law of large numbers in \eqref{eq:score_grad_ulln}, 
\begin{equation*}
	\frac{1}{n}\sum_{i=1}^n H'(\theta^*;x_i)=H(\theta^*)+O_p(1/\sqrt{n})
\end{equation*}
by the central limit theorem, where $H(\theta):=\bbE_{p_*}[H'(\theta;X)]=\bbE_{p_*}[\nabla^2_\theta \ln p_\theta(x)]$, and
\begin{align*}
	\frac{1}{m}\sum_{i=1}^m H'(\theta^*;G_{\hat\theta_1}(z_i))&=\frac{1}{m}\sum_{i=1}^m H'(\theta^*;G_{\theta^*}(z_i))+O_p(\hat\theta_1-\theta^*)\\
	&=H(\theta^*)+O_p(1/\sqrt{m})+O_p(1/\sqrt{n}),
\end{align*}
where the first equality follows from Taylor expansion and the second equality comes from the central limit theorem and the fact that $\hat\theta_1-\theta^*=O_p(1/\sqrt{n})$.
Therefore, we have $\hat\theta_1-\hat\theta_0=O_p(1/n).$
\end{proof}

\medskip
\begin{proof}[Proof of Corollary~\ref{cor:localgan}]
To simplify the notation, throughout this proof, we let $\cI=\cI(\theta^*)$. 
According to Theorem~\ref{thm:asy_normal}, Corollary~\ref{cor:f_equiv}, and Lemma~\ref{lem:pf_localgan2}, it suffices to show $H_g^{-1}C\Sigma_{d,\theta^*} C^\top H_g^{-1}=(1+1/\lambda)\cI$ for $\hat\theta_0$. 
We know that given $\theta$, the optimal discriminator in Algorithm~\ref{alg:localgan*} is given by $D^*(s)=(\theta^*-\theta)^\top s$ up to higher order terms which do not affect the asymptotic variance. 
Then we have $\nabla_s D_{\psi}(s)=\psi$, $\nabla_\psi D_\psi(s)=s$, and $$h'_{D_\psi}(\theta)=-\nabla_\theta S(\theta^*;G_\theta(Z))^\top \nabla_s D_\psi(S(\theta^*;G_\theta(Z))).$$
Then we compute $H_g$, $C$ and $\Sigma_{d,\theta^*}$. It is easy to see $H_g=-\bbE_{p_*}[\nabla^2_\theta\ln p_\theta(X)]=\cI$. Let $p_*(s)$ be the distribution of $S(\theta^*;X)$ with $X\sim p_*(x)$. Then by noting that $\bbE_{p_*(s)}[S]=0$ and $\bbE_{p_*(s)}[SS^\top]=\cI^{-1}$ and that the discriminator does not have an intercept, we have
\begin{equation*}
	\Sigma_{d,\theta^*}=(1+1/\lambda)\bbE_{p_*(s)}\big[\nabla_\psi D_{\psi^*}(S)\nabla_\psi D_{\psi^*}(S)^\top\big]=(1+1/\lambda)\cI^{-1}.
\end{equation*}
Also, we notice that
\begin{align*}
	\bbE_{p_z(z)}\Big[\nabla_\theta S(\theta^*;G_\theta(Z))^\top\Big]\Big|_{\theta^*}&=\nabla_\theta \bbE_{p_\theta(x)}\Big[S(\theta^*;X)^\top\Big]\Big|_{\theta^*}\\
	&=\int\big[\nabla_\theta p_{\theta^*}(x)\big] S(\theta^*;x)^\top dx\\
	&=\int p_*(x)\big[\nabla_\theta \ln p_{\theta^*}(x)\big] S(\theta^*;x)^\top dx\\
	&= \bbE_{p_*(x)}\Big[S(\theta^*;X)S(\theta^*;X)^\top\Big]=\cI,
\end{align*}
where the third equality is due to the fact that $p_*(x)=p_{\theta^*}(x)$. Then we have
\begin{equation*}
	C=-\bbE_{p_z(z)}\Big[\nabla_\theta S(\theta^*;G_\theta(Z))^\top\Big]\Big|_{\theta^*}=-\cI.
\end{equation*}
Hence, we immediately know the asymptotic variance of $\hat\theta_0$ and also $\hat\theta_1$ is given by $$H_g^{-1}C\Sigma_{d,\theta^*} C^\top H_g^{-1}=(1+1/\lambda)\cI,$$ which completes the proof.
\end{proof}

\section{Discussion on the regularity assumptions}
\label{app:ass_eg}

In this section, we give a concrete example to illustrate the satisfaction of the regularity assumptions, as mentioned in Remark~\ref{rem:ass_eg}. We restate the example here. Given a multivariate Gaussian as the real distribution on $\cX=\bbR^d$, i.e., $p_*(x)=\cN(\mu_0,\Sigma_0)$ where $\|\mu_0\|<\infty$ and $0<\|\Sigma_0\|<\infty$. Consider the generative model and discriminator class as follows:
\begin{itemize}
\setlength{\itemsep}{2pt}
\item Latent prior $p_z=\cN(0,\id_d)$ on $\cZ$.
\item Linear generator class $\cG=\{G_\theta(z)=\theta_0+\theta_1z:\theta_0\in\bbR^{d},\theta_1\in\bbR^{d\times d},\|\theta_0\|\leq g_0,0<g_1\leq\|\theta_1\|\leq g_2\}$. Then $p_\theta(x)=\cN(\theta_0,\theta_1\theta_1^\top)$ with $\theta=(\theta_0,\theta_1)$.
\item Quadratic discriminator class $\cD=\{\psi_0+\psi_1^\top x+x^\top \psi_2 x: \psi_0\in\bbR, \psi_1\in\bbR^d, \psi_2\in\bbR^{d\times d}, 0<a_0\leq\|\psi_2\|\leq a_1,\|\psi_1\|\leq a_2,|\psi_0|\leq a_3\}$, where $a_i$'s depend on $g_i$'s, $\|\mu_0\|$, and $\|\Sigma_0\|$.
\end{itemize}

Then the true and generated densities are given by
\begin{equation}\label{eq:eg_dist}
\begin{split}
	p_*(x)&=(2\pi)^{-\frac{d}{2}}[\det(\Sigma_0)]^{-\frac{1}{2}} e^{-\frac{1}{2}(x-\mu_0)^\top\Sigma_0^{-1} (x-\mu_0)}\\
	p_\theta(x)&=(2\pi)^{-\frac{d}{2}}[\det(\theta_1\theta_1^\top)]^{-\frac{1}{2}} e^{-\frac{1}{2}(x-\theta_0)^\top(\theta_1\theta_1^\top)^{-1} (x-\theta_0)}.
\end{split}	
\end{equation}
Then the optimal discriminator is given by
\begin{equation}\label{eq:eg_D*}
	D^*_{\theta}(x)=\ln(p_*(x)/p_\theta(x))=\psi_0^*+{\psi_1^*}^\top x+x^\top \psi_2^* x,
\end{equation}
where $\psi_0^*\in\bbR$, $\psi_1^*\in\bbR^{d}$, and $\psi_2^*\in\bbR^{d\times d}$ depend on parameters $\theta$.

We now verify the satisfaction of the assumptions on this example. Note that for each $D\in\cD$, we have $\nabla_x D(x)=\psi_1+2\psi_2 x$ and $\nabla_x^2 D(x)=2\psi_2$. According to (\ref{eq:eg_dist}), both $p_*$ and $p_\theta$ are multivariate Gaussians  whose covariances are bounded away from 0 and infinity. From (\ref{eq:eg_D*}) and the choice of $\cD$ we know that conditions~\ref{ass:D_smooth}-\ref{ass:d_compact} hold. Because the parameters are bounded, condition~\ref{ass:D_bound} holds. 

Then we verify the uniform Lipschitz continuity of $\cD$ on any compact subset $K$ of $\cX$. Note there exists $B>0$ such that for all $x\in K$, we have $\|x\|\leq B$. Then for all $D\in\cD$ and  $x,x'\in K$,
\begin{align*}
	\|D(x)-D(x')\|&=\|(\psi_0+\psi_1^\top x+x^\top \psi_2 x)-(\psi_0+\psi_1^\top x'+{x'}^\top \psi_2 x')\|\\
	&\leq \|x^\top \psi_2 x-x^\top \psi_2 x'\|+\|x^\top \psi_2 x'-{x'}^\top \psi_2 x'\|+\|\psi_1^\top(x-x')\|\\
	&\leq (\|\psi_2x\|+\|\psi_2x'\|+\|\psi_1\|)\|x-x'\| \\
	&\leq (2a_1B+a_2)\|x-x'\|.
\end{align*}
The uniform Lipschitz continuity of $\{\nabla D:D\in\cD\}$ can be similarly obtained. Thus, condition~\ref{ass:D_cont} holds.

For the envelope conditions, note that
\begin{equation*}
	\sup_{D\in\cD}|D(x)|^2=\sup_{\psi_0,\psi_1,\psi_2}\big|\psi_0 + \psi_1^\top x + x^\top \psi_2 x\big|^2\leq c_1\|x\|^4+c_2\|x\|^3+c_3\|x\|^2+c_4\|x\|+c_5,
\end{equation*}
\begin{equation*}
	\sup_{D\in\cD}\|\nabla D(x)\|^2=\sup_{\psi_1,\psi_2}\|\psi_1+2\psi_2x\|^2\leq c_6\|x\|^2+c_7\|x\|+c_8,
\end{equation*}
where $c_i$'s are absolute constants. 
Since Gaussian distributions have finite forth moments, we know condition~\ref{ass:envelope} holds. 

We have $\nabla_{\theta_0} G_\theta(z)=\id$ and $\nabla_{[\theta_1]_{ij}} G_\theta(z)=(0,\dots,z_j,\dots,0)^\top$ with only the $i$th element being nonzero. Note that $\bbE Z_j^2=1$. Thus we have $\bbE\|\nabla_\theta G_\theta(Z)\|^2$ is uniformly bounded for all $\theta\in\Theta$, which verifies condition~\ref{ass:G_smooth}.

Note $D(G_\theta(z))=\psi+\psi_1^\top(\theta_0+\theta_1 z)+(\theta_0+\theta_1 z)^\top \psi_2(\theta_0+\theta_1 z)$ and $\Psi$ and $\Theta$ are compact. Then we have 
\begin{align*}
	\bbE\sup_{\theta\in\Theta,D\in\cD}|D(G_\theta(Z))|\leq \sup_{\theta\in\Theta,D\in\cD}\bbE[c_9\|z\|^2+c_{10}\|z\|+c_{11}]<\infty,
\end{align*}
where $c_i$'s are absolute constant. This verifies the second part of condition~\ref{ass:emp_envelop}.

For the first part of condition~\ref{ass:emp_envelop} and condition~\ref{ass:pl}, we take the KL divergence as an example. 
Then \eqref{eq:h'_D} becomes
\begin{equation*}
	h'(z;\theta)=e^{D^*_\theta(G_\theta(z))}\nabla_\theta G_\theta(z)^\top \nabla_x D^*_\theta(G_\theta(z)).
\end{equation*}
Then we have
\begin{eqnarray*}
	&&\bbE_{Z\sim p_z}\left[\sup_{\theta\in\Theta,D\in\cD}\frac{p_*(G_\theta(Z))}{p_\theta(G_\theta(Z))}\big\|\nabla_\theta G_\theta(Z)^\top \nabla_x D(G_\theta(Z))\big\|\right]\\
	&=& \int \sup_{\theta\in\Theta,D\in\cD} p_z(z)\frac{p_*(G_\theta(z))}{p_\theta(G_\theta(z))}\big\|\nabla_\theta G_\theta(z)^\top \nabla_x D(G_\theta(z))\big\|dz\\
	&=& \int \sup_{\theta\in\Theta,D\in\cD} p_\theta(x)\frac{p_*(x)}{p_\theta(x)}\big\|g_\theta(x)^\top \nabla_x D(x)\big\|dx\\
	&=& \bbE_{X\sim p_*(x)}\left[\sup_{\theta\in\Theta,D\in\cD}\big\|g_\theta(X)^\top \nabla_x D(X)\big\|\right] \\
	&\leq& \bbE_{X\sim p_*(x)}\left[c_{12}\|x\|^2+c_{13}\|x\|+c_{14}\right]<\infty,
\end{eqnarray*}
where the second equality is from reparametrization $X=G_\theta(Z)$, the first inequality is due to the linearity of $g_\theta(x)$ and $\nabla D(x)$ and bounded parameter space, the last inequality is because Gaussian distributions have finite second moments, and $c_i$'s are absolute constant. Hence condition~\ref{ass:emp_envelop} holds.

The KL objective function is
\begin{align*}
	L(\theta)&=-\bbE_{p_*(x)}[\ln p_\theta(X)]+\bbE_{p_*(x)}[\ln p_*(X)]\\
	&=\bbE_{p_*(x)}\left[\frac{1}{2}(X-\theta_0)^\top(\theta_1\theta_1^\top)^{-1}(X-\theta_0)+\ln\det(\theta_1\theta_1^\top)\right]+\text{const},
\end{align*}
whose smoothness is easy to see. 

Consider reparametrization $\theta_1'=(\theta_1\theta_1^\top)^{-1}$ with $\theta_1$ and $\theta_1'$ is one-to-one. and rewrite the objective as
\begin{equation*}
	L(\theta_0,\theta_1')=\bbE_{p_*(x)}\left[\frac{1}{2}(X-\theta_0)^\top\theta_1'(X-\theta_0)-\ln\det(\theta_1')\right],
\end{equation*}
which is strongly convex in both $\theta_0$ and $\theta_1'$ and hence satisfies the weaker PL condition. The PL condition remains satisfied in the original parameter $\theta$, which verifies condition~\ref{ass:pl}.


\medskip
The regularity conditions for the consistency of $f$-GAN in Assumption~\ref{ass:fgan_cons} are conditions \textit{A1-A10} and the following uniform convergence for all $l^i_f$ listed in Table~\ref{tab:fgan_loss}:
\begin{equation*}
	\sup_{\theta\in\Theta,D\in\cD}\left|-\frac{1}{n}\sum_{i=1}^nl_1^f(x_i;D)-\frac{1}{m}\sum_{i=1}^ml_2^f(G_\theta(z_i);D)+\bbE[l^f_1(X;D)]+\bbE[l^f_2(G_\theta(Z);D)]\right|\pto0.
\end{equation*}

{\section{Discussion on the connections between GAN and NCE}}\label{app:nce}

\subsection{Connections and differences}
Noise-contrastive estimation (NCE)~\cite{Gutmann2012NoiseContrastiveEO} is a method for parametric density estimation. We first introduce the \emph{problem setup} and compare it with that of generative models studied in this paper. As in the setup of generative models, we assume an i.i.d. sample $x_1,\dots,x_n$ of a random variable $X\sim p_*$ and the unknown true data density $p_*$ is modeled by a parametrized family $\{p_\theta(x):\theta\in\Theta\}$ where $\theta$ is a vector of parameters. Here the parameter $\theta$ equates to the parameter of the generator in our work, since the generator, parametrized as $G_\theta(\cdot)$ induces the parametric density $p_\theta(x)$. In density estimation, the parametrized densities $p_\theta(x)$ needs to be explicitly computed given any $\theta$ and $x$. However, it may require nontrivial computational techniques like MCMC to sample new data from it. In contrast, in generative models, as we have mentioned in the introduction, the generator $G_\theta(\cdot)$ is often parametrized by neural networks, making the induced density implicit, i.e., we cannot write down the closed-form density of $p_\theta(x)$, while the generator readily provides a way to sample new data. This is the first difference in the problem setups of NCE and GAN. Second, NCE is tailored to unnormalized models $p^0_{\theta_0}(x)$ with parameter denoted by $\theta_0$, that is, $C(\theta):=\int p^0_{\theta_0}(x)dx$ may not be equal to 1; otherwise, standard methods like maximum likelihood estimate would be the first choice when $p_\theta(x)$ is normalized and explicit. In contrast, in generative models, we often have normalized models where the induced probability density satisfies $\int p_\theta(x)dx=1$ by definition. Due to the above two differences in the problem setup, NCE in general cannot be applied to generative models to learn a generator. 


Next, we describe the \emph{formulation} of the NCE method using our notations. 
The idea to handle the unnormalized model is to consider $c=\ln1/C$ as an additional parameter of the model and extend the unnormalized model $p^0_{\theta_0}(x)$ to include a normalizing parameter $c$ and estimate 
\begin{equation}\label{eq:unnormalized_model}
	\ln p_\theta(x)=\ln p^0_{\theta_0}(x)+c,
\end{equation}
with parameter $\theta=(c,\theta_0)$. 
The main idea of NCE is to describe the properties of $p_*$ through contrasting it to some reference (noise) distribution $p_n$ which is prespecified. 
 
Suppose we have an i.i.d. sample $\tilde{x}_1,\dots,\tilde{x}_m$ from the noise distribution $p_n$, where the sample size $m=\lambda n$ for $\lambda>1$. The comparison between the two data sets is performed via classification. 
To implement this, NCE considers the following parametrized discriminator class with the same parameter to the parametrized model \eqref{eq:unnormalized_model}
\begin{equation}\label{eq:nce_dis}
	\{D_\theta(x)=\ln p_\theta(x)-\ln p_n(x):\theta\in\Theta\}.
\end{equation}
Then NCE learns the parameter through the logistic regression
\begin{equation}\label{eq:nce_est}
	\hat\theta_{\rm NCE}=\argmin_\theta \left[ \frac{1}{n}\sum_{i=1}^{n}\ln(1+e^{-D_\theta(x_i)}\lambda)+\frac{\lambda}{m}\sum_{i=1}^{m}\ln(1+e^{D_\theta(\tilde{x}_i)}/\lambda)\right].
\end{equation}

Now let us compare the NCE method and the GAN/AGE method introduced in Section~\ref{sec:method}. 
Both NCE and AGE ($f$-GAN uses other classification losses) adopt logistic regression for density ratio estimation, which utilizes the well-known result in statistics. However they use it in different ways: 
\begin{enumerate}[label=(\roman*)]
\setlength{\itemsep}{2pt}
\setlength{\parskip}{2pt}
\item The parametrization of the discriminator is different: the discriminator class of NCE is specialized as \eqref{eq:nce_dis} which shares the parameter $\theta$ with the model and requires densities $p_\theta$ and $p_n$ to be evaluated explicitly; in comparison, GAN uses general parametrization separated from the generator, e.g., $\{D_\psi(x):\psi\in\Psi\}$ with a different set of parameters, and allows implicit densities induced by complex generators, which is often the case in applications of generative models. 
\item The two classes to be distinguished in classification through logistic regression are different: $p_*$ versus $p_n$ in NCE; $p_*$ versus $p_\theta$ in GAN. Note that the role of the noise distribution $p_n$ in NCE is different from that of the generated distribution $p_\theta$ in GAN in that the noise distribution is kept fixed throughout training while the generated distribution $p_\theta$ keeps updating. 
\item The purposes of classification are different: the classification directly results in the NCE estimator by \eqref{eq:nce_est}, while GAN needs subsequent steps to learn the generator that minimizes a certain $f$-divergence based on the discriminator estimation. 
\end{enumerate}

Moreover, the NCE formulation \eqref{eq:nce_dis} becomes
\begin{equation*}
	\min_\theta \left[\bbE_{p_*(x)}[\ln(1+e^{-D_\theta(X)}\lambda)]+\lambda\bbE_{p_n(x)}[\ln(1+e^{D_\theta({X})}/\lambda)]\right],
\end{equation*}
{from which it is not clear what population loss NCE is optimizing.} This is in contrast to generative models where the goal at the population level is to minimize an $f$-divergence:
\begin{equation*}
	\min_\theta D_f(p_*, p_\theta).
\end{equation*}

In addition, we discuss the connections between the \emph{asymptotic analysis} in this paper and previous work on NCE \cite{Gutmann2012NoiseContrastiveEO}. As clarified above, NCE is essentially a classification problem and the asymptotic analysis of the NCE estimator only involves analysis of an empirical logistic regression with a specialized parametrization \eqref{eq:nce_dis}. The analysis in our paper regarding the discriminator estimation is similar in that they both involves logistic regression. However, our discriminator analysis considers more general discriminator parametrization and other classification losses used in various $f$-GANs with comparison of their asymptotic variances. Moreover, the analysis of the discriminator mainly serves as the basis for the analysis of the generator estimation, where the latter is the goal of generative models and the main focus of this paper. As we have pointed out, the subsequent procedure for generator estimation is not relevant to NCE, and neither is the analysis. Also, \cite{Gutmann2012NoiseContrastiveEO} only considered the correctly specified case where $p_*\in\{p_\theta\}$ while a large part of our paper focuses on the more general case with model misspecification (e.g., in Section~\ref{sec:misspecify}; Section~\ref{sec:theory} does not assume the generative model is correctly specified). 

Apart from the asymptotic distribution, \cite[Corollary 6]{Gutmann2012NoiseContrastiveEO} showed that NCE asymptotically attains the Cram\'er-Rao lower bound as the number of noise samples goes to infinity (i.e., $\lambda\to\infty$). At a high level, this result shares the similar spirit as our Corollary~\ref{cor:localgan} in that both NCE and GAN can be asymptotically efficient when we sample infinitely many artificial data. However, the difference lies in two aspects. First, the artificial data in NCE are sampled from a known, fixed noise distribution with an explicit density $p_n$ and the noise distribution will not be learned during training; the artificial data in GAN are sampled from the generator which is updated throughout the training algorithm and its distribution $p_\theta$ does not have an explicit form in general. Second, the asymptotic efficiency of NCE is directly attained by simply taking the limit of the asymptotic variance as $\lambda\to\infty$, while the asymptotic efficiency of GAN requires much more efforts such as the design of the local GAN algorithm and the construction of the linear discriminator class based on Fisher score features. Besides, as far as we notice, there is no essential relationship between the derivation of the two results on asymptotic efficiency in that one cannot be reduced to the other, e.g., by changing the model in some way. This is because the the asymptotic distribution of local GAN is obtained as a corollary of the asymptotic distribution of generator estimation of AGE in Theorem~\ref{thm:asy_normal} which is irrelevant to NCE analysis as explained in the previous paragraph.

\subsection{An illustrative example}
To better illustrate the connections and differences in the asymptotics of GANs and NCE, in the following we provide an example of normalizing flows, which is another important class of generative models where the transformation is invertible and hence induces explicit densities. As we mentioned above, the generated density is normalized, so in principle maximum likelihood estimate applies in this case and there is no need of applying GAN or NCE. Nevertheless, just for illustrative comparison, let us look at the asymptotic distributions of GAN and NCE estimators in this case. 
\begin{example}[Normalizing flows]
Normalizing flows \cite{Papamakarios2021NormalizingFF} consider a transformation, which is a diffeomorphism, from the latent variable $Z\in\bbR^d$ to $X\in\bbR^d$, denoted by $X=G_\theta(Z)$ with $\theta$ being the parameters. Note that for $G(\cdot)$ to be a diffeomorphism, $Z$ is required to have the same dimension to $X$. Then the density of generated data can be obtained by a change of variables 
\begin{equation*}
	p_\theta(x)=p_z(G^{-1}_\theta(x))|\det \nabla_\theta G^{-1}_\theta(x)|,
\end{equation*}
which forms a (normalized) parametrized family. Next, we discuss the differences between GAN and NCE from several aspects.

\textbf{Method}\quad Based on the explicit density, one can parametrize the discriminator class by \eqref{eq:nce_dis} and obtains NCE through a single logistic regression \eqref{eq:nce_est}. 
To implement GAN, one does not utilize the explicit form of $p_\theta$ and construct a separate parametrized discriminator class $\{D_\psi(x):\psi\in\Psi\}$. Then one adopt AGE (Algorithm~\ref{alg:age}) or $f$-GAN (Algorithm~\ref{alg:fgan}) to learn $\theta$ where in each step we have a fixed value of $\theta$ and apply logistic regression or other classification as in lines 3 in the algorithms. Alternatively, one could utilize the explicit density of $p_\theta$ and adopt the local GAN algorithm (Algorithm~\ref{alg:localgan}). 

\textbf{Asymptotics}\quad Since the parametrized model is normalized, NCE no longer has the parameter $c$. Then according to \cite[Theorem~3]{Gutmann2012NoiseContrastiveEO}, under suitable regularity conditions which are omitted here, the asymptotic variance of NCE is given by
\begin{equation*}
	\Sigma_{\rm NCE}=\bbE_{p_*}\left[S(\theta^*;X)S(\theta^*;X)^\top\frac{\lambda p_n(X)}{p_*(X)+\lambda p_n(X)}\right]^{-1},
\end{equation*}
where $S(\theta;x)=\nabla_\theta\ln p_\theta(x)$ is the Fisher score function as defined in the main text. On the other hand, by Corollary~\ref{cor:localgan}, the asymptotic variance of local GAN estimator is 
\begin{equation*}
	\Sigma_{\rm localGAN}=\left(1+\frac{1}{\lambda}\right)\bbE_{p_*}\left[S(\theta^*;X)S(\theta^*;X)^\top\right]^{-1}.
\end{equation*}
We see that the asymptotic variances of GAN and NCE are different in general. 
\cite{Gutmann2012NoiseContrastiveEO} suggested that a good candidate for the noise distribution $p_n$ is a distribution which is close to the data distribution $p_*$. In the most ideal case where $p_n=p_*$, we have $\Sigma_{\rm NCE}=\Sigma_{\rm localGAN}$ \cite[Corollary~7]{Gutmann2012NoiseContrastiveEO}. This suggests that in practical cases where $p_n$ differs from $p_*$, the asymptotic variance of NCE tend to be inferior to that of the local GAN.  
In addition, as $\lambda\to\infty$, both $\Sigma_{\rm NCE}$ and $\Sigma_{\rm localGAN}$ approaches the Cram\'er-Rao lower bound $\mathcal{I}(\theta^*)^{-1}$. 
\end{example}

\subsection{Generalization of NCE and $f$-GAN}
Lastly, we would like to point out the relationship between $f$-GAN and a generalization of NCE parametrized by nonlinear functions \cite{Pihlaja2010AFO}. Under the same  model setup as above, the generalized NCE \cite{Pihlaja2010AFO} proposes the following formulation
\begin{equation*}
	\min_\theta \left[ \frac{1}{n}\sum_{i=1}^{n}g_1\left(\frac{p_\theta(x_i)}{p_n(x_i)}\right)-\frac{\lambda}{m}\sum_{i=1}^{m}g_2\left(\frac{p_\theta(\tilde x_i)}{p_n(\tilde x_i)}\right)\right],
\end{equation*}
where $g_1(\cdot)$ and $g_2(\cdot)$ are nonlinear functions satisfying $g'_2(r)/g'_1(r)=r$ for all $r>0$. Table~\ref{tab:nce_nonlinear} lists several choices of the nonlinear functions, where in the last column we list the corresponding $f$-GAN whose loss functions coincide in the relationa $g_1(r)=-l_1^f(\ln r)$ and $g_2(r)=l_2^f(\ln r)$ with $l_i^f$ defined in Table~\ref{tab:fgan_loss}. The connection and significant differences between the generalized NCE and $f$-GAN are analogous to those between the original NCE and AGE as we elaborated above. 
It may be of independent interest to look into the connection between the generalized NCEs and $f$-divergence minimization. 

\begin{table}[h]
\centering
\caption{Nonlinear functions of generalized NCE}\label{tab:nce_nonlinear}
\begin{tabular}{cccc}
\toprule
Name in \cite{Pihlaja2010AFO} & $g_1(r)$ & $g_2(r)$ & $f$-GAN \\\midrule
Importance Sampling & $\ln r$ & $r$ & KL \\
Inverse Importance Sampling & $-1/r$ & $-\ln r$ & RevKL \\
NCE & $\ln(\frac{r}{r+1})$ & $\ln(1+r)$ & 2JS  \\
- & $-1/\sqrt{r}$ & $\sqrt{r}$ & $H^2$ \\
\bottomrule
\end{tabular}
\end{table}

\section{Additional experimental results}

\subsection{\revise{Experiments of the two-sample scheme}}\label{app:exp_twosample}

We provide simulation studies on the generator estimation with AGE and $f$-GAN using the two-sample scheme introduced at the end of Section~\ref{sec:misspecify}. As in Section~\ref{sec:exp_g}, we experiment on the two settings Laplace-Gaussian and Gaussian2, consider KL, reverse KL and JS divergences as the objectives respectively, and adopt AGE and $f$-GAN with $n= 100$ or 1000 and $\lambda$ varying from 1 to 1000. Tables \ref{tab:lap_gaus_kl_2}-\ref{tab:gaus2_js_2} report the results of the empirical variances and squared biases of generator estimation in the two settings with the three divergences as the objective, where the metrics are computed from 500 random repetitions. 

As in the one-sample scheme shown in the main text, we observe that AGE leads to smaller variance than $f$-GAN and the advantage is much more significant when $n$ or $\lambda$ is small, which is consistent with our theory. The biases are significantly smaller than the variances, which verifies the consistency. The variances in the two-sample scheme tend to be smaller than those in the one-sample scheme, especially when $n$ is large. This empirically suggests that the covariance terms in the asymptotic variances \eqref{eq:var_g} and \eqref{eq:var_g_fgan} tend to be positive (definite), which is consistent with the calculations in Figure~\ref{fig:cov}. In the KL case, we empirically observe that $f$-GAN is more computationally unstable in the two-sample scheme, producing unreasonable estimations or divergence in many experimental runs. This may be due to the exponential term in the loss function of $f$-GAN-KL. In comparison, AGE does not suffer from this problem.

\begin{table}
\centering
\caption{Results of generator estimation under KL objective (Laplace-Gaussian) using the two-sample scheme algorithms.}
\label{tab:lap_gaus_kl_2}
\subtable[Var]{
\begin{tabular}{cccc}
\toprule
$n$						&	$\lambda$	&	AGE	&	$f$-GAN \\\midrule
\multirow{4}{*}{100}	&	1		&	0.2248	&	873.87	\\
										&	10		&	0.1192	&	189.90	\\
										&	100	&	0.0647	&	33.444	\\
										&	1000	&	0.0438	&	0.0738	\\\midrule
\multirow{4}{*}{1000}	&	1		&	0.0326	&	192.12	\\
										&	10		&	0.0089	&	0.1072	\\
										&	100	&	0.0064	&	0.0108 \\
										&	1000	&	0.0045	&	0.0086	\\\bottomrule
\end{tabular}}
\subtable[Bias$^2\times10^2$]{
\begin{tabular}{cc}
\toprule
AGE	&	$f$-GAN \\\midrule
0.4416	&	32660	\\
0.0782	&	4553.1	\\
0.0639	&	208.01	\\
0.0036	&	0.0827	\\\midrule
0.2803	&	112.57	\\
0.1241	&	0.3164	\\
0.0355	&	0.1857 \\
0.0036	&	0.0195	\\\bottomrule
\end{tabular}}
\end{table}

\begin{table}
\centering
\caption{Results of generator estimation under reverse KL objective (Laplace-Gaussian) using the two-sample scheme algorithms.}
\label{tab:lap_gaus_rkl_2}
\subtable[Var]{
\begin{tabular}{cccc}
\toprule
$n$						&	$\lambda$	&	AGE	&	$f$-GAN \\\midrule
\multirow{4}{*}{100}	&	1		&	0.1447	&	2.7353	\\
										&	10		&	0.0470	&	2.7567	\\
										&	100	&	0.0354	&	1.9430	\\
										&	1000	&	0.0318	&	0.8757	\\\midrule
\multirow{4}{*}{1000}	&	1		&	0.0108	&	0.0117	\\
										&	10		&	0.0046	&	0.0053	\\
										&	100	&	0.0038	&	0.0055 \\
										&	1000	&	0.0038	&	0.0048	\\\bottomrule
\end{tabular}}
\subtable[Bias$^2\times10^3$]{
\begin{tabular}{cc}
\toprule
AGE	&	$f$-GAN \\\midrule
5.3990	&	1027.4	\\
2.7367	&	1072.2	\\
4.4395	&	471.92	\\
3.0498	&	97.238	\\\midrule
0.6931	&	0.4517	\\
0.4680	&	0.3909	\\
0.1066	&	0.1012 \\
0.0918	&	0.0621	\\\bottomrule
\end{tabular}}
\end{table}

\begin{table}
\centering
\caption{Results of generator estimation under JS objective (Laplace-Gaussian) using the two-sample scheme algorithms.}
\label{tab:lap_gaus_js_2}
\subtable[Var]{
\begin{tabular}{cccc}
\toprule
$n$						&	$\lambda$	&	AGE	&	$f$-GAN \\\midrule
\multirow{4}{*}{100}	&	1		&	0.1490	&	0.1490	\\
										&	10		&	0.0379	&	0.0430	\\
										&	100	&	0.0271	&	0.0359	\\
										&	1000	&	0.0258	&	0.0376	\\\midrule
\multirow{4}{*}{1000}	&	1		&	0.0119	&	0.0119	\\
										&	10		&	0.0045	&	0.0052	\\
										&	100	&	0.0039	&	0.0046 \\
										&	1000	&	0.0035	&	0.0042	\\\bottomrule
\end{tabular}}
\subtable[Bias$^2\times10^2$]{
\begin{tabular}{cc}
\toprule
AGE	&	$f$-GAN \\\midrule
1.0320	&	1.0320	\\
0.7498	&	0.7652	\\
0.7731	&	0.6813	\\
0.3176	&	0.3434	\\\midrule
0.3685	&	0.3685	\\
0.0501	&	0.0452	\\
0.0014	&	0.0010 \\
0.0014	&	0.0023	\\\bottomrule
\end{tabular}}
\end{table}

\begin{table}
\centering
\caption{Results of generator estimation under KL objective (Gaussian2) using the two-sample scheme algorithms.}
\label{tab:gaus2_kl_2}
\subtable[Var]{
\begin{tabular}{cccc}
\toprule
$n$						&	$\lambda$	&	AGE	&	$f$-GAN \\\midrule
\multirow{4}{*}{100}	&	1		&	0.0210	&	702.53	\\
										&	10		&	0.0091	&	116.33	\\
										&	100	&	0.0075	&	5.4212	\\
										&	1000	&	0.0061	&	0.0082	\\\midrule
\multirow{4}{*}{1000}	&	1		&	0.0023	&	513.42	\\
										&	10		&	0.0010	&	1.5584	\\
										&	100	&	0.0007	&	0.0015 \\
										&	1000	&	0.0005	&	0.0008	\\\bottomrule
\end{tabular}}
\subtable[Bias$^2\times10^3$]{
\begin{tabular}{cc}
\toprule
AGE	&	$f$-GAN \\\midrule
0.0143	&	15225	\\
0.0676	&	5156.5	\\
0.0085	&	902.88	\\
0.0053	&	15.274	\\\midrule
0.0032	&	15739	\\
0.0029	&	191.64	\\
0.0016	&	0.0286 \\
0.0004	&	0.0009	\\\bottomrule
\end{tabular}}
\end{table}

\begin{table}
\centering
\caption{Results of generator estimation under reverse KL objective (Gaussian2) using the two-sample scheme algorithms.}
\label{tab:gaus2_rkl_2}
\subtable[Var]{
\begin{tabular}{cccc}
\toprule
$n$						&	$\lambda$	&	AGE	&	$f$-GAN \\\midrule
\multirow{4}{*}{100}	&	1		&	0.0180	&	0.1642	\\
										&	10		&	0.0063	&	0.1232	\\
										&	100	&	0.0044	&	0.0851	\\
										&	1000	&	0.0045	&	0.0632	\\\midrule
\multirow{4}{*}{1000}	&	1		&	0.0020	&	0.0042	\\
										&	10		&	0.0008	&	0.0035	\\
										&	100	&	0.0005	&	0.0036 \\
										&	1000	&	0.0004	&	0.0031	\\\bottomrule
\end{tabular}}
\subtable[Bias$^2\times10^4$]{
\begin{tabular}{cc}
\toprule
AGE	&	$f$-GAN \\\midrule
8.2728	&	59.387	\\
3.0087	&	47.343	\\
1.1463	&	17.734	\\
1.3074	&	2.2793	\\\midrule
0.1604	&	0.3344	\\
0.0244	&	0.6478	\\
0.0094	&	0.2341 \\
0.0021	&	0.4596	\\\bottomrule
\end{tabular}}
\end{table}

\begin{table}
\centering
\caption{Results of generator estimation under JS objective (Gaussian2) using the two-sample scheme algorithms.}
\label{tab:gaus2_js_2}
\subtable[Var$\times10^2$]{
\begin{tabular}{cccc}
\toprule
$n$						&	$\lambda$	&	AGE	&	$f$-GAN \\\midrule
\multirow{4}{*}{100}	&	1		&	1.9094	&	1.9094	\\
										&	10		&	0.7129	&	0.7689	\\
										&	100	&	0.6291	&	0.6704	\\
										&	1000	&	0.6204	&	0.6620	\\\midrule
\multirow{4}{*}{1000}	&	1		&	0.1880	&	0.1880	\\
										&	10		&	0.0752	&	0.0776	\\
										&	100	&	0.0634	&	0.0685 \\
										&	1000	&	0.0602	&	0.0629	\\\bottomrule
\end{tabular}}
\subtable[Bias$^2\times10^5$]{
\begin{tabular}{cc}
\toprule
AGE	&	$f$-GAN \\\midrule
3.9081	&	3.9081	\\
2.6643	&	4.8058	\\
0.1813	&	0.5517	\\
0.5296	&	0.7845	\\\midrule
1.0240	&	1.0240	\\
0.0689	&	0.1025	\\
0.0310	&	0.0511 \\
0.0005	&	0.0032	\\\bottomrule
\end{tabular}}
\end{table}

\subsection{\revise{Experiments of local GAN}}\label{app:exp_localgan}

In this section, we present the experimental results of local GAN to support the theory in Section~\ref{sec:opt_gan}. 

\subsubsection{Linear generator class}
We first consider the problem of Gaussian mean estimation with the true distribution being a one-dimension Gaussian distribution $p_*(x)=\cN(1,1)$. The goal is to learn the mean of $p_*$ through a linear generator class $G_\theta(Z)=\theta+ Z$ where $\theta\in[0.1,10^3]$ and $Z\sim\cN(0,1)$. Then the optimal discriminator is $D^*_\theta(x)=(\theta^2-1)/2+(1-\theta)x$,
which motivates the construction of the linear discriminator class $$\cD=\{D_\psi(x):\psi_0+\psi_1x,\psi=(\psi_0,\psi_1)^\top\in\Psi\}$$ with $\Psi$ being a compact subset of $\bbR^2$ containing the optimal discriminator class $\{\psi^*\in\bbR^2:D_{\psi^*}=D^*_\theta,\theta\in[0.1,10^3]\}$. In addition, to construct the local GAN discriminator, we note that the score function in this case is given by $S(\theta;x)=x-\theta$. Suppose we have an initial root-$n$ consistent estimator $\hat\theta_0$. According to \eqref{eq:localgan_dis}, we construct the discriminator class of local GAN as follows 
$$\cD_l=\{D_\psi(x):\psi(x-\hat\theta_0)\}.$$
We compare four methods: MLE, AGE, local GAN with $\hat\theta_0$ being the AGE estimator (named localGAN1), and local GAN with $\hat\theta_0=0.5$. 
We report three metrics: the empirical variance Var $=\hat{\mathrm{Var}}(\hat\theta)$, squared bias Bias$^2=(\hat\bbE\hat\theta-\theta^*)^2$, and expected negative log-likelihood $=\hat\bbE[\bbE_{p_*}\ln p_{\hat\theta}(X)]$, all of which are obtained based on 500 random repetitions. 

Table~\ref{tab:gaus2_local} shows the results for different sample size $n$ and ratio $\lambda$. The results of MLE are the same for different $\lambda$ because it does not rely on the generated sample with size $\lambda n$. As we have more and more generated samples ($\lambda$ grows), the variances, biases and negative log-likelihoods of AGE and local GAN decrease in general. Notably, when the ratio $\lambda$ is sufficiently large, that is, when we have sufficiently many generated samples, all four methods result in similar variances and log-likelihoods, which supports the asymptotic efficiency of GAN and local GAN in these cases. 
In addition, we find that local GAN with a fixed initial estimator (with a constant bias) results in a smaller variance and larger likelihood than local GAN with AGE being the initial estimator, especially when the sample size is relatively small (e.g., 100).

\begin{table}
\centering
\caption{Results of MLE, AGE, and local GAN with a linear generator class for Gaussian mean estimation.}
\label{tab:gaus2_local}
\subtable[Var]{
\begin{tabular}{cccccc}
\toprule
$n$						&	$\lambda$	&	MLE	&	AGE & localGAN1 & localGAN2 \\\midrule
\multirow{4}{*}{100}	&	1		&	0.0108	&	0.0127 & 0.0127 & 0.0109	\\
										&	10		&	0.0108	&	0.0105 & 0.0105 & 0.0099	\\
										&	100	&	0.0108	&	0.0101 & 0.0101 & 0.0094	\\\midrule
\multirow{4}{*}{1000}	&	1		&	0.0010	&	0.0013 & 0.0013 & 0.0011	\\
										&	10		&	0.0010	&	0.0009 & 0.0009 & 0.0010	\\
										&	100	&	0.0010	&	0.0010 & 0.0010 & 0.0009 \\\bottomrule
\end{tabular}}
\subtable[Bias$^2\times10^3$]{
\begin{tabular}{cccccc}
\toprule
$n$						&	$\lambda$	&	MLE	&	AGE & localGAN1 & localGAN2 \\\midrule
\multirow{4}{*}{100}	&	1		&	0.0179	&	1.2085 & 1.2086 & 2.7761	\\
										&	10		&	0.0179	&	0.2482 & 0.2452 & 0.7418	\\
										&	100	&	0.0179	&	0.1750 & 0.1725 & 0.0753	\\\midrule
\multirow{4}{*}{1000}	&	1		&	0.0011	&	0.2707 & 0.2676 & 0.3626	\\
										&	10		&	0.0011	&	0.1036 & 0.1016 & 0.2365	\\
										&	100	&	0.0011	&	0.1264 & 0.1243 & 0.1807 \\\bottomrule
\end{tabular}}
\subtable[Negative log-likelihood]{
\begin{tabular}{cccccc}
\toprule
$n$						&	$\lambda$	&	MLE	&	AGE & localGAN1 & localGAN2 \\\midrule
\multirow{4}{*}{100}	&	1		&	0.9244	&	0.9259 & 0.9259 & 0.9258	\\
										&	10		&	0.9244	&	0.9243 & 0.9243 & 0.9242	\\
										&	100	&	0.9244	&	0.9240 & 0.9240 & 0.9236	\\\midrule
\multirow{4}{*}{1000}	&	1		&	0.9195	&	0.9197 & 0.9197 & 0.9197	\\
										&	10		&	0.9195	&	0.9195 & 0.9195 & 0.9195	\\
										&	100	&	0.9195	&	0.9195 & 0.9195 & 0.9194 \\\bottomrule
\end{tabular}}
\end{table}

\begin{table}
\caption{Negative log-likelihood of MLE, AGE, and local GAN with a neural network generator class.}
\label{tab:local_flow}
\begin{tabular}{cccccc}\toprule
	Oracle & MLE & AGE & localGAN1 & localGAN2 & localGAN5\\\midrule
	2.838 & 2.860 & 2.916 & 2.878 & 2.861 & 2.858\\\bottomrule
\end{tabular}
\end{table}

\subsubsection{Neural network generator class}
Next, we consider a more complex case where the generator class is parametrized by neural networks (NNs). Specifically, we consider a two-dimensional Gaussian as $p_*$ and aim at learning an invertible neural network \cite{Papamakarios2021NormalizingFF} to generate from the distribution. For the specific network architecture, we adopt the coupling layer from \cite{durkan2019neural}. Again we write the NN generator as $G_\theta(z)$ and the generated distribution is given by  
\begin{equation*}
	p_\theta(x)=p_z(G^{-1}_\theta(x))|\det \nabla_\theta G^{-1}_\theta(x)|.
\end{equation*}
Based on the above explicit density and the consequent explicit score function, we can then construct the locally linear discriminator class. We obtain the initial estimator adopted in the local GAN algorithm by randomly sampling from a small neighborhood of the ground truth. 
The training sample size is 512 and we report the negative log-likelihood for a given estimation evaluated on a test sample of size 50k. 

In Table~\ref{tab:local_flow}, we compare the MLE, AGE, localGAN1 (local GAN with $\lambda=1$), localGAN2 (local GAN with $\lambda=2$), and localGAN5 (local GAN with $\lambda=5$) against the oracle $-\bbE_{p_*}[\ln p_*(X)]$.  We observe that local GAN with a large enough $\lambda$ (e.g., $\lambda=5$ in this case) can approach MLE in terms of the log-likelihood, while AGE and local GAN with $\lambda=1$ leads to worse likelihoods. These results are also consistent with the theory in Section~\ref{sec:localgan}.


\subsection{Experiments on real data}\label{app:realdata}
To investigate the performance of our proposed method on real data, we apply it on a real data set of human faces, CelebA \cite{liu2015deep}, which contains 202,599 images. 
As a representative, due its relationship with MLE, we consider the KL divergence as the objective and compare the performance of AGE and $f$-GAN. 
We use the Fr\'echet Inception Distance (FID) \cite{fid} as a quantitative evaluation metric which is commonly used in image generation literature. FID is defined as the Wasserstein-2 distance between the Gaussian approximations of the vision-relevant features obtained from an inception model with real and generated data as the input. 

The FIDs of AGE-KL and $f$-GAN-KL are 36.9411 and 46.3798 (the smaller the better) respectively, which indicates the advantages of AGE in generating real images with better quality. Figure \ref{fig:gen} presents some generated samples from AGE-KL and $f$-GAN-KL, where we see most samples from AGE are of high fidelity and look like the real data, while $f$-GAN samples are poorer and contain more collapsed generations. Both the quantitive and qualitative results demonstrate that AGE can better learn the true data distribution, which is consistent with our theory and simulation results. 

\begin{figure}
\centering
\subfigure[Real]{
\includegraphics[width=.106\linewidth]{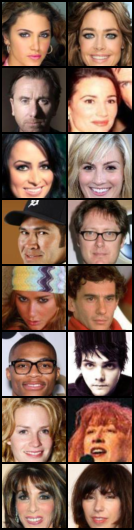}}
\subfigure[Samples from AGE-KL]{
\includegraphics[width=.42\linewidth]{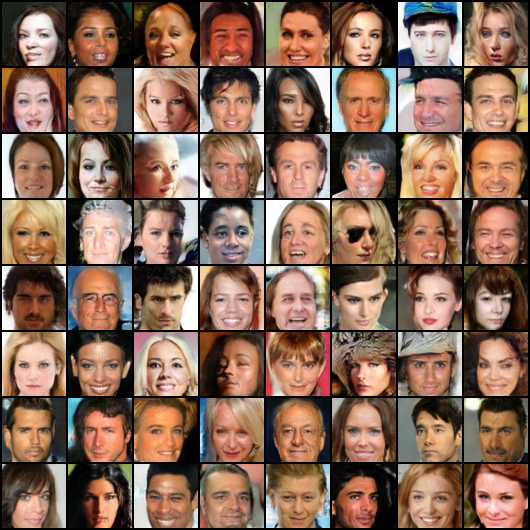}}
\subfigure[Samples from $f$-GAN-KL]{
\includegraphics[width=.42\linewidth]{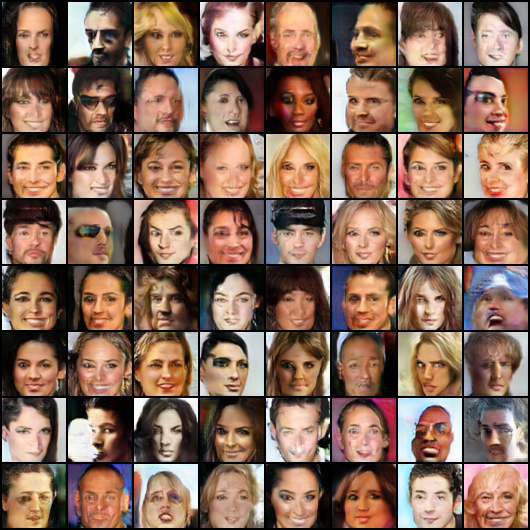}}
\caption{Real and generated data on CelebA.}\label{fig:gen}
\end{figure}

\section{Implementation details}\label{app:exp_detail}

\subsection{Implementation details of simulations in Section~\ref{sec:exp_g} and Appendix~\ref{app:exp_twosample}}
In both AGE and $f$-GAN algorithms, the meta-parameters include the initial parameter $\theta_0$, the learning rate $\eta$ and the time step $T$. We use $T=100$ for all simulations and tune the learning rate manually. For Laplace-Gaussian, we set the initial value $\theta_0=0.1$ and $\eta=0.5$; for Gaussian2, we set $\theta_0=0.5$ and $\eta=0.5$. We empirically find that AGE is robust to a wider range of meta-parameters, while $f$-GAN is much more sensitive and even collapses during some runs, leading to some extremely poor performance. 

\subsection{Implementation details of experiments in Appendix~\ref{app:exp_localgan}}
We run all experiments for $T=100$ steps. For the first experiment, we set the initial value $\theta_0=0.5$ (for AGE) and  $\eta=1$ (for AGE and local GAN). For the second experiment, we use the Adam optimizer with $\beta_1=0.9$, $\beta_2=0.999$, and a learning rate of $1\times10^{-3}$ for optimizing the discriminator and the generator in AGE and local GAN and for obtaining MLE. The discriminator is updated 20 times at each step $t$ in the AGE and local GAN algorithms.

\subsection{Implementation details of real data experiments in Appendix~\ref{app:realdata}}
Our experiments are performed based on the machine learning framework \texttt{PyTorch} \cite{pytorch}. Both AGE and $f$-GAN algorithms are implemented with the following settings.
We pre-process the images by taking center crops of $128\times128$ resizing to the $64\times64$ resolution. We adopt the BigGAN/SAGAN~\cite{biggan,zhang2019self} architectures for the discriminator and generator which utilize convolutional neural networks with self-attention layers and use spectral normalization~\cite{miyato2018spectral} in both the discriminator and generator. Details for the network architectures are given in Figure \ref{fig:arch_sagan} and Table \ref{tab:sagan}. We use $\lambda=10$ for all experiments. For optimizing the discriminator and the generator, we use the Adam optimizer with $\beta_1=0$, $\beta_2=0.999$, and a learning rate of $1\times10^{-4}$ for the discriminator and $5\times10^{-5}$ for the generator and a mini-batch size of 200. At each step $t$ in the algorithm, the discriminator is updated using one mini-batch step. Models were trained for around 50 epochs on NVIDIA RTX 2080 Ti.

\begin{figure}[h]
\centering
\subfigure[]{
\includegraphics[width=0.2\linewidth]{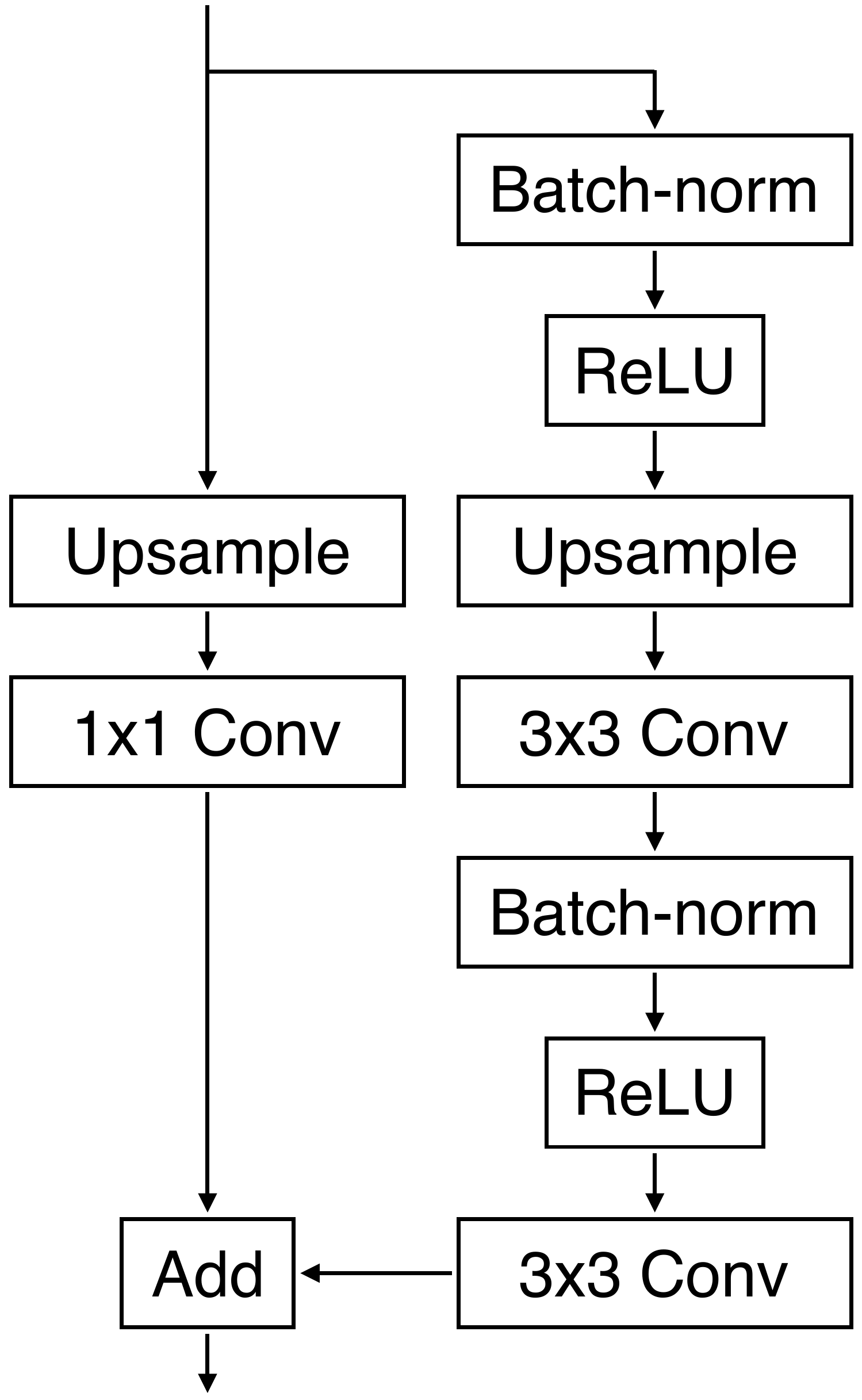}}\hspace{2cm}
\subfigure[]{
\includegraphics[width=0.2\linewidth]{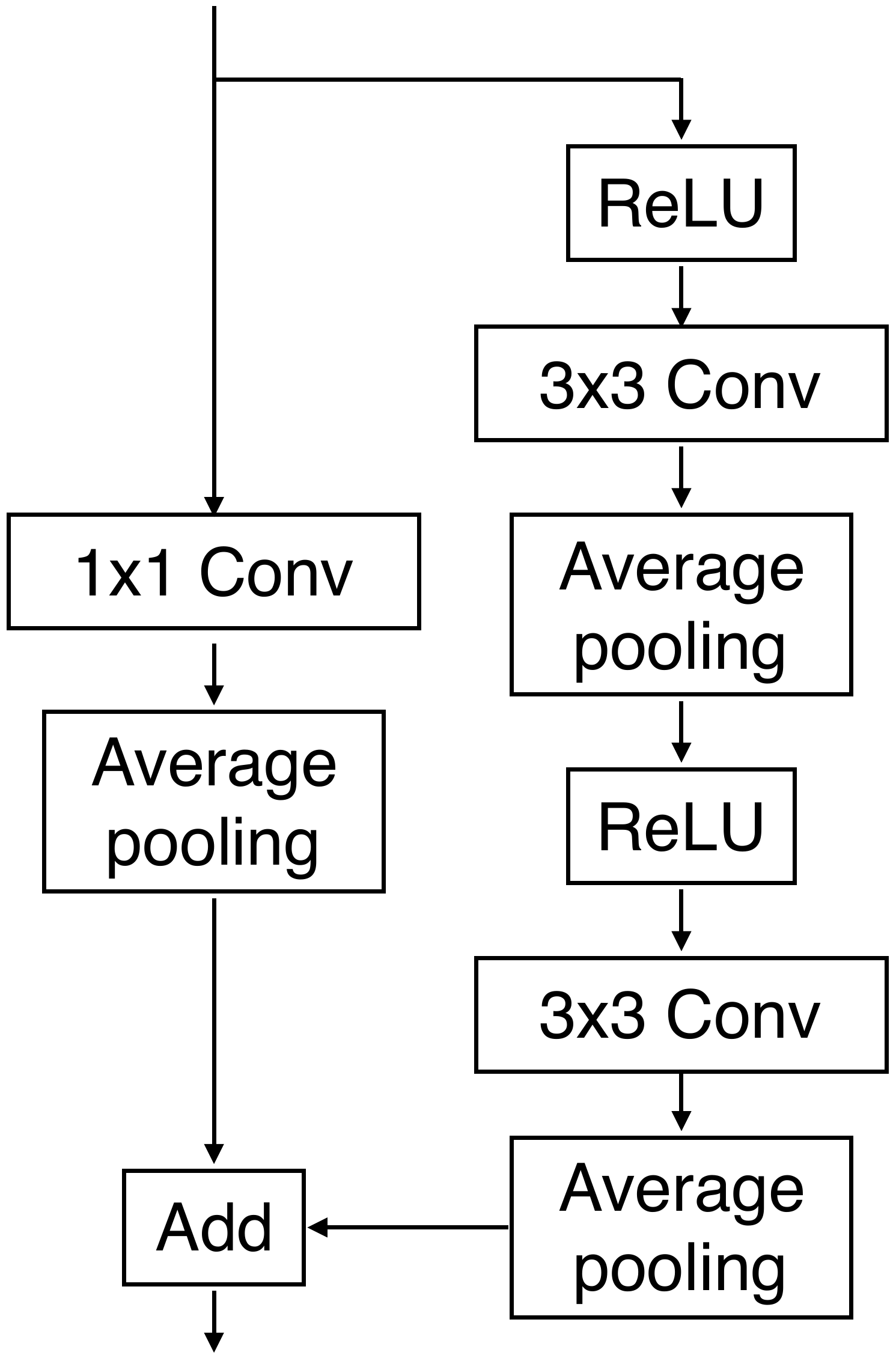}}
\caption{(a) A residual block (ResBlock up) in the BigGAN generator where we use nearest neighbor interpolation for upsampling; (b) A residual block (ResBlock down) in the BigGAN discriminator.}
\label{fig:arch_sagan}
\end{figure}

\begin{table}[h]
\centering
\caption{BigGAN architecture with $k=100$ and $ch=32$.}
\subtable[Generator]{
\begin{tabular}{c}
\toprule
Input: $Z\in\mathbb{R}^k\sim\mathcal{N}(0,\id_k)$\\\midrule
Linear $\to4\times4\times16ch$\\\midrule
ResBlock up $16ch\to16ch$\\\midrule
ResBlock up $16ch\to8ch$\\\midrule
ResBlock up $8ch\to4ch$\\\midrule
Non-Local Block $(64\times64)$ \\\midrule
ResBlock up $4ch\to2ch$\\\midrule
BN, ReLU, $3\times3$ Conv $2ch\to3$\\\midrule
Tanh\\\bottomrule
\end{tabular}
}
\hspace{1cm}
\subtable[Discriminator]{
\begin{tabular}{c}
\toprule
Input: RGB image $X\in\mathbb{R}^{64\times64\times3}$\\\midrule
ResBlock down $ch\to2ch$\\\midrule
Non-Local Block $(64\times64)$ \\\midrule
ResBlock down $2ch\to4ch$\\\midrule
ResBlock down $4ch\to8ch$\\\midrule
ResBlock down $8ch\to16ch$\\\midrule
ResBlock $16ch\to16ch$\\\midrule
ReLU, Global average pooling\\\midrule
Linear $\to1$\\\bottomrule
\end{tabular}
}
\label{tab:sagan}
\end{table}

\end{appendices}

\end{document}